\documentclass{amsart}

\usepackage{amssymb,amsmath,latexsym,amsfonts,amscd,amsthm,bm}

\usepackage{amsmath}
\usepackage{epsfig,amsthm}
\usepackage{latexsym}
\usepackage{amsfonts}
\usepackage{amssymb}
\usepackage{amscd}
\usepackage{mathrsfs}
\usepackage[all,cmtip]{xy}
\usepackage{enumerate}
\usepackage{color}

\usepackage[colorlinks]{hyperref}
\hypersetup{pdfstartview={FitH},         linkcolor=blue,  citecolor=green }

\newtheorem{thm}{Theorem}[section]
\newtheorem{cor}[thm]{Corollary}
\newtheorem{lem}[thm]{Lemma}
\newtheorem{prop}[thm]{Proposition}

\theoremstyle{definition}
\newtheorem{defn}[thm]{Definition}
\newtheorem{rmk}[thm]{Remark}
\numberwithin{equation}{section}

\newcommand{\Fo}{{F_{\bullet}}}
\newcommand{\qo}{{q_{\bullet}}}

\newcommand{\Eo}{{E_{\bullet}}}

\newcommand{\po}{{\mathfrak{p}_{\bullet}}}

\newcommand{\psio}{{\psi_{\bullet}}}

\newcommand{\tr}{{\mathrm{tr}}}

\newcommand{\tV}{{\tilde{\mathfrak{V}}}}

\newcommand{\sgn}{{\mathrm{sgn}}}
\newcommand{\Ind}{{\mathrm{Ind}}}
\newcommand{\cInd}{{\mathrm{cInd}}}
\newcommand{\Res}{{\mathrm{Res}}}

\newcommand{\tame}{{\mathrm{tm}}}
\newcommand{\wild}{{\mathrm{wd}}}
\newcommand{\diag}{{\mathrm{diag}}}
\newcommand{\Asai}{{\mathrm{Asai}}}

\begin{document}

\title[Endoscopic classification of very cuspidal ...]{Endoscopic classification of very cuspidal representations of quasi-split unitary groups}


\author{Kam Fai Tam}
\address{Dept. of Math. \& Stat.,
McMaster University,
1280 Main Street West,
Hamilton, Ontario,
Canada L8S 4K1}
\email{geotam@math.mcmaster.ca}





\begin{abstract}
Let $\mathbf{G}$ be an unramified quasi-split unitary group over a p-adic field of odd residual characteristic. The goal of this paper is to describe the supercuspidal representations within certain L-packets of $\mathbf{G}$, which are classified by Arthur and Mok using the theory of endoscopy. The description is given in terms of the cuspidal types constructed by Bushnell-Kutzko and Stevens. As a starting example, we require the parameters of our packets to satisfy certain regularity conditions, such that these packets consist of very cuspidal representations in the sense of Adler and Reeder. To achieve our goal, we first interpret the question as to study the reducibilities of some parabolically induced representations, using a theory of M{\oe}glin and Shahidi; we then apply a relation, given by Blondel, between these reducibilities and the structures of some Hecke algebras, where the latter can be computed using a Theorem of Lusztig. We can interpret our final result as explicitly describing the local Langlands correspondence for $\mathbf{G}$.
\end{abstract}


 \maketitle


\tableofcontents


\newpage\section{Introduction}


As mentioned in the abstract, we describe in this paper the supercuspidal representations within certain endoscopic L-packets of a quasi-split unitary group over a p-adic field, and this description is given by inducing from cuspidal types. We therefore obtain a relation between the theory of endoscopy and the theory of cuspidal types, both of which are of fundamental interest.

We first recall the basics we need from the two theories. We then state our main Theorem \ref{intro-thm}, which summarizes the two detailed theorems (Theorems \ref{theorem-inertial-class} and \ref{theorem-stable-packet}) of this paper. After that, we detail briefly the methodology for proving the theorem. Finally, we provide several remarks related to the literature.

\subsection{Endoscopic classification}\label{intro-endo-class}

Let $\Fo$ be a p-adic field whose residual characteristic is odd, and let $F/\Fo$ be a field extension, either trivial or quadratic. Let $\mathbf{G}$ be a quasi-split connected classical group over $\Fo$, which is either a special orthogonal group or a symplectic group when $F=\Fo$, or a unitary group when $F/\Fo$ is quadratic. The tempered representations of $G=\mathbf{G}(\Fo)$ and their stable packets, or simply packets, have been classified by \cite[Theorem 1.5.1]{Arthur-new-book} and \cite[Theorem 2.5.1]{Mok-unitary} using the theory of endoscopy, among many important results.

In this paper, we only focus on unitary groups. We hence denote ${G}=\mathrm{U}_n(F/\Fo)$. Under the endoscopic classification, the discrete packets (those consist of discrete series representations) can be parametrized by $n$-dimensional complex representations of the Weil-Deligne group
$\mathcal{WD}_{F}:=\mathcal{W}_F\times \mathrm{SU}_2$ of $F$, without multiplicity and whose components are  conjugate-self-dual. We call these representations \emph{discrete parameters}.

The classification of discrete packets can be described more explicitly using \cite{Moeg-base-change}, \cite{Moeg-endosc-L-param}. Suppose that $\tilde{\phi}$ is a discrete parameter and $\Pi_{\tilde{\phi}}$ is the corresponding packet, then the \emph{extended cuspidal support} of each representation in $\Pi_{\tilde{\phi}}$ coincides with the cuspidal support of the representation $\tilde{\pi}_{\tilde{\phi}}$ corresponding to ${\tilde{\phi}}$ under the local Langlands correspondence for $\mathrm{GL}_n$ (\cite{HT}, \cite{Hen-simple}, \cite{Scholze-LLC}). 

For the precise notion of extended cuspidal support, we refer to \cite{Moeg-base-change} to avoid the details, but we provide an example which will be adopted throughout our paper. We assume that $\tilde{\phi}$ is trivial on $\mathrm{SU}_2$. By the definition in \cite[5.3]{Moeg-base-change}, that a supercuspidal representation $\tilde{\pi}_\circ$ of $\mathrm{GL}_{n_\circ}(F)$ is in the extended cuspidal support of a representation $\pi$ of $G$ is equivalent to that the parabolically induced representation
\begin{equation}\label{parabolic-induced-classical-gp}
\tilde{\pi}_\circ|\det|^s\rtimes \pi:=\Ind_P^{{G}_W}(\tilde{\pi}_\circ|\det|^s\boxtimes \pi)
\end{equation}
is reducible at $s=1$. Here ${G}_W$ is a larger classical group of the same type as ${G}$ and contains a parabolic subgroup $P$ whose Levi component $M$ is isomorphic to $\mathrm{GL}_{n_\circ}(F)\times {G}$. Hence $\pi\in \Pi_{\tilde{\phi}}$ if and only if the Langlands parameter of $\tilde{\pi}_\circ$ is one of the components of $\tilde{\phi}$. As a remark, it is also known that $\Pi_{\tilde{\phi}}$ consists of generic supercuspidal representations.

The reducibility of (\ref{parabolic-induced-classical-gp}) for quasi-split unitary groups was studied in \cite{Moeg-base-change}, some of whose ideas was originated from \cite{silberger-special}, \cite{shahidi-complementary}. We refer to \cite{Shahidi-Cogdell-2014} for a comprehensive study for quasi-split classical groups. For representations induced from generic supercuspidal representations, their reducibilities are related to the poles of certain intertwining operators, given in terms of Langlands-Shahidi or Asai-Shahidi L-functions \cite{shahidi-complementary}. In \cite{Shahidi-Cogdell-2014}, we call $\tilde{\pi}_{\tilde{\phi}}$ the \emph{local transfer} of (the representations in) the packet ${\Pi}_{\tilde{\phi}}$. For unitary groups, we follow the terminology in \cite{Moeg-base-change} and call $\tilde{\pi}_{\tilde{\phi}}$ the \emph{base change} of ${\Pi}_{\tilde{\phi}}$.


The goal of this paper is to describe the representations in $\Pi_{\tilde{\phi}}$, where this description will be made more precise in the next section. As a starting example, we only focus on the following parameters in this paper. We require that
\begin{enumerate}[(i)]
  \item $F/\Fo$ is unramified,
  \item $\mathbf{G}$ is a quasi-split unitary group $\mathrm{U}_n=\mathrm{U}_{n,F/\Fo}$, \label{intro-very-cusp-unitary}
  \item $\tilde{\phi}$ is trivial on the $\mathrm{SU}_2$ factor of $\mathcal{WD}_F$,
  \item  each component $\tilde{\phi}_i$ of $\tilde{\phi}$, with $i$ belongs to an index set $I$, is of the form $\Ind_{E_i/F}\tilde{\xi}_i$, where $E_i/F$ is an unramified extension (necessarily of odd degree by conjugate-self-duality), and $\tilde{\xi}_i$ is a character of $E^\times_i$, viewed as a character of $\mathcal{W}_{E_i}$ by class field theory, \label{intro-very-cusp-char}
  \item all characters $\tilde{\xi}_i$ have the same level, say $d$, and
  \item all restrictions $\tilde{\xi}_i|_{U_{E_i}^d}$, where $i$ ranges over $I$, are not Galois-conjugate to each other.
\end{enumerate}

When $d=0$, these parameters are examples of the depth-zero parameters in \cite{DR} for general unramified reductive groups. In this case when $\mathbf{G}$ is odd special orthogonal or symplectic, the extended cuspidal supports of the corresponding representations can be determined by the results in \cite{Savin-level-0}. When $d$ is positive, these parameters are examples of those in \cite{reeder-pos-depth}, where the corresponding representations are called very cuspidal representations. We follow this terminology and call $\tilde{\phi}$ a \emph{very cuspidal parameter}.

\subsection{Supercuspidal representations}\label{intro-construct-supercusp}

In this section, we describe the supercuspidal representations by constructing their underlying cuspidal types using the above parameters. Then we can state our main result, Theorem \ref{intro-thm}.

We first summarize how to construct supercuspidal representations of a general linear group. Suppose that $\tilde{\phi}_\circ$ is a parameter of the form $\tilde{\phi}_\circ=\Ind_{\mathcal{W}_{E_\circ}}^{\mathcal{W}_F}\tilde{\xi}_\circ$ satisfying the conditions in (\ref{intro-very-cusp-char}) above. In section \ref{section Maximal types for GL}, we construct a supercuspidal representation $\tilde{\pi}_{\tilde{\xi}_\circ}$ of $\mathrm{GL}_{n_\circ}(F)$, where $n_\circ=[E_\circ:F]$, compactly induced from an extension of a maximal simple type, a notion defined \cite{BK}. This extended maximal simple type is constructed from the character $\tilde{\xi}_\circ$ if we regard $E_\circ^\times$ as an unramified elliptic maximal torus in $\mathrm{GL}_{n_\circ}(F)$. By \cite{Hen-unram}, the representation $\tilde{\pi}_{\tilde{\xi}_\circ}$ is parametrized by $\tilde{\phi}_\circ$ under the local Langlands correspondence for $\mathrm{GL}_{n_\circ}$.

When $G=\mathbf{G}(\Fo)$ is an unramified quasi-split unitary group, we construct supercuspidal representations of $G$ using a similar process. We know that every unramified elliptic maximal torus of $\mathbf{G}=\mathrm{U}_n$ is of the form $\prod_{i\in I}\mathrm{U}_{1,E_i/\Eo_i}$, where $\Eo_i/\Fo$ is an unramified extension of odd degree $n_i$, $E_i/\Eo_i$ is quadratic unramified, and $\sum_{i\in I}n_i=n$. Let $\mathbf{T}$ be a torus of this form. In contrast to the $\mathrm{GL}_n$-case, the $\Fo$-embeddings of $\mathbf{T}$ in $\mathbf{G}$ are in general not conjugate to each other in ${G}$. We use an index $x$ to stand for a ${G}$-conjugacy class of an embedding of ${T}=\mathbf{T}(\Fo)$ in ${G}$, and denote the image in $G$ (up to conjugacy) by $T_x$.

Now suppose that $\{\tilde{\xi}_i\}_{i\in I}$ is a set of skew-characters, which means that each $\tilde{\xi}_i$ is a conjugate-self-dual character of $E_i^\times$ and $\tilde{\xi}_i|_{\Eo_i^\times}$ is trivial. By descending the character $\tilde{\xi}:=\boxtimes_{i\in I}\tilde{\xi}_i$ of $\prod_{i\in I}E_i^\times$ to a character $\xi:=\boxtimes_{i\in I}\xi_i$ of $T_x$ canonically, we construct a maximal semi-simple type (\cite{stevens-supercusp}) from $\tilde{\xi}$, which is a \emph{cuspidal type} as defined in \cite{Stevens-Miya}. 
This cuspidal type is then compactly-induced to a supercuspidal representation $\pi_{x,\tilde{\xi}}$ of $G$.

Our construction of the cuspidal type of $\pi_{x,\tilde{\xi}}$ is a combination of \cite{stevens-supercusp} and \cite{adler-cusp}: we construct from $\tilde{\xi}$ a character, called a semi-simple character in \cite{stevens-ss-char}, of a compact subgroup of $G$ using \cite{adler-cusp} (see also \cite[Section 3.2]{reeder-pos-depth}), and extend this character to a cuspidal type using the method of beta-extension in \cite[Section 4]{stevens-supercusp}.

We emphasize that it is important to consider different conjugacy classes of $\Fo$-embeddings $x$ of $T$ in $G$. From many examples (e.g. the depth-0 case \cite{DR}) it is known that different embeddings give rise to non-isomorphic supercuspidal representations $\pi_{x,\tilde{\xi}}$, for a fixed $\tilde{\xi}$. 

With the descriptions of the above supercuspidal representations, we now state our main result, summarizing Theorems \ref{theorem-inertial-class} and \ref{theorem-stable-packet} in a simpler way.
\begin{thm}\label{intro-thm}
Let $n$ be an odd integer, $\tilde{\phi}$ be a very cuspidal parameter, and $\tilde{\xi}$ be the associated character. There exists a unique quadratic (hence tamely ramified) character $\tilde{\nu}_{x,\tilde{\xi}}^{y}$ of $\prod_{i\in I}E_i^\times$ such that
  $$\Pi_{\tilde{\phi}}=\{\pi_{x,\tilde{\xi}\cdot\tilde{\nu}_{x,\tilde{\xi}}^{y}}\}_{x}.$$
  In other words, the base change of the above supercuspidal representations is $\tilde{\pi}_{\tilde{\xi}}$, the representation of $\mathrm{GL}_n(F)$ whose cuspidal support is $\{\tilde{\pi}_{\tilde{\xi}_i}\}_{i\in I}$.
\end{thm}
The Theorem also holds for even $n$, except that we need to modify $\tilde{\xi}$ by multiplying to it a character related to an endoscopic datum. This will be explained in Section \ref{section endos classification}.

We call $\tilde{\nu}_{x,\tilde{\xi}}^{y}$ an \emph{amending character}. In the next section, we will explain the notation $\tilde{\nu}_{x,\tilde{\xi}}^{y}$ for this character and provide a brief idea of computing it, together with the methodology for proving Theorem \ref{intro-thm}.

\subsection{Methodology}

Using M{\oe}glin's theory discussed in Section \ref{intro-endo-class}, our work is to determine which data $({\tilde{\xi}_\circ},x,\tilde{\xi})$ give rise to the reducibility of \begin{equation}\label{intro-testing-reducibility}
    \tilde{\pi}_{\tilde{\xi}_\circ}|\det|^s\rtimes {\pi}_{x,\tilde{\xi}}
  \end{equation}
at $s=1$. We use a result of \cite{Blondel-Weil}: the real-parts of $s$ at which (\ref{intro-testing-reducibility}) is reducible are determined by the structure of a Hecke algebra of a (non-cuspidal) type in the larger unitary group ${G}_W$, such that this type is a \emph{covering type} of the cuspidal type of the supercuspidal representation $\tilde{\pi}_{\tilde{\xi}_\circ}\boxtimes \pi_{x,\tilde{\xi}}$ of the Levi subgroup $M\cong \mathrm{GL}_{n_\circ}(F)\times {G}$. The notion of covering types is defined in \cite{BK-cover} for general reductive group, which implies the following statement in terms of category theory: the category of right modules over this Hecke algebra determines a Bernstein component, the full-subcategory of smooth representations of ${G}_W$ whose irreducible subquotients have supports in the inertial class of $\tilde{\pi}_{\tilde{\xi}_\circ}\boxtimes\pi_{x,\tilde{\xi}}$. (For a similar statement specific for classical groups, see \cite{Heiermann-Hecke-Bernstein}.)

We then combine the machineries of \cite{BK-cover}, \cite{Blondel-Weil}, \cite{stevens-supercusp} in Sections \ref{section repres} and \ref{section Hecke algebra}, which is too technical to explain in this Introduction. To put it simply, they suggest that $E_\circ$ is one of the $E_i$, for $i\in I$, and
the two characters $\tilde{\xi}_\circ$ and $\tilde{\xi}_{i}$ differ from one another by a quadratic character $\tilde{\nu}_{i,x,\tilde{\xi}}^{y}$. We will compute in Section \ref{section amending} this quadratic character using the aforementioned Hecke algebra, which is reduced to consider certain Hecke algebras for finite unitary groups, whose structures are determined by a Theorem of Lusztig \cite[Theorem 8.6.2]{Lusztig-book} (see also \cite{Morris-tame-ram-intertwining-alg} for the depth zero situation).

The notation $y$ of the amending character stands for a "hyperspecial" vertex in an appropriate Bruhat-Tits building, so that we obtain a finite unitary group of the largest possible rank in order to apply Lusztig's theorem to compute the parameters of the Hecke algebras.

Now the inertial class of $\tilde{\pi}_{\tilde{\xi}_\circ}$ is determined. Within this class there are two conjugate-self-dual representations, the possible candidates for the reducibility of (\ref{intro-testing-reducibility}) at $s=1$, different from each other by the quadratic unramified character of $\mathrm{GL}_{n_\circ}(F)$. An application on the equality between Asai-Shahidi L-function \cite{shahidi-complementary} and Asai L-function, proved by Henniart \cite{Hen-ext-sym}, determines the one that gives rise to the reducibility.

We hence obtain a collection of conjugate-self-dual supercuspidal representations which belong to the extended cuspidal support of $\Pi_{\tilde{\phi}}$. By applying M{\oe}glin's inequality \cite[4. Proposition]{Moeg-base-change} on estimating the ranks of the underlying unitary groups, we claim that our collection exhausts the entire extended cuspidal support.

Finally, we set
\begin{equation*}
  \tilde{\nu}_{x,\tilde{\xi}}^{y}=\boxtimes_{i\in I}\tilde{\nu}_{i,x,\tilde{\xi}}^{y}
\end{equation*}
 to obtain our amending character of $\prod_{i\in I}E_i^\times$ in Theorem \ref{intro-thm}. We will compute this character in Corollary \ref{character-quad-or-not}. Actually, this character is the sign character of the $(\prod_{i\in I}E_i^\times)$-action on a finite quotient of two Moy-Prasad filtration subgroups (\cite{Moy-Prasad}) of $G_W$. These compact subgroups will be described using the language of lattices and hereditary orders from \cite{stevens-supercusp}, as we will use his results heavily in our methodology.

We remark that a similar kind of amending characters has already appeared in \cite[Section 3.3]{Blondel-Weil}, where she studied the reducibility of (\ref{intro-testing-reducibility}) in the "disjoint" case; in our language, this means that $\tilde{\xi}_\circ$ is not Galois conjugate to any $\tilde{\xi}_i$ for $i\in I$. This is reduced to study the reducibility of $\Ind_{}^{}\tilde{\pi}_{\tilde{\xi}_\circ}|\det|^s$, a parabolically induced representation of $\mathrm{U}_{2n_\circ}$ from the Siegel parabolic containing the Levi isomorphic to $\mathrm{GL}_{n_\circ}$. This kind of parabolically induced representations is previously studied in \cite{Goldberg-Siegel-case-U22}, \cite{Goldberg-Siegel-case-unitary}, \cite{MR-unitary} for unitary groups, in \cite{Shahidi-twisted-endos}, \cite{MR-classical} for split classical groups, and in \cite{GKS} for general classical groups using the theory of covering types.

If $z$ is a vertex adjacent to $y$, we can define another amending character $\tilde{\nu}_{x,\tilde{\xi}}^{z}=\boxtimes_{i\in I}\tilde{\nu}_{i,x,\tilde{\xi}}^{z}$ analogously for $z$ in place of $y$. We describe in Proposition \ref{transition-with-amending} a relation between our amending characters $\tilde{\nu}_{i,x,\tilde{\xi}}^{y}$ and $\tilde{\nu}_{i,x,\tilde{\xi}}^{z}$ with a character in \cite[Corollary 6.13]{stevens-supercusp}. In our paper, this character is denoted by $\tilde{\chi}_{y,i}^z$ defined on $\mathfrak{o}^\times_{E_i}$, and is called the transition character between beta-extensions relative to the parahoric subgroups associated to $y$ and $z$. Our Proposition \ref{transition-with-amending} simply says that
$$\tilde{\chi}_{y,i}^z=(\tilde{\nu}_{i,x,\tilde{\xi}}^{y}\tilde{\nu}_{i,x,\tilde{\xi}}^{z})|_{\mathfrak{o}^\times_{E_i}}.$$
A similar relation between these characters is also known in the Siegel case \cite{Blondel-Weil} (see the discussion before and in Proposition 3.17 \emph{loc. cit.}).

To pick out the correct representations for the reducibility in (\ref{intro-testing-reducibility}), we view $\mathbf{G}$ as an elliptic twisted endoscopic group of $\tilde{\mathbf{G}}=\Res_{F/\Fo}\mathrm{GL}_n$. In general, an elliptic twisted endoscopic datum of $\tilde{\mathbf{G}}$ is of the form
$$\mathbf{H}=\mathrm{U}_{n_1,F/\Fo}\times \mathrm{U}_{n_2,F/\Fo},\,\qquad n_1,n_2\in\mathbb{Z}_{\geq0}\text{ and }n_1+n_2=n,$$
attached with a choice among two sign-data that determines the embedding of the L-group ${}^L\mathbf{H}$ into $ {}^L\tilde{\mathbf{G}}$ (except when $n_1=n_2$, in which case the two data are equivalent in an appropriate sense). We can also study the base changes of the packets of $\mathbf{H}(\Fo)$ to the representations of $\tilde{\mathbf{G}}(\Fo)=\mathrm{GL}_n(F)$. In Corollary \ref{theorem-H-stable-packet}, we generalize Theorem \ref{intro-thm}, the Theorem for $\mathbf{H}=\mathbf{G}$, to a statement for arbitrary elliptic twisted endoscopic group $\mathbf{H}$ of $\tilde{\mathbf{G}}$.

\subsection{Remarks related to the literature}

\begin{enumerate}[(i)]

  \item The quasi-split condition for a classical group is applied because, in order to describe the local transfers, we view the group as a twisted endoscopic group of $\mathrm{GL}_n$. This is the view-point in \cite{Moeg-base-change}, \cite{Arthur-new-book}, \cite{Mok-unitary}. However, as pointed out in \cite{Moeg-endosc-L-param}, this condition is only for simplicity. Now the endoscopic classification for inner forms of classical groups is available (the essential idea is in \cite[Chapter 9]{Arthur-new-book} for special orthogonal and symplectic groups, and in \cite{unitary-inner} for unitary groups). The author expects that the same method in this paper applies to describe the local transfers (whenever defined) for inner forms without significant modification.

\item The above constructions of extended maximal types and cuspidal types from characters of elliptic maximal tori are well-known, where the origin can be traced back to \cite{Howe} for general linear groups (and similar constructions in \cite{Gerardin-unram}, \cite{Carayol-cusp}). The method was then generalized to \cite{BK}, where we allow the underlying types to be directly extended from characters of compact subgroups,
    without appealing to characters of elliptic maximal tori, and hence obtain an exhaustive construction. Similar exhaustive methods were then applied to special linear groups \cite{BK-SLn-2} and to classical groups \cite{stevens-supercusp}, \cite{Stevens-Miya} (when the residual characteristic is odd). There are other similar methods \cite{adler-cusp}, \cite{Yu-cusp}, \cite{Kim-exhaust} and recently some novel methods \cite{Gross-Reeder}, \cite{Reeder-Yu} aiming at constructing supercuspidal representations for general connected reductive groups.

  \item Attempts to describing local transfers include
  \begin{itemize}
    \item the depth-0 case, for $\mathbf{G}=\mathrm{U}_2$ and $\mathrm{U}_3$, by \cite{Adler-Lansky-unram} when $F/\Fo$ is unramified, and by \cite{Adler-Lansky-ram} when $F/\Fo$ is ramified;
        \item the positive level case, by \cite{Blasco-u2} for $\mathbf{G}=\mathrm{U}_2$, and by \cite{Blasco-u3} for $\mathbf{G}=\mathrm{U}_3$ when the packet is a singleton.
  \end{itemize}
  In the cases above, they directly computed the representations in a packet by applying the twisted trace formula \cite{Row-u3}. For the use of extended cuspidal support, we have the following.
  \begin{itemize}
    \item For $\mathbf{G}$ being odd special orthogonal or symplectic, \cite{Savin-level-0} describes the local transfers for depth-0 parameters from \cite{DR}, and proves that the local transfers are the same as in \cite{Jiang-Soudry-SO-odd}, \cite{Cogdell-Kim-PS-Shahidi};
        \item \cite{Lust-GSP4} adopts the same method for depth-0 $\mathrm{GSP}_4$ and proves that the local transfers are the same as in \cite{Gan-Takeda-GSP4}.
  \end{itemize}

  It seems to the author that our paper is the first place to apply the method of extended cuspidal support for positive level supercuspidal representations.

      \item Although we only study the endoscopic classification for parameters under the restrictive very cuspidal condition above, the author expects that the methodology should apply to more general discrete parameters, as long as their components are induced from characters.

      Furthermore, the author was informed that Blondel's ongoing work (jointed with Henniart and Stevens) has obtained some results on the relation between covering types and packets of representations whose underlying cuspidal types are not necessarily constructed from characters. When combined with the endoscopic classification, their results should lead to describing explicitly the local transfers of packets for arbitrary classical groups.

\end{enumerate}

\newpage\section{Unitary groups}

\subsection{Notations}

Let $\Fo$ be a p-adic field of odd residual characteristic $p$. Let $F/\Fo$ be the unramified quadratic extension. We denote by $\mathfrak{o}_{F_\bullet}$ and $\mathfrak{p}_{F_\bullet}$ (resp. $\mathfrak{o}_{F}$ and $\mathfrak{p}_{F}$) the ring of integers and its maximal ideal of $\Fo$ (resp. $F$). We choose a uniformizer $\varpi$ such that $\varpi\mathfrak{o}_{F_\bullet}=\mathfrak{p}_{F_\bullet}$. We denote by $\mathbf{k}_{\Fo}$ (resp. $\mathbf{k}_{F}$) the residual field of $\Fo$ (resp. $F$) with cardinality $\qo$ (resp. $q$).

The multiplicative group $\Fo^\times$ decomposes into a product of subgroups $$\left<\varpi\right>\times{\boldsymbol{\mu}}_\Fo\times U_\Fo^1 .$$
These subgroups are respectively the group generated by the uniformizer $\varpi$, the group ${\boldsymbol{\mu}}_\Fo$ of order $\qo-1$ containing the roots of unity of order prime to $p$, and the 1-unit group $U_\Fo^1:=1+\mathfrak{p}_\Fo$. We will identify ${\boldsymbol{\mu}}_\Fo$ with $\mathbf{k}_\Fo^\times$ by the natural projection.

We fix, once and for all, an additive character $\psio$ of $\Fo$ whose kernel is $\po$, so that its composition with the trace $\mathrm{tr}_{F/\Fo}$ is an additive character $\psi_F$ of $F$ whose kernel is $\mathfrak{p}$.

We denote by $\delta_{F/\Fo}$ the character associated to the quadratic extension $F/\Fo$ by local class field theory, which is the unique quadratic character on $\Fo^\times$ trivial on the norm group $N_{F/\Fo}(F^\times)$.

We denote by $\Gamma_{\Fo} $ the absolute Galois group of $\Fo$ and by $c$ a Frobenius element in  $\Gamma_{\Fo} $ such that
$${}^{c}x\equiv x^{\qo}\mod\mathfrak{p}_{F_\bullet}\text{, for all }x\in\mathfrak{o}_{F_\bullet}.$$
We denote the non-trivial element in $\Gamma_{F/\Fo}=\Gamma_{\Fo}/\Gamma_{F}$ also by ${c}$.


Some notations on general group theory. If $G$ acts on a set $X$ and $X'$ is a subset of $X$, we denote by $N_G(X')$ the subgroup of elements in $G$ that leave $X'$ invariant, and by $Z_G(X')$ the subgroup of elements in $G$ that act trivially on $X'$. If $\pi$ and $\pi'$ are representations of two groups $G$ and $G'$ respectively, we denote by $\pi\boxtimes \pi'$ their exterior tensor product, which is a representation of $G\times G'$. If $H$ is a subgroup of $G$ and $\pi$, $\pi'$ are representations of $H$, we denote by $I_G(\pi,\pi')$ the subset of elements in $G$ that intertwines $\pi$ and $\pi'$, and by $I_G(\pi)$ for $\pi=\pi'$.

\subsection{Unitary groups and subgroups}

Let $V$ be a vector space of $F$-dimension $n$, equipped with a non-degenerate Hermitian form $h_V$. The group $G_V$ of isometries is the unitary group $\mathrm{U}_{n}(F/\Fo)$, which is the subgroup of the $\Fo$-points of an algebraic group $\mathbf{G}_V=\mathrm{U}_{n,F/\Fo}$ over $\Fo$, an $\Fo$-form of the split group $\mathrm{GL}_{n,\Fo}$. Explicitly, if we set
\begin{equation}\label{J-matrix}
  J=J_V=\begin{pmatrix}
    &&&1
    \\
    &&-1&
    \\
    &\dots&&
    \\
    (-1)^n&&&
  \end{pmatrix},
\end{equation}
then we let $\Gamma_F$ act on ${G}_V$ trivially and $c\in \Gamma_{\Fo}-\Gamma_F$ act on $\mathrm{G}_V$ by
\begin{equation}\label{sigma-action}
  {}^{\sigma}x=J{}^c({}^{t}x^{-1})J^{-1}
\end{equation}
where ${}^cx$ is the $c$-conjugate on all entries of $x$, and ${}^tx$ is the transpose of a matrix $x$.

Let $\tilde{\mathbf{G}}_V$ be the algebraic group $\Res_{F/\Fo}\mathrm{GL}_{n,F}$ over $\Fo$, where $c\in \Gamma_{\Fo}-\Gamma_{F}$ acts as
$$(x,y)\mapsto ({}^{\sigma}y,{}^{\sigma}x)$$
 The subgroup of the $\Fo$-points is
 $$\tilde{G}_V=\tilde{{G}}_V(\Fo)={G}_V(F)=\mathrm{GL}_{n}(F).$$
  The action (\ref{sigma-action}) then descends to an involution $\sigma$ on $\tilde{G}_V$, and so $(\tilde{G}_V)^\sigma=G_V$.

Note that if $\dim V=1$, then $G_V\cong \ker {N}_{F/\Fo}$, which decomposes into a product of subgroups $${\boldsymbol{\mu}}_{F/\Fo}\times U_{F/\Fo}^1,$$
where ${\boldsymbol{\mu}}_{F/\Fo}$ is the group of order $\qo+1$ containing the roots of unity of order prime to $p$ and whose norm is trivial, and $U_{F/\Fo}^1:=\{x\in U_F^1,\,{}^cx=x^{-1}\}$. Note that we can identify ${\boldsymbol{\mu}}_{F/\Fo}$ with $\mathrm{U}_1(\mathbf{k}_F/\mathbf{k}_{\Fo})$ by the natural projection.

We also denote by $\tilde{A}_V=\mathrm{End}_F(V)$ the algebra of $F$-endomorphisms of $V$, which is the $F$-points of $\mathbf{A}_V=\mathrm{End}_{\bar{F}}(V\otimes_F\bar{F})$ as an $\bar{F}$-variety. We also define an action of $\Gamma_{F/\Fo}$ on $\tilde{A}_V$ by ${}^\sigma X=-J{}^c({}^{t}X) J^{-1}$, and write $A_V=(\tilde{A}_V)^\sigma$.


In this paper, we only consider our unitary group $\mathbf{G}_V$ to be quasi-split over $\Fo$, which means that $\mathbf{G}_V$ contains a Borel subgroup invariant under the action of $\Gamma_{\Fo}$. From the chosen matrix $J$, the Borel subgroup is the group of upper triangular unitary matrices.

Moreover, we only consider the Hermitian space $V$ admitting an orthogonal decomposition $V=\oplus_{i\in I}V_i$, where $I$ is a finite index set. We then have an embedding
\begin{equation}\label{Levi-embedding-F}
  \prod_{i\in I}\tilde{\mathbf{G}}_{V_i}\hookrightarrow \tilde{\mathbf{G}}_V
\end{equation}
 as a Levi subgroup defined over $F$. If such embedding is defined over $F$, then it restricts to an embedding
 $ \prod_{i\in I}\tilde{{G}}_{V_i}\hookrightarrow \tilde{{G}}_V$, and all such embeddings are conjugate by $\tilde{G}_V$. We now require that the involution $\sigma$ of $ \tilde{G}_V$ restricts to the corresponding involution $\sigma_i$ on each $\tilde{G}_{V_i}$, so that the $\Fo$-embedding
\begin{equation}\label{Levi-embedding-Fo}
  \prod_{i\in I}{\mathbf{G}}_{V_i}\hookrightarrow {\mathbf{G}}_V
\end{equation}
restricts to an embedding
 $ \prod_{i\in I}{G}_{V_i}\hookrightarrow {G}_V$. However, different embeddings may not be conjugate to each other under ${G}_V$, and there may not be any $\Fo$-parabolic subgroup containing this Levi subgroup.

We will also consider unitary groups of higher ranks and their subgroups. Fix $i_\circ\in I$ and let $V_{i_\circ,-} $ and $V_{i_\circ,+}$ be two vector spaces isomorphic to $V_{i_\circ}$, then write
$$W=V_{i_\circ,-}\oplus V\oplus V_{i_\circ,+}.$$
By choosing a basis $\mathcal{B}_-=\{\mathbf{e}_{-1},\dots,\mathbf{e}_{-n_{i_\circ}}\}$ for $V_{i_\circ,-}$ and another $\mathcal{B}_+=\{\mathbf{e}_1,\dots,\mathbf{e}_{n_{i_\circ}}\}$ for $ V_{i_\circ,+}$, we equip $W$ with a Hermitian form $h_W$ defined by
\begin{equation*}
h_W|_{V\times V }=h_V\text{ and }V\perp (V_{i_\circ,-}\oplus V_{i_\circ,+})
\end{equation*}
and
\begin{equation}\label{even-Hermitian}
\begin{split}
  &h_{W}(\mathbf{e}_{i},\mathbf{e}_{ j})=h_{W}(\mathbf{e}_{- i},\mathbf{e}_{-j})=0, \text{ and }
  \\
  &h_{W}(\mathbf{e}_{-i},\mathbf{e}_{j})=\varpi\delta_{ij}\text{, for }i,j=1,\dots,n_{i_\circ}.
\end{split}
\end{equation}
For convenience, we still denote the involution on $\tilde{G}_W$ by $\sigma$ and let $G_W=(\tilde{G}_W)^\sigma$ be the corresponding unitary group. We denote by $P$ the parabolic subgroup of $G_W$ stabilizing the flag
$$0\subset V_{i_\circ,-} \subset V_{i_\circ,-}\oplus V\subset W,$$
and by $M$ the Levi subgroup of $P$. If we write $\tilde{G}_{V_{i_\circ}}=\mathrm{GL}_F(V_{i_\circ,-})$ (or $\mathrm{GL}_F(V_{i_\circ,+})$), then $M\cong \tilde{G}_{V_{i_\circ}}\times G_V$ given by an embedding
$$\mathcal{I}_+^M:\tilde{G}_{V_{i_\circ}}\times G_V\rightarrow G_W,\,(g,h)\mapsto({}^{\sigma_{i_\circ}} g,h,g).$$
We also denote by $U$ the unipotent subgroup of $P$, and by $U^-$ the opposite of $U$.

\subsection{Embeddings of Tori}

Let $I$ be a finite set, and for $i\in I$, let $\Eo_i$ be the unramified extension of $\Fo$ of degree $n_i$, such that $E_i:=\Eo_i\otimes_\Fo F$ is the unramified extension of $F$ of the same degree.  Hence each $n_i$ is necessarily odd. We denote by $c_i$ the non-trivial automorphism in $\Gamma_{E_i/\Eo_i}$.

Suppose that the Hermitian space $V$ has an orthogonal decomposition $V=\oplus_{i\in I}V_i$ such that $\dim V_i=n_i$. By identifying $V_i$ with $E_i$ as an $F$-vector space, we have an $F$-embedding of an elliptic torus
 \begin{equation*}
  \tilde{\mathbf{T}}_i=\Res_{E_i/\Fo}\mathrm{GL}_{1,E_i}\hookrightarrow \tilde{\mathbf{G}}_{V_i},
\end{equation*}
and hence an $F$-embedding of the product torus
\begin{equation*}
  \tilde{\mathbf{T}}=\prod_{i\in I}\tilde{\mathbf{T}}_i\hookrightarrow  \tilde{\mathbf{G}}_{V}
\end{equation*}
 using the embedding (\ref{Levi-embedding-F}) of Levi subgroup.

 We assume that $\tilde{T}= \tilde{\mathbf{T}}(\Fo)$ is embedded into $ \tilde{{G}}_{V}$ such that ${c}$ acts on each factor $\tilde{{T}}_i:=\tilde{\mathbf{T}}_i(\Fo)$ as ${c}_i$, the non-trivial automorphism in $\Gamma_{E_i/\Eo_i}$. Using the embedding (\ref{Levi-embedding-Fo}), if $\mathbf{T}_i=\Res_{\Eo_i/\Fo}\mathrm{U}_{1,E_i/\Eo_i}$, then there is an $\Fo$-embedding
\begin{equation*}
  \mathbf{T}=\prod_{i\in I}\mathbf{T}_i\hookrightarrow  {\mathbf{G}}_{V}
\end{equation*}
as an elliptic unramified maximal torus, and every such torus is of this form. On the subgroup of $\Fo$-points, we have ${T}=\tilde{{T}}^\sigma=\prod_{i\in I}\tilde{{T}}_i^{\sigma_i}$, where ${\sigma_i}$ is the corresponding involution on $\tilde{G}_{V_i}$ such that ${\sigma_i}|_{\tilde{{T}}_i}(t_i)={}^{{{c}}_i}t_i^{-1}$.

 An embedding $\mathcal{I}:\mathbf{T}\hookrightarrow \mathbf{G}_V$ considered above is defined over $\Fo$, which means that
$$\mathcal{I}\circ \gamma_\mathbf{T}=\gamma_{\mathbf{G}_V}\circ \mathcal{I}\text{, for all }\gamma\in \Gamma_\Fo,$$
where $\gamma_\mathbf{T}$ (resp. $\gamma_{\mathbf{G}_V}$) is the automorphism of $\mathbf{T}$ (resp. ${\mathbf{G}_V}$) defined by $\gamma$. Fix one $\Fo$-embedding $\mathcal{I}_0$ with image $\mathbf{T}_0\subset \mathbf{G}_V$. Suppose that $\mathcal{I}:\mathbf{T}\rightarrow \mathbf{G}_V$ is another $\Fo$-embedding, which is necessarily conjugate under $\mathbf{G}_V=\mathbf{G}_V(\bar{\Fo})$, i.e., there is $g\in \mathbf{G}_V$ such that $g^{-1}\mathbf{T}_0(\Fo)g=\mathcal{I}(\mathbf{T})(\Fo)$. For any $\gamma\in \Gamma_{\Fo}$ and $t\in \mathbf{T}(\Fo)$, the relation
$${}^\gamma (g^{-1}\mathcal{I}_0(t) g)={}^\gamma g^{-1}\mathcal{I}_0(t){}^\gamma g={}^\gamma \mathcal{I}(t)=\mathcal{I}(t)= g^{-1}\mathcal{I}_0(t)g$$
implies that $g({}^\gamma g^{-1})\in \mathbf{T}_0$. This defines a cocycle $\Gamma_\Fo\rightarrow \mathbf{T}$ whose class lies in
\begin{equation*}
  \mathcal{D}(\mathbf{G}_V,\mathbf{T},\Fo):=\ker(H^1(\Gamma_\Fo,\mathbf{T})\xrightarrow{(\mathcal{I}_0)_*} H^1(\Gamma_\Fo,\mathbf{G}_V)).
\end{equation*}

\begin{prop}
  The set of $G_V$-conjugacy classes of $\Fo$-embeddings of $\mathbf{T}$ in $\mathbf{G}_V$ within the $\mathbf{G}_V$-conjugacy class is bijectively parameterized by $\mathcal{D}(\mathbf{G}_V,\mathbf{T},\Fo)$.
\end{prop}

From now on we abbreviate $\mathcal{D}(\mathbf{G}_V,\mathbf{T},\Fo)$ to $\mathcal{D}$. For computation, we provide a well-known description of $\mathcal{D}$. It is known that
$$H^1(\Gamma_\Fo,\mathbf{G}_V)\cong \mathbb{Z}/2.$$
The following facts are well-known.

\begin{prop}
  There is an isomorphism
   $$H^1(\Gamma_\Fo,\mathbf{T})\cong(\mathbb{Z}/2)^{\#I},$$
   and we can parametrize the map $\mathcal{I}_0$ such that
   $$\mathcal{D}\cong \{(e_i)_{i\in I}\in (\mathbb{Z}/2)^{\#I}\text{, where }\sum_{i\in I} e_i=0\}.$$
\end{prop}
\proof
See \cite[Proposition 3.5.2(c)]{Row-u3}.
\qed

\begin{cor}\label{partition-as-embeddings}
  We can parametrize $\mathcal{D}$ by partitions of $I$ into two subsets
  $$I=I_\mathrm{o}\sqcup I_\mathrm{e}$$
  such that the cardinality of one of the subsets (say $\#I_\mathrm{e}$) is even.
\end{cor}
\proof

 We can explicitly describe such a (non-unique) parametrization using the idea in \cite[Section 1.7]{walds-nil-orbit}. Suppose that we are given a pair $(z,x)=((z_i),(x_i))_{i\in I}$, where
\begin{enumerate}[(i)]
  \item each $z_i$ generates the unramified field extension $E_i/F$ such that ${}^{\sigma_i}z_i=z_i$;
      \item $x_i\in E_i$ such that ${}^{c_i}x_i=x_i$, or equivalently $x_i\in \Eo_i$, and $\prod_{i\in I}\delta_{E_i/\Eo_i}(x_i)=1$;
      \label{condition on xi}
      \item $n=\sum_{i\in I}n_i$, where $n_i=[E_i:F]$ is odd.
\end{enumerate}
   Hence $z\in {T}$ is embedded into ${G}_V$ by $\mathcal{I}_0$ as a regular element, such that $Z_{G_V}(\mathcal{I}_0(z))=T_0:=\mathbf{T}_0(\Fo)$. Two pairs $(z,x)$ and $(z',x')$ are called $\mathbf{G}_V$-conjugate (or stably conjugate) if $z_i\in \Gamma_{\Fo}z'_i$ for all $i\in I$, and called $G_V$-conjugate if furthermore $\delta_{E_i/\Eo_i}(x_i)=\delta_{E_i/\Eo_i}(x_i')$ for all $i\in I$.  
 If we identify the element $x=(x_i)_{i\in I}$ to the element in $x\in \mathcal{D}$, then $T_x$ and $T_{x'}$ are $\mathbf{G}_V$-conjugate (resp. $G_V$-conjugate) if and only if $(z,x)$ and $(z',x')$ are $\mathbf{G}_V$-conjugate (resp. $G_V$-conjugate).

 Note that, for fixed $z$, we can choose an $x$ with $x_i\in \{1,\varpi\}$ to represent a $G_V$-conjugacy class of $(z,x)$. If we define $I_{\mathrm{o}}$ (resp. $I_{\mathrm{e}}$) to be the subset consisting of $i\in I$ such that $x_i=1$ (resp. $x_i=\varpi$), then the last condition of $x=(x_i)_{i\in I}$ in (\ref{condition on xi}) implies that $\#I_{\mathrm{e}}$ must be even.

From the orthogonal decomposition $V=\oplus_{i\in I}V_i$, we identify each $V_i$ with $E_i$ as an $F$-vector space. If we define a Hermitian form on $V_E:=\oplus_{i\in I}E_i$ by $h_E=\perp_{i\in I}h_{E_i}$ where
 $$h_{E_i}(w_i,w'_i)=tr_{E_i/F}(x_i{}^{c_i}w_iw_i')$$
 for all $w_i,\,w'_i\in E_i$, then the last condition of $x=(x_i)_{i\in I}$ in (\ref{condition on xi}) implies that $(V_E,h_{E})$ is isometrically isomorphic to $(V,h_V)$ as an Hermitian space, and the corresponding unitary group $\mathbf{G}_{V_E}$ is $\Fo$-isomorphic to $\mathbf{G}_V$. The torus $\mathbf{T}$ is embedded into $\mathbf{G}_{V_E}$ as a diagonal torus, and the embedding $\mathcal{I}_x:\mathbf{T}\rightarrow \mathbf{G}_{V_E}$ is defined by composing the torus embedding with the isomorphism from $\mathbf{G}_{V_E}$ to $\mathbf{G}_V$.

\qed

\subsection{Lattice sequences}\label{section lattice seq}

Let $\Lambda$ be an $\mathfrak{o}_F$-lattice sequence in $V$, defined as a function $\Lambda$ from $\mathbb{Z}$ to the set of $\mathfrak{o}_F$-lattice in $V$ such that
\begin{enumerate}[(i)]
  \item $\Lambda(k)\subseteq \Lambda(j)$ if $j\leq k$;
  \item there exists a positive integer $e=e(\Lambda/\mathfrak{o}_F)$ such that $\varpi\Lambda(k)=\Lambda(k+e)$ for all $k\in \mathbb{Z}$.
\end{enumerate}
The integer $e(\Lambda/\mathfrak{o}_F)$ is called the $\mathfrak{o}_F$-period of $\Lambda$.

Recall that $h_V$ is an Hermitian form on $V$ that defines the unitary group $G_V$. Given an $\mathfrak{o}_F$-lattice  $L$ in $V$, we define the conjugate-dual of $L$ as
$$L^*:=\{v\in V,\,h_V(v,L)\subseteq \mathfrak{p}_F\}.$$
An $\mathfrak{o}_F$-lattice sequence $\Lambda$ in $V$ is called conjugate-self-dual if there exists $j\in \mathbb{Z}$ such that $\Lambda(k)^*=\Lambda(j-k)$  for all $k\in \mathbb{Z}$. As in \cite{stevens-supercusp}, we may scale the indices of $\Lambda$ such that
\begin{equation}\label{even-period-assumption}
  j=1\text{ and the period of }\Lambda\text{ is always even}.
\end{equation}
 For example, if $\Lambda$ is the lattice sequence representing $\{\mathfrak{p}_F^k\}_{k\in \mathbb{Z}}$ in $V=F$, and $h_V$ is defined by $(x,y)\mapsto x{}^{c}y$, then $\Lambda$ is enumerated as
$$\cdots,\,\Lambda(-1)=\Lambda(0)=\mathfrak{o}_F,\,\Lambda(1)=\Lambda(2)=\mathfrak{p}_F,\cdots$$
and so on. We then extend the domain of $\Lambda$ from $\mathbb{Z}$ to $\mathbb{R}$ by defining $\Lambda(r)=\Lambda(\lceil r\rceil)$ for all $r\in \mathbb{R}$, and write $\Lambda(r+)=\cup_{s>r}\Lambda(s)$.

As in \cite[Section 2.8]{BK-ss-type}, we define the direct sum of two lattices. Given $F$-vector spaces $V_i$, for $i=1,2$, and $\mathfrak{o}_F$-lattice sequences $\Lambda_i$ of $V_i$ of period $e_i$, we define $\Lambda=\Lambda_1\oplus\Lambda_2$ to be an $\mathfrak{o}_F$-lattice sequence of $V_1\oplus V_2$ of period $e=\mathrm{lcm}(e_1,e_2)$ by
$$\Lambda(r)=\Lambda_1(e_1r/e)\oplus\Lambda_2(e_2r/e),\,\text{for all }r\in \mathbb{R}.$$
We can check that $(\Lambda_1\oplus\Lambda_2)\oplus\Lambda_3= \Lambda_1\oplus (\Lambda_2\oplus\Lambda_3)$ and $\Lambda_1\oplus\Lambda_2\cong \Lambda_2\oplus\Lambda_1$, so that we can define the notion of direct sum inductively.

In the sequel, we will study the conjugate-self-dual lattice sequences $\Lambda_x,\,x\in\mathcal{D}$, defined as follows. Suppose $I=I_\mathrm{o}\sqcup I_\mathrm{e}$ is a partition $I=I_\mathrm{o}\sqcup I_\mathrm{e}$ corresponding to an embedding $x\in \mathcal{D}
$ as in Corollary \ref{partition-as-embeddings}. We write $V_{I_\mathrm{o}}=\oplus_{i\in I_\mathrm{o}}V_i$ and  $V_{I_\mathrm{e}}=\oplus_{i\in I_\mathrm{e}}V_i$, such that the decomposition $V=V_{I_\mathrm{o}}\oplus V_{I_\mathrm{e}}$ is orthogonal. We choose a basis $\mathcal{B}_{I_\mathrm{o}}$ for $V_{I_\mathrm{o}}$ such that the matrix of the Hermitian form $h_{I_\mathrm{o}}$ has an expression similar to (\ref{J-matrix}) and, since $n_\mathrm{e}=\dim V_{I_\mathrm{e}}$ is even, we choose a basis $\mathcal{B}_{I_\mathrm{e}}$ for $V_{I_\mathrm{e}}$, enumerated as $\{\mathbf{e}_1,\dots\mathbf{e}_{n_\mathrm{e}/2},\mathbf{e}_{-n_\mathrm{e}/2},\dots,\mathbf{e}_{-1}\}$, such that the relations similar to (\ref{even-Hermitian}) hold for $h_{I_\mathrm{e}}$ (in place of $h_W$). It is clear that $(V,h_V)$ and $(V_{I_\mathrm{o}}\oplus V_{I_\mathrm{e}},h_{I_\mathrm{o}}\perp h_{I_\mathrm{e}})$ are isometrically isomorphic. If we now define $L_{I_\mathrm{o}}$ and $L_{I_\mathrm{e}}$ be the $\mathfrak{o}_F$-lattices generated by $\mathcal{B}_{I_\mathrm{o}}$ and $\mathcal{B}_{I_\mathrm{e}}$ respectively, then we check that
$$L_{I_\mathrm{o}}^*=\varpi L_{I_\mathrm{o}}\text{ and }L_{I_\mathrm{e}}^*=L_{I_\mathrm{e}}.$$
We define two lattice sequences $\Lambda_{x,I_\mathrm{o}}$ and $\Lambda_{x,I_\mathrm{e}}$ containing $L_{I_\mathrm{o}}$ and $L_{I_\mathrm{e}}$ respectively, both of period 2, with the indices shifted such that
\begin{equation*}
  \begin{split}
    &\Lambda_{x,I_\mathrm{o}}(0)\supsetneq\Lambda_{x,I_\mathrm{o}}(1)=\Lambda_{x,I_\mathrm{o}}(2),
    \\
    &\Lambda_{x,I_\mathrm{e}}(0)=\Lambda_{x,I_\mathrm{e}}(1)\supsetneq\Lambda_{x,I_\mathrm{e}}(2).
  \end{split}
\end{equation*}
We then define
$$\Lambda_x=\Lambda_{x,I_\mathrm{o}}\oplus\Lambda_{x,I_\mathrm{e}}.$$
Hence $\Lambda_x$ is conjugate-self-dual and its period is 2.

\subsection{Orders and subgroups}\label{section orders}

Let $\tilde{G}=\tilde{G}_V$ or $\tilde{G}_W$, and write ${G}=\tilde{G}^\sigma$. For any lattice sequence $\Lambda$ in $V$ or $W$, we associate to $\Lambda$ a hereditary order $\mathfrak{A}_{\Lambda}$ in $\tilde{A}_V=\mathrm{End}_F(V)$ or  $\mathrm{End}_F(W)$ and its Jacobson radical $\mathfrak{P}_{\Lambda}$. We then define, for $r\in\mathbb{R}_{\geq 0}$,
$$\mathfrak{P}_{\Lambda}^{r}=\bigcup_{\begin{smallmatrix}
  n\in \mathbb{Z}_{\geq0}
  \\
  n\geq r
\end{smallmatrix}}\mathfrak{P}_{\Lambda}^n\text{ and }\mathfrak{P}_{\Lambda}^{r+}=\bigcup_{\begin{smallmatrix}
  n\in \mathbb{Z}_{>0}
  \\
  n> r
\end{smallmatrix}}\mathfrak{P}_{\Lambda}^n.$$
We then define a decreasing filtration $\{U_\Lambda^r\}_{r\geq 0}$ of compact open subgroups of $\tilde{G}$ by
\begin{equation*}
U_{\Lambda}=U^0_{\Lambda}=\mathfrak{A}_{\Lambda}^\times,\,U^r_{\Lambda}=1+\mathfrak{P}_{\Lambda}^r\text{, and }U^{r+}_{\Lambda}=1+\mathfrak{P}_{\Lambda}^{r+},
\end{equation*}
for $r\in\mathbb{R}_{>0}$. If $\Lambda$ is conjugate-self-dual, then $\mathfrak{P}_{\Lambda}^{r}$ and $U^r_{\Lambda}$ are conjugate-self-dual, which means that they are $\sigma$-invariant. we define
$$U^r_{\Lambda,F/\Fo}=U^r_{\Lambda}\cap G\text{ and }U^{r+}_{\Lambda,F/\Fo}=U^{r+}_{\Lambda}\cap G$$
for all $r\in\mathbb{R}_{\geq0}$.

Write $E=\oplus_{i\in I}E_i$ and denote $\mathfrak{o}_E=\oplus_{i\in I}\mathfrak{o}_{E_i}$, $E^\times=\prod_{i\in I}E^\times_i\cong \tilde{T}$ and
$(E^\times)^\sigma=\prod_{i\in I}\mathrm{U}_1(E_i/\Eo_i)\cong  {T}$. For a fixed $i_\circ\in I$, we view $E_{i_\circ}^\times$ as an elliptic maximal torus in $\tilde{G}_{V_{i_\circ}}$. Given an embedding $x\in \mathcal{D}$, we define a morphism $\mathcal{I}_{i_\circ,x}: E^\times \rightarrow G_W$ by
\begin{equation}\label{image-of-E}
\begin{split}
 \mathcal{I}_{i_\circ,x}: &E^\times \rightarrow \tilde{T}_{i_\circ}\times {T}\rightarrow M\cong\tilde{G}_{V_{i_\circ}}\times G_V\subset  G_W,
  \\
  &t=(t_i)_{i\in I}\mapsto (t_{i_\circ},t({}^ct^{-1}))\mapsto ({}^{\sigma_{i_\circ}}t_{i_\circ},\mathcal{I}_x(t({}^ct^{-1})),t_{i_\circ}).
\end{split}
\end{equation}
Denote the image by $\mathcal{I}_{i_\circ,x}(E^\times)$ for a moment, and regard it as a subtorus of $\tilde{T}_{i_\circ}\times {T}_x$.

Suppose that $\Lambda$ is a lattice sequence in $W$ such that each $\Lambda(r)$ is invariant by $\mathcal{I}_{i_\circ,x}(\prod_{i\in I}\mathfrak{o}^\times_{E_i})$. In the sequel, we will say for short that $\Lambda$ is an $\mathfrak{o}_{E}$-lattice sequence, and write  $\mathfrak{P}^r_{\Lambda,E}$ the subset of elements in  $\mathfrak{P}^r_{\Lambda}$ centralized by $\mathcal{I}_{i_\circ,x}(\prod_{i\in I}\mathfrak{o}^\times_{E_i})$. Note that we write $\mathfrak{A}_{\Lambda,E}=\mathfrak{P}^r_{\Lambda,E}$. We define, for $r\in\mathbb{R}_{\geq0}$,
\begin{equation*}
    U^r_{\Lambda,E}=U^r_{\Lambda}\cap Z_{\tilde{G}_W}(\mathcal{I}_{i_\circ,x}(E^\times))=1+\mathfrak{P}^r_{\Lambda,E}
  \end{equation*}
  and if $\Lambda$ is furthermore conjugate-self-dual, we define
\begin{equation*}
    U^r_{\Lambda,E/\Eo}=U^r_{\Lambda}\cap Z_{G_W}(\mathcal{I}_{i_\circ,x}(E^\times))
  \end{equation*}
 and similarly for $U^{r+}_{\Lambda,E}$ and $U^{r+}_{\Lambda,E/\Eo}$.

Finally, as $T=\mathbf{T}(\Fo)$ and $\tilde{T}=\mathbf{T}(F)$, we denote the image $\mathcal{I}_x(T)$ in ${G}_V$ by $T_x$ and $\mathcal{I}_x(\tilde{T})$ of $\tilde{G}_V$ by $\tilde{T}_x$, for $x\in \mathcal{D}$. For $r\in \mathbb{R}_{\geq0}\cup\{t+\text{, where }t\in\mathbb{R}_{\geq0} \}$, if we denote
$$\tilde{T}^r=\prod_{i\in I}\tilde{T}_i^r\text{ and }{T}^r=\prod_{i\in I}{T}_i^r,$$
where
$$\tilde{T}_i^r=U_{E_i}^r\text{ and }{T}_i^r=\mathrm{U}_1(E_i/\Eo_i)\cap U_{E_i}^r,$$
then we set
$$\tilde{T}_x^r=\mathcal{I}_x(\tilde{T}^r)\text{ and }{T}_x^r=\mathcal{I}_x({T}^r).$$

\subsection{Lattices of higher ranks}\label{section lattice higher ranks}

Given $i_\circ\in I$ and $x\in \mathcal{D}$, we define a lattice sequence $\Lambda_{i_\circ,x}$ in $W=V_{i_\circ,-}\oplus V\oplus V_{i_\circ,+}$ as follows. Let $\Lambda_{i_\circ,-}$ and $\Lambda_{i_\circ,+}$ be two lattice sequences isomorphic to $\Lambda_{i_\circ}$. Define a conjugate-self-dual lattice sequence $\Lambda_{i_\circ,x}$ in $W$ of period 6 by
$$\Lambda_{i_\circ,x}(r)=\Lambda_{i_\circ,-}(\frac{r-1}{3})\oplus \Lambda_x(\frac{r}{3})\oplus \Lambda_{i_\circ,+}(\frac{r+1}{3}),$$
for all $r\in \mathbb{R}_{\geq 0}$.

Following \cite[Section 7.2.2]{stevens-supercusp}, we define two conjugate-self-dual lattice sequences in $W$, $\mathfrak{M}^y_{i_\circ,x}$
and $\mathfrak{M}^z_{i_\circ,x}$, each of period 2, by assigning each of them to one of the following two lattice sequences:
$$0\mapsto\Lambda_{i_\circ,x}(-2)\text{ and }1\mapsto \Lambda_{i_\circ,x}(3)$$
or
$$0\mapsto \Lambda_{i_\circ,x}(0)\text{ and }1\mapsto \Lambda_{i_\circ,x}(1),$$
such that the reductive quotient
\begin{equation*}
  U_{\mathfrak{M}^w_{i_\circ,x,E/\Eo}}/U^1_{\mathfrak{M}^w_{i_\circ,x},E/\Eo}\cong \begin{cases}
  \mathrm{U}_3(\mathbf{k}_{E_{i_\circ}}/\mathbf{k}_{\Eo_{i_\circ}})\times \prod_{i\in I-\{i_\circ\}}\mathrm{U}_1(\mathbf{k}_{E_{i}}/\mathbf{k}_{\Eo_{i}})&\text{ when }w=y,
  \\
   \mathrm{U}_2(\mathbf{k}_{E_{i_\circ}}/\mathbf{k}_{\Eo_{i_\circ}})\times \prod_{i\in I}\mathrm{U}_1(\mathbf{k}_{E_{i}}/\mathbf{k}_{\Eo_{i}})&\text{ when }w=z.
\end{cases}
\end{equation*}

Alternatively, we can define $\Lambda_{i_\circ,x}$, $\mathfrak{M}^y_{i_\circ,x}$ and $\mathfrak{M}^z_{i_\circ,x}$ as follows. Define the following conjugate-self-dual lattice sequences in $V_{\pm i_\circ}:=V_{i_\circ,-}\oplus V_{i_\circ}\oplus V_{i_\circ,+}$,
$$\Lambda_{\pm i_\circ}(r)=\Lambda_{i_\circ,-}(\frac{r-1}{3})\oplus \Lambda_{i_\circ}(\frac{r}{3})\oplus \Lambda_{i_\circ,+}(\frac{r+1}{3}),$$
and  $\mathfrak{M}^y_{\pm i_\circ}$ and $\mathfrak{M}^z_{\pm i_\circ}$ by assigning each of them to one of the following two lattice sequences:
$$0\mapsto\Lambda_{\pm i_\circ}(-2)\text{ and }1\mapsto \Lambda_{\pm i_\circ}(3)$$
or
$$0\mapsto \Lambda_{\pm i_\circ}(0)\text{ and }1\mapsto \Lambda_{\pm i_\circ}(1),$$
such that the reductive quotient
\begin{equation*}
  U_{\mathfrak{M}^w_{\pm i_\circ},E/\Eo}/U^1_{\mathfrak{M}^w_{\pm i_\circ},E/\Eo}\cong \begin{cases}
  \mathrm{U}_3(\mathbf{k}_{E_{i_\circ}}/\mathbf{k}_{\Eo_{i_\circ}}) &\text{ when }w=y,
  \\
   \mathrm{U}_2(\mathbf{k}_{E_{i_\circ}}/\mathbf{k}_{\Eo_{i_\circ}})\times \mathrm{U}_1(\mathbf{k}_{E_{i_\circ}}/\mathbf{k}_{\Eo_{i_\circ}})&\text{ when }w=z.
\end{cases}
\end{equation*}
We then define
$$\Lambda_{i_\circ,x}=\Lambda_{\pm i_\circ}\oplus_{i\neq i_\circ}\Lambda_{i}\text{ and }\mathfrak{M}^w_{i_\circ,x}=\mathfrak{M}^w_{\pm i_\circ}\oplus_{i\neq i_\circ}\Lambda_{i},$$
for $w=y$ or $z$, such that
$$\Lambda_{i_\circ,x}\cap \mathrm{End}_F(V)=\mathfrak{M}^w_{i_\circ,x}\cap \mathrm{End}_F(V)=\Lambda_x.$$

\newpage\section{Representations}\label{section repres}

\subsection{Characters}\label{section characters}

Let ${Z_{F/\Fo}}$ be the set of skew characters of $F^\times$ relative to the quadratic extension $F/\Fo$, defined by
$${Z_{F/\Fo}}=\{\text{characters }\tilde{\chi}:F^\times\rightarrow\mathbb{C}^\times,\,{}^{c}\tilde{\chi}^{-1}=\tilde{\chi}\}.$$
We then have a partition ${Z_{F/\Fo}}={Z_{F/\Fo}^+}\sqcup {Z_{F/\Fo}^-}$, where
$${Z_{F/\Fo}^+}=\{\tilde{\chi}\in {Z_{F/\Fo}},\,\tilde{\chi}|_{\Fo^\times}=1\},\text{and } {Z_{F/\Fo}^-}=\{\tilde{\chi}\in {Z_{F/\Fo}},\,\tilde{\chi}|_{\Fo^\times}=\delta_{F/\Fo}\}.$$
We call these characters $+$-skew and $-$-skew respectively. We denote $\tilde{\chi}_+\in  {Z_{F/\Fo}^+}$ the trivial character and fix, once and for all, a character $\tilde{\chi}_-\in {Z_{F/\Fo}^-}$. For example, we may choose $\tilde{\chi}_-$ to be the unramified quadratic character of $F^\times$.

\begin{prop}\label{U1-base-change}
The base-change map
$$(\mathrm{U}_1(F/\Fo))^\wedge\rightarrow (F^\times)^\wedge,\,\xi\mapsto(\tilde{\xi}=\xi\circ(1-c):x\mapsto \xi(x\cdot{}^{c}x^{-1}))$$
is injective, and its image is ${Z_{F/\Fo}^+}$.
\end{prop}
\proof
 The statement follows from the fact that
\begin{equation}\label{base=change-seq-exact}
  1\rightarrow \Fo^\times \rightarrow F^\times  \xrightarrow{1-{c}} \mathrm{U}_1(F/\Fo)\rightarrow1
\end{equation}
is exact.
\qed\\



\begin{prop}\label{U1-base-change-compact}
 \begin{enumerate}[(i)]
   \item The morphism $ U_F^1  \xrightarrow{1-c} U^1_{F/\Fo}$ is surjective.
   \item The base-change
$$(U^1_{F/\Fo})^\wedge\xrightarrow{(1-c)^*} (U_F^1  )^\wedge  ,\,\theta\mapsto\tilde{\theta}=\theta\circ(1-c)$$
is injective.
\item The image of this base-change is the set of skew characters (satisfying ${}^c\tilde{\theta}^{-1}=\tilde{\theta}$) of $U_F^1$.
 \end{enumerate}
\end{prop}
\begin{proof}
  From (\ref{base=change-seq-exact}), for each $y\in U^1_{F/\Fo}$, there is $x\in F^\times$ such that $(1-c)x=y$. Write $x=\varpi^kzu$, then $(1-c)(\varpi^kz)\cdot (1-c)(u)=y$. Since $y$ is a pro-p element, we have $(1-c)(\varpi^kz)=1$ and $(1-c)u=y$. This proves the first statement, and the second and third statement follows easily.
\end{proof}

Let $E_\circ/F$ be a finite unramified extension. For $\epsilon=+$ or $-$, we define $\tilde{\chi}^{E_\circ}_\epsilon=\tilde{\chi}_\epsilon\circ N_{E_\circ/F}$. If we choose $\tilde{\chi}_\epsilon$ as above, then $\tilde{\chi}^{E_\circ}_+$ is again trivial, and $\tilde{\chi}^{E_\circ}_-$ is again unramified quadratic, since $E_\circ/F$ is odd unramified.


\begin{prop}\label{wild-part-tilde}
Suppose that $\tilde{\xi}_\circ$ is a character of $E_\circ^\times$. There is a unique character  $\tilde{\xi}_{\circ,\mathrm{wd}}$ of $E_\circ^\times$ such that
  \begin{enumerate}[(i)]
    \item $\tilde{\xi}_{\circ,\mathrm{wd}}|_{U_{E_\circ}^1}=\tilde{\xi}_\circ|_{U_{E_\circ}^1}$;
    \item $\tilde{\xi}_{\circ,\mathrm{wd}}(\varpi)=1$;
    \item $\tilde{\xi}_{\circ,\mathrm{wd}}$ has a finite order of a $p$-power.
  \end{enumerate}
\end{prop}
The character $\tilde{\xi}_{\circ,\mathrm{tm}}=\tilde{\xi}^{-1}_{\circ,\mathrm{wd}}\tilde{\xi}_\circ$ is tamely ramified. If $\tilde{\xi}_\circ$ is a skew character (only when $[E_\circ:F]$ is odd), then $\tilde{\xi}_{\circ,\mathrm{wd}}$ is always $+$-skew by construction, and so
$$\tilde{\xi}_\circ\text{ is }\epsilon\text{-skew if and only if }\tilde{\xi}_{\circ,\mathrm{tm}}\text{ is }\epsilon\text{-skew}.$$

\subsubsection{Regular characters}\label{section regular char}

Suppose now ${\tilde{\xi}}_\circ$ is a character of $E_\circ^\times$ of positive level $d$, i.e., $\tilde{\xi}_\circ|_{U_{E_\circ}^{2d}}\neq 1$ and $\tilde{\xi}_\circ|_{U_{E_\circ}^{2d+}}\equiv 1$. (Beware of the convention we used in (\ref{even-period-assumption}).) We only consider $d$-regular characters, which means that ${}^\gamma({\tilde{\xi}}_\circ|_{U_{E_\circ}^{2d}})$ are all different for $\gamma$ ranging over $\Gamma_{E_\circ/F}$. (When we say ${\tilde{\xi}}_\circ$ is $d$-regular, we already know that the level of ${\tilde{\xi}}_\circ$ is $d$.)

We can choose an element $b_\circ \in  E_\circ $ such that
 \begin{equation}\label{element-b}
   {\tilde{\xi}}_\circ(1+x)=\psi_F(\mathrm{tr}_{E_\circ/F}(b_\circ x)),\text{ for all }x\in \mathfrak{p}_{E_\circ}^{ d+}.
 \end{equation}
 This $b_\circ$ is unique in $ \mathfrak{p}_{E_\circ}^{ -2d}$ modulo $ \mathfrak{p}_{E_\circ}^{ -d}$. Note that the $d$-regularity of ${\tilde{\xi}}_\circ$ implies that $E_\circ=F[b_{\circ,-d}]$, where $b_{\circ,-d}\in {\boldsymbol{\mu}}_{E_\circ}$ is the $(-d)$th coefficient in the $\varpi$-expansion of $b_\circ$.


More generally, we describe regular characters of an elliptic unramified maximal torus $G_V$. For each $i\in I$, let $E_i/F$ be the unramified extension of odd degree $n_i$. Take ${\tilde{\xi}}=\boxtimes_{i\in I}{\tilde{\xi}}_i$ a character of $\tilde{T}=\prod_{i\in I}E_i^\times$. We impose the following conditions.
\begin{itemize}
\item All $\tilde{\xi}_i$ have the same level $d$, and are $d$-regular.
  \item All $\tilde{\xi}_i$ are $+$-skew.
  \item All $\tilde{\xi}|_{U_{E_i}^{2d}}$ are not Galois conjugate to each other, for $i$ ranging over $I$.
\end{itemize}
 We call this kind of characters \emph{$d$-regular +-skew characters}. For each $\tilde{\xi}_i$, there exists a character $\xi_i$ of $\mathrm{U}_1(E_i/\Eo_i)$ such that $\tilde{\xi}_i=\xi_i\circ(1-{c}_i)$. Write $\xi=\boxtimes_{i\in I}\xi_i$ as a character of $T=\prod_{i\in I}\mathrm{U}_1(E_i/\Eo_i)$, then $\xi$  is a regular character considered in \cite[Sec. 3.6]{reeder-pos-depth}.

\subsection{Maximal types for GL}\label{section Maximal types for GL}

Let $E_\circ/F$ be a finite unramified extension of degree $n$, and $E_\circ^\times$ is embedded into $\tilde{G}_{V_\circ}\cong \mathrm{GL}_{n_\circ}(F)$ as an elliptic unramified maximal torus (where all such embeddings are conjugate with each other.) Given a regular character of $E_\circ^\times$, we construct a maximal type of $\tilde{G}_{V_\circ}$ in three steps.

\subsubsection{First step: simple characters}

Let $\Lambda_\circ$ be an $E_\circ^\times$-invariant lattice sequence in $V$, and ${\tilde{\xi}}_\circ$ be a $d$-regular character of level $d>0$. We define the following subgroups
\begin{equation*}
\begin{split}
   & \tilde{H}_{\Lambda_\circ}^1=U_{E_\circ}^1U^{d+}_{\Lambda_\circ},\, \tilde{J}_{\Lambda_\circ}^1=U_{E_\circ}^1U^{d}_{\Lambda_\circ},
  \\
  & \tilde{J}_{\Lambda_\circ}=U_{E_\circ}U^{d}_{\Lambda_\circ}\text{, and } \tilde{\mathbf{J}}_{\Lambda_\circ}=E_\circ^\times U^{d}_{\Lambda_\circ}.
\end{split}
\end{equation*}

We construct a character ${\tilde{\theta}}_{\Lambda_\circ}={\tilde{\theta}}_{{\Lambda_\circ},\tilde{\xi}_{\circ}}$ of $\tilde{H}_{\Lambda_\circ}^1$, whose restriction ${\tilde{\theta}}_{\Lambda_\circ}|_{U_{E_{\circ}}^1}$ is ${\tilde{\xi}}_\circ|_{U_{E_\circ}^1}$. We first note that
$$U_{E_\circ}^{d+}/U_{E_\circ}^{2d+}\cong \mathfrak{p}_{E_\circ}^{d+}/\mathfrak{p}_{E_\circ}^{2d+}\text{ and }U_{\Lambda_\circ}^{d+}/U_{\Lambda_\circ}^{2d+}\cong \mathfrak{P}_{\Lambda_\circ}^{d+}/\mathfrak{P}_{\Lambda_\circ}^{2d+}.$$
From the restriction $\tilde{\xi}_\circ|_{U_{E_\circ}^{d+}}$, we obtain a character, still denoted by $\tilde{\xi}_\circ$, on the additive group $\mathfrak{p}_{E_\circ}^{d+}$ trivial on $\mathfrak{p}_{E_\circ}^{2d+}$. From the root-space decomposition
\begin{equation}\label{root-space-decomp-GL}
  \mathfrak{P}_{\Lambda_\circ}^{d+}/\mathfrak{P}_{\Lambda_\circ}^{2d+}\cong  \mathfrak{p}_{E_\circ}^{d+}/\mathfrak{p}_{E_\circ}^{2d+}\oplus \mathfrak{n}_{\Lambda_\circ}^{d+}/ \mathfrak{n}_{\Lambda_\circ}^{2d+},
\end{equation}
where $\mathfrak{n}_{\Lambda_\circ}$ is the sum of root spaces in $\tilde{A}_{V_\circ}$ (on which ${E_\circ}^\times$ acts non-trivially by conjugation) and $\mathfrak{n}_{\Lambda_\circ}^{r}=\mathfrak{n}_{\Lambda_\circ}\cap \mathfrak{P}_{\Lambda_\circ}^r$, we can extend $\tilde{\xi}_\circ$ trivially to $ \mathfrak{n}_{\Lambda_\circ}^{d+}$, and hence to a character
$\tilde{\theta}_{\Lambda_\circ}$ on $U_{\Lambda_\circ}^{d+}$ trivial on $U_{\Lambda_\circ}^{2d+}$,
which agrees with $\tilde{\xi}_\circ$ on $U_{\Lambda_\circ}^{d+}\cap U_{E_\circ}^1=U_{E_\circ}^{d+}$. Finally, we extend $\tilde{\theta}_{\Lambda_\circ}$ to a character, also denoted by $\tilde{\theta}_{\Lambda_\circ}$, of $\tilde{H}^1=U_{E_\circ}^1U^{d+}_{\Lambda_\circ}$ which agrees with $\tilde{\xi}_\circ$ on $ U_{E_\circ}^1$.

\begin{rmk}
The simple character can be constructed alternatively as follows. Recall from Section \ref{section regular char} that
$${\tilde{\xi}}_\circ(1+x)=\psi_F(\tr_{E_\circ/F}(b_\circ x))\qquad\text{ for all }x\in \mathfrak{p}_{E_\circ}^{d+}$$
for a fixed $b_\circ\in \mathfrak{p}_{E_\circ}^{-2d}$ modulo $\mathfrak{p}_{E_\circ}^{-d}$. We define
$$\psi_{b_\circ}(1+X)=\psi_F(\tr_{\tilde{A}_{V_\circ}/F}(b_\circ X))\qquad\text{ for all }1+X\in U^{d+}_{\Lambda_\circ},$$
 and define ${\tilde{\theta}}_\circ'$ on $\tilde{H}^1_{}=U^1_{E_\circ}U^{d+}_{\Lambda_\circ}$ by
$${\tilde{\xi}}_\circ\text{ on }U_{E_\circ}^1\qquad\text{and}\qquad\psi_{b_\circ}\text{ on }U^{d+}_{\Lambda_\circ}.$$
We now check that $\tilde{\theta}_\circ'=\tilde{\theta}_\circ$. Clearly $\tilde{\theta}_\circ'|_{U_{E_\circ}^1}=\tilde{\theta}_\circ|_{U_{E_\circ}^1}=\tilde{\xi}_\circ|_{U_{E_\circ}^1}$. If $1+X\in {U}_{\Lambda_\circ}^{d+}$, then
$$\tilde{\theta}_\circ'(1+X)=\psi_F(\tr_{A_\circ/F}(b_\circ X))=\psi_F(\tr_{E_\circ/F}(b_\circ s_{E_\circ}^{\tilde{A}_{V_\circ}}(X))),$$
where $s_{E_\circ}^{\tilde{A}_{V_\circ}}:\mathfrak{P}_{\Lambda_\circ}^{d+}/\mathfrak{P}_{\Lambda_\circ}^{2d+}\rightarrow \mathfrak{p}_{E_\circ}^{d+}/\mathfrak{p}_{E_\circ}^{2d+}$ is the natural projection defined by the root-space decomposition (\ref{root-space-decomp-GL}), and can also be defined by a tame co-restriction
\begin{equation}\label{tame-cores}
  s_{E_\circ}^{\tilde{A}_{V_\circ}}:\mathfrak{P}_{\Lambda_\circ}^{r}\rightarrow \mathfrak{p}_{E_\circ}^{r}\text{, for all }r\in \mathbb{R},
\end{equation}
on ${\tilde{A}_{V_\circ}}$ relative to $E_\circ/F$ (see its definition and properties in \cite[(1.3.3)-(1.3.8)]{BK}). The above is then equal to
$\tilde{\xi}_\circ(1+s_{E_\circ}^{\tilde{A}_{V_\circ}}(X))$, which is equal to $\tilde{\theta}_\circ(1+X)$ by its construction.

\qed\end{rmk}

According to \cite[Section 3.2]{BK}, the simple character $\tilde{\theta}_{\Lambda_\circ}$ is based on a simple stratum of the form
$$\tilde{\mathfrak{s}}_\circ=[\Lambda_\circ,d,0,\beta_\circ],$$
where $\Lambda_\circ$ is an $\mathfrak{o}_F$-lattice chain, and $\beta_\circ\in \mathfrak{P}_{\Lambda_\circ}^{-2d}$ such that $E_\circ=F[\beta_\circ]$ is a field and $\Lambda_\circ$ can be regarded as an $\mathfrak{o}_{E_\circ}$-lattice chain. We can view $\beta_\circ$ is an approximation of the element $b_\circ$ defined in (\ref{element-b}), in the sense that
\begin{equation}\label{b-approx}
  b_\circ\equiv \beta_\circ \mod \mathfrak{p}_{E_\circ}^{-d}.
\end{equation}
  The $d$-regularity implies that $\tilde{\mathfrak{s}}_\circ$ is minimal over $F$ \cite[(1.4.14)]{BK}, or equivalently, $E$ is generated by $b_{\circ,-d}$ over $F$ as known in Section \ref{section regular char}. The stratum $\tilde{\mathfrak{s}}_\circ$ is called skew if  $\Lambda_\circ$ is conjugate-self-dual and ${}^\sigma\beta_\circ=-{}^c\beta_\circ=\beta_\circ$.

\subsubsection{Second step: Heisenberg representations}

The quotient $\tV=\tilde{J}^1/\tilde{H}^1$ is naturally a $\mathbf{k}_F$-vector space. Since $\tilde{\theta}_{\Lambda_\circ}$ is a simple character, the space $\mathfrak{V}$ is equipped with a non-degenerate alternating form \cite[Sec 3.5]{BK}
$$h_{\tilde{\theta}_{\Lambda_\circ}}:(1+x,1+y)\mapsto \tilde{\theta}_{\Lambda_\circ}([x,y]).$$
Using the theory of Heisenberg representation, there exists an irreducible representation $\tilde{\eta}_{\Lambda_\circ}=\tilde{\eta}_{\Lambda_\circ,\tilde{\xi}_{\circ}}$ of $\tilde{J}_{\Lambda_\circ}^1$, unique up to isomorphism, whose restriction to $\tilde{H}_{\Lambda_\circ}^1$ is a multiple of $\tilde{\theta}_{\Lambda_\circ}$.

\subsubsection{Third step: beta extensions}

\begin{prop}\label{beta-extension-tilde}
\begin{enumerate}[(i)]
 \item There is a unique irreducible representation $\tilde{\kappa}_{\tilde{\xi}_\circ}$ of $\tilde{J}_{\Lambda_\circ}$ such that
  \begin{enumerate}[(a)]
    \item $\tilde{\kappa}_{\tilde{\xi}_\circ}|_{\tilde{J}_{\Lambda_\circ}^1}\cong \tilde{\eta}_{\Lambda_\circ}$;
    \item $\tilde{\kappa}_{\tilde{\xi}_\circ}$ is intertwined by every element in ${{E}_\circ}^\times$;
    \item the character $\det\tilde{\kappa}_{\tilde{\xi}_\circ}$ has a finite order of a $p$-power.
  \end{enumerate}
  \item All extensions satisfying (a) and (b) above are of the form $\tilde{\kappa}_{\tilde{\xi}_\circ}\cdot\tilde{\chi}$ for a character $\tilde{\chi}$ of ${U}_{{E}_\circ}/U_{{E}_\circ}^1$.
  \item All $\tilde{\kappa}_{\tilde{\xi}_\circ}\cdot\tilde{\chi}$ do not intertwine with each other.
\end{enumerate}
 \end{prop}
\proof
This is \cite[Lemma 1 of Sec. 2.3]{BH-ET1}, which follows from \cite[(5.2.2)]{BK}. Note that the reductive quotient $\tilde{J}_{\Lambda_\circ}/\tilde{J}_{\Lambda_\circ}^1$ is isomorphic to $U_{E_\circ}/U_{E_\circ}^1\cong {\boldsymbol{\mu}}_{E_\circ}$.\qed

\begin{prop}
  There is a unique irreducible representation $\tilde{\boldsymbol{\kappa}}_{\tilde{\xi}_\circ}$ of $\tilde{\mathbf{J}}_{\Lambda_\circ}$ such that
  \begin{enumerate}[(i)]
    \item $\tilde{\boldsymbol{\kappa}}_{\tilde{\xi}_\circ}|_{\tilde{J}_{\Lambda_\circ}}\cong \tilde{\kappa}_{\tilde{\xi}_\circ}$;
    \item $\varpi\in \ker \tilde{\boldsymbol{\kappa}}_{\tilde{\xi}_\circ}$;
    \item the character $\det\tilde{\boldsymbol{\kappa}}_{\tilde{\xi}_\circ}$ has a finite order of a $p$-power.
  \end{enumerate}
\end{prop}
\proof
This is \cite[Lemma 2 of Sec. 2.3]{BH-ET1}.
\qed\\

 Now the character $\tilde{\xi}_{\circ,\mathrm{tm}}=\tilde{\xi}^{-1}_{\circ,\mathrm{wd}}\tilde{\xi}_\circ$ is tamely ramified. We inflate it to a character $\tilde{\xi}_{\circ,\mathrm{tm}}$ of $\tilde{\mathbf{J}}_{\Lambda_\circ}=E_\circ^\times U^{d/2}_{\Lambda_\circ}$ which is trivial on $\tilde{J}_{\Lambda_\circ}^1$ and define
\begin{equation*}
\tilde{\boldsymbol{\lambda}}_{\tilde{\xi}_{\circ}}=\tilde{\boldsymbol{\kappa}}_{\tilde{\xi}_{\circ}}\cdot \tilde{\xi}_{{\circ},\mathrm{tm}}.
\end{equation*}

\begin{prop}\label{repres-determines-char-GL}
\begin{enumerate}[(i)]
 \item The representation $$\tilde{\pi}_{\tilde{\xi}_{\circ}}=\mathrm{cInd}_{\tilde{\mathbf{J}}_{\Lambda_{\circ}}}^{\tilde{G}_{V_\circ}}\tilde{\boldsymbol{\lambda}}_{\tilde{\xi}_{\circ}}$$
  is an irreducible supercuspidal representation of $\tilde{G}_{V_\circ}$.
  \item $\tilde{\pi}_{\tilde{\xi}_{\circ}}\cong \tilde{\pi}_{\tilde{\xi}_{\circ}'}$ if and only if ${\tilde{\xi}}_{\circ}'={}^\gamma {\tilde{\xi}}_{\circ}$ for some $\gamma\in \Gamma(E_{\circ}/F)$.
\end{enumerate}
\end{prop}
\proof
This is \cite[Theorem 2.3]{BH-ET1}.
\qed

Suppose now that $\tilde{\xi}_{\circ}$ is skew, i.e., ${}^\sigma\tilde{\xi}_{\circ}={}^{{c}}\tilde{\xi}^{-1}_{\circ}=\tilde{\xi}_{\circ}$, then $\tilde{\xi}_{\circ}|_{{\boldsymbol{\mu}}_{E_{\circ}}}$ has order $\qo^{n_\circ}+1$ and $\tilde{\xi}_{\circ}(\varpi)=\pm1$. The supercuspidal representation $\tilde{\pi}_{\tilde{\xi}_{\circ}}$ is therefore $\sigma$-invariant because of the $\sigma$-invariance of the subgroups $\tilde{H}_{\Lambda_\circ}^1,\,\tilde{J}_{\Lambda_\circ}^1,\, \tilde{J}_{\Lambda_\circ}$ and $\tilde{\mathbf{J}}_{\Lambda_\circ}$, and the uniqueness of $\tilde{\boldsymbol{\kappa}}_{\tilde{\xi}_{\circ}}$ and $\tilde{\xi}_{{\circ},\wild}$.

\subsection{Cuspidal types for unitary groups}\label{section Cuspidal types for U}

\subsubsection{Compact subgroups}

Before we construct cuspidal types for the unitary group $G_V$, we construct some compact subgroups in a more general setting. Given $i_\circ\in I$ and $x\in \mathcal{D}$, we consider the lattice sequences defined in Sections \ref{section lattice seq} and \ref{section lattice higher ranks}:
\begin{equation}\label{one-of-the-lattices}
  \Lambda_{x}\text{ for }G_V,\,
  \Lambda_{i_\circ,x}, \mathfrak{M}^y_{i_\circ,x}\text{, or }\mathfrak{M}^z_{i_\circ,x}\text{ for }G_W.
\end{equation}
Note that the periods $e_{\Lambda_{x}}$ and $e_{\mathfrak{M}^w_{i_\circ,x}}$ are both 2, while $e_{\Lambda_{i_\circ,x}}=6$.

We define the following subgroups of $G_V$,
\begin{equation}\label{compact-subgp-GV}
\begin{split}
   &{H}^1_{\Lambda_x}=T_x^1U^{d+}_{\Lambda_x,F/\Fo}
  \\
  &{J}^1_{\Lambda_x}=T_x^1U^{d}_{\Lambda_x,F/\Fo},
  \\
  &{J}_{\Lambda_x}=T_x U^{d}_{\Lambda_x,F/\Fo},
  \end{split}
\end{equation}
such that the reductive quotient is
$${J}_{\Lambda_{x}}/{J}^1_{\Lambda_{x}}\cong T_x/T_x^1\cong \prod_{i\in I}\mathrm{U}_1({\mathbf{k}}_{E_{i}}/{\mathbf{k}}_{\Eo_{i}}). $$
Fix $i_\circ\in I$, we define the following subgroups of $G_W$,
\begin{equation}\label{compact-subgp-GW-Lambda}
\begin{split}
   &{H}^1_{\Lambda_{i_\circ,x}}=(\tilde{T}^1_{i_\circ}\times T_x^1)U^{3d+}_{\Lambda_{i_\circ,x},F/\Fo},
  \\
  &{J}^1_{\Lambda_{i_\circ,x}}=(\tilde{T}^1_{i_\circ}\times T_x^1)U^{3d}_{\Lambda_{i_\circ,x},F/\Fo},
  \\
  &{J}_{\Lambda_{i_\circ,x}}=(\tilde{T}^0_{i_\circ}\times T_x)U^{3d}_{\Lambda_{i_\circ,x},F/\Fo},
  \end{split}
\end{equation}
such that the reductive quotient is
$${J}_{\Lambda_{i_\circ,x}}/{J}^1_{\Lambda_{i_\circ,x}}\cong {\mathbf{k}}^\times_{E_{i_\circ}} \times\left(\prod_{i\in I}\mathrm{U}_1({\mathbf{k}}_{E_{i}}/{\mathbf{k}}_{\Eo_{i}})\right). $$
We also define the following subgroups of $G_W$,
\begin{equation}\label{compact-subgp-GW-M}
\begin{split}
   &{H}^1_{{\mathfrak{M}^w_{i_\circ,x}}}=U^{1}_{{\mathfrak{M}^w_{i_\circ,x}},E/\Eo}U^{d+}_{{\mathfrak{M}^w_{i_\circ,x}},F/\Fo},
  \\
  &{J}^1_{{\mathfrak{M}^w_{i_\circ,x}}}
  =U^{1}_{{\mathfrak{M}^w_{i_\circ,x}},E/\Eo}U^{d}_{{\mathfrak{M}^w_{i_\circ,x}},F/\Fo},
  \\
  &{J}_{{\mathfrak{M}^w_{i_\circ,x}}}
  =U^{}_{{\mathfrak{M}^w_{i_\circ,x}},E/\Eo}U^{d}_{{\mathfrak{M}^w_{i_\circ,x}},F/\Fo},
  \end{split}
\end{equation}
such that the reductive quotient is
\begin{equation*}
  {J}_{{\mathfrak{M}^w_{i_\circ,x}}}/{J}^1_{{\mathfrak{M}^w_{i_\circ,x}}}\cong
  \begin{cases}
    \mathrm{U}_3({\mathbf{k}}_{E_{i_\circ}}/{\mathbf{k}}_{\Eo_{i_\circ}}) \times\left(\prod_{i\in I-\{i_\circ\}}\mathrm{U}_1({\mathbf{k}}_{E_{i}}/{\mathbf{k}}_{\Eo_{i}})\right)&\text{ when }w=y,
    \\
    \mathrm{U}_2({\mathbf{k}}_{E_{i_\circ}}/{\mathbf{k}}_{\Eo_{i_\circ}}) \times\left(\prod_{i\in I}\mathrm{U}_1({\mathbf{k}}_{E_{i}}/{\mathbf{k}}_{\Eo_{i}})\right)&\text{ when }w=z.
  \end{cases}
\end{equation*}

Each of the subgroups in (\ref{compact-subgp-GV}), (\ref{compact-subgp-GW-Lambda}), (\ref{compact-subgp-GW-M}) can be defined as the $\sigma$-fixed points of a compact subgroup in $\tilde{G}_V$ or $\tilde{G}_W$, using the following lemma.

\begin{lem}\label{compact groups sigma invariant}
Let $\mathfrak{L}$ be one of the lattice sequences in (\ref{one-of-the-lattices}), and let $r,s$ be two real numbers such that $s>r\geq 0$, then $$(U_{\mathfrak{L},E}^rU_{\mathfrak{L}}^s)^\sigma=(U_{\mathfrak{L},E}^r)^\sigma(U_{\mathfrak{L}}^s)^\sigma=U_{\mathfrak{L},E/\Eo}^rU_{\mathfrak{L},F/\Fo}^s.$$
\end{lem}
Note that if $\mathfrak{L}=\Lambda_{x}$, then $U^r_{\mathfrak{L},E}$ and $U^r_{\mathfrak{L},E/\Eo}$ are compact tori $\tilde{T}^r$ and ${T}_x^r$.
\proof
Suppose $t\in U_{\mathfrak{L},E}^r$ and $1+u\in U_{\mathfrak{L}}^s$ such that ${}^\sigma(t(1+u))=t(1+u)$, then $t^{-1}\cdot{}^\sigma t=(1+x)\cdot {}^\sigma(1+x)^{-1}$. Suppose that $t^{-1}\cdot{}^\sigma t=1+y\in U_{\mathfrak{L},E}^{r+}$, we want to solve for $1+x\in U_{\mathfrak{L},E}^{r+}$ such that $1+y=(1+x)\cdot {}^\sigma(1+x)^{-1}$. We obtain
$${}^\sigma(1+y)(1+y)=(t^{-1}\cdot{}^\sigma t)(t^{-1}\cdot{}^\sigma t)=1,$$
so that ${}^\sigma y+y+{}^\sigma yy=0$ and ${}^{{\sigma}} \bar{y}+\bar{y}\equiv 0\mod \mathfrak{P}_{\mathfrak{L},E}^{r++}$, where $\mathfrak{P}_{\mathfrak{L},E}^{r++}$ is the largest lattice $\mathfrak{P}_{\mathfrak{L},E}^{t}$, with $t\in \mathbb{R}$, such that $\mathfrak{P}_{\mathfrak{L},E}^{t}\subsetneq\mathfrak{P}_{\mathfrak{L},E}^{r+}$, and $\bar{y}$ is the image of $y$ in $\mathbf{k}_{\mathfrak{L},E}\cong \mathfrak{P}_{\mathfrak{L},E}^{r+}/\mathfrak{P}_{\mathfrak{L},E}^{r++}$ such that ${\sigma}$ acts on $\mathbf{k}_{\mathfrak{L},E}$ accordingly. Since $\ker(1+\sigma)=\mathrm{image}(1-\sigma)$ on $\mathbf{k}_{\mathfrak{L},E}$, we can solve for $\bar{x}\in \mathbf{k}_{\mathfrak{L},E}$ such that $\bar{y}=(1-\sigma)\bar{x}$. An argument similar to Hensel's Lemma implies that
$1+y=(1+x)\cdot {}^\sigma(1+x)^{-1}$ for some $x\in \mathfrak{P}_{\mathfrak{L},E}^{r+}$ whose image in
$\mathbf{k}_{\mathfrak{L},E}\cong \mathfrak{P}_{\mathfrak{L},E}^{r+}/\mathfrak{P}_{\mathfrak{L},E}^{r++}$ is $\bar{x}$. We then replace $t$ by $t'=t(1+x)$ and $1+u$ by $1+u'=(1+x)^{-1}(1+u)$, so that ${}^\sigma t'=t'$, ${}^\sigma (1+u')=1+u'$, and
$t(1+u)=t'(1+u')\in U_{\mathfrak{L},,E/\Eo}^rU_{\mathfrak{L},F/\Fo}^s$. We have proved that $(U_{\mathfrak{L},E}^rU_{\mathfrak{L}}^s)^\sigma\subseteq U_{\mathfrak{L},,E/\Eo}^rU_{\mathfrak{L},F/\Fo}^s$, and the converse is clear.
\qed

For example, if we define
\begin{equation*}
  \tilde{{H}}^1_{\Lambda_{x}}=\tilde{T}^1_xU^{d+}_{\Lambda_{x}}\text{ and }\tilde{H}^1_{\Lambda_{i_\circ,x}}=(\tilde{T}^1_{i_\circ}\times \tilde{T}^1_{i_\circ}\times \tilde{T}_x^1)U^{3d+}_{\Lambda_{i_\circ,x}},
\end{equation*}
then ${{H}}^1_{\Lambda_{x}}=(\tilde{{H}}^1_{\Lambda_{x}})^\sigma$ and ${H}^1_{\Lambda_{i_\circ,x}}=(\tilde{H}^1_{\Lambda_{i_\circ,x}})^\sigma$. Also if $\tilde{H}^1_{{\mathfrak{M}^w_{i_\circ,x}}}=U^{1}_{{\mathfrak{M}^w_{i_\circ,x}},E}U^{d+}_{{\mathfrak{M}^w_{i_\circ,x}}}$, then
${H}^1_{{\mathfrak{M}^w_{i_\circ,x}}}=(\tilde{H}^1_{{\mathfrak{M}^w_{i_\circ,x}}})^\sigma$.

\subsubsection{First step: semi-simple characters}

Take $\mathfrak{L}$ be one of the lattices in (\ref{one-of-the-lattices}), and let $\tilde{T}_\mathfrak{L}$ be the torus
$$\tilde{T}_x\text{ when }\mathfrak{L}=\Lambda_x\text{, and }\tilde{T}_{i_\circ}\times \tilde{T}_{i_\circ}\times \tilde{T}_x\text{ when }\mathfrak{L}=\Lambda_{i_\circ,x}\text{ or }\mathfrak{M}_{i_\circ,x}^w,$$
embedded into $\tilde{G}$, where $\tilde{G}=\tilde{G}_V$ or $\tilde{G}_W$ respectively, such that its action on $V$ or $W$ leaves $\mathfrak{L}$ invariant. This implies that $\tilde{T}_\mathfrak{L}\cap G_V={T}_x$ and $\tilde{T}_\mathfrak{L}\cap G_W= \tilde{T}_{i_\circ}\times {T}_x$. We also set $\tilde{T}^r_\mathfrak{L}$ to be $\tilde{T}^r_x$ and $\tilde{T}^r_{i_\circ}\times \tilde{T}^r_{i_\circ}\times \tilde{T}^r_x$ in the two respective cases, for all $r\in \mathbb{R}_{\geq0}\cup\{t+\text{, where }t\in\mathbb{R}_{\geq0} \}$.

Let ${\tilde{\xi}}$ be a $d$-regular $+$-skew character of $\tilde{T}_x$ and ${\tilde{\xi}}={{\xi}}\circ(1-{c})$ for a character ${{\xi}}$ of $T_x$. Define a character $\tilde{\xi}_\mathfrak{L}$ on $\tilde{T}_\mathfrak{L}$ by
$$\tilde{\xi}\text{ when }\mathfrak{L}=\Lambda_x\text{, and }\tilde{\xi}_{i_\circ}\boxtimes \tilde{\xi}_{i_\circ}\boxtimes \tilde{\xi}\text{ when }\mathfrak{L}=\Lambda_{i_\circ,x}\text{ or }\mathfrak{M}_{i_\circ,x}^w.$$

We first construct a character $\tilde{{\theta}}_\mathfrak{L}=\tilde{{\theta}}_{\mathfrak{L},\tilde{\xi}}$ of $\tilde{H}^1_\mathfrak{L}$ for $\mathfrak{L}=\Lambda_{x}$ or $\mathfrak{L}=\Lambda_{i_\circ,x}$, whose restriction on $\tilde{T}^1_\mathfrak{L}$ is $\tilde{\xi}_\mathfrak{L}|_{\tilde{T}^1_\mathfrak{L}}$. We first note that
$$\tilde{T}_\mathfrak{L}^{d+}/\tilde{T}_\mathfrak{L}^{2d+}\cong \tilde{\mathfrak{t}}_\mathfrak{L}^{d+}/\tilde{\mathfrak{t}}_\mathfrak{L}^{2d+}\text{ and }U_{\mathfrak{L}}^{e_\mathfrak{L} d+}/U_{\mathfrak{L}}^{e_\mathfrak{L} 2d+}\cong \mathfrak{P}_{\mathfrak{L}}^{e_\mathfrak{L} d+}/\mathfrak{P}_{\mathfrak{L}}^{e_\mathfrak{L} 2d+}$$
where $\tilde{\mathfrak{t}}_\mathfrak{L}$ is the Lie-algebra of $\tilde{T}_\mathfrak{L}$ and $\tilde{\mathfrak{t}}^r_\mathfrak{L}=\tilde{\mathfrak{t}}_\mathfrak{L}\cap \mathfrak{P}_\mathfrak{L}^r$. From the restriction $\tilde{\xi}_\mathfrak{L}|_{\tilde{T}_\mathfrak{L}^{d+}}$, we obtain a character, still denoted by $\tilde{\xi}_\mathfrak{L}$, on $\tilde{\mathfrak{t}}_\mathfrak{L}^{d+}$ trivial on $\tilde{\mathfrak{t}}_\mathfrak{L}^{2d+}$. From the root-space decomposition
$$\mathfrak{P}_{\mathfrak{L}}^{e_\mathfrak{L} d+}/\mathfrak{P}_{\mathfrak{L}}^{e_\mathfrak{L} 2d+}\cong  \tilde{\mathfrak{t}}_\mathfrak{L}^{d+}/\tilde{\mathfrak{t}}_\mathfrak{L}^{2d+}\oplus \mathfrak{n}_{\mathfrak{L}}^{e_\mathfrak{L} d+}/ \mathfrak{n}_{\mathfrak{L}}^{e_\mathfrak{L} 2d+},$$
where $\mathfrak{n}_{\mathfrak{L}}$ is the sum of root spaces in $\tilde{A}_{V}$ (on which $T_{\mathfrak{L}}$ acts non-trivially by conjugation) and $\mathfrak{n}_{\mathfrak{L}}^{r}=\mathfrak{n}_{\mathfrak{L}}\cap \mathfrak{P}_{\mathfrak{L}}^r$, we then extend $\tilde{\xi}_\mathfrak{L}$ trivially to a character
$\tilde{\theta}_\mathfrak{L}$ on $U_{\mathfrak{L}}^{e_\mathfrak{L} d+}$ trivial on $U_{\mathfrak{L}}^{e_\mathfrak{L} 2d+}$,
which agrees with $\tilde{\xi}_\mathfrak{L}$ on $U_{\mathfrak{L}}^{e_\mathfrak{L} d+}\cap \tilde{T}_\mathfrak{L}=\tilde{T}_\mathfrak{L}^{d+}$. Finally, we extend $\tilde{\theta}_\mathfrak{L}$ to a character, still denoted by $\tilde{\theta}_\mathfrak{L}$, of $\tilde{H}_\mathfrak{L}^1=\tilde{T}_\mathfrak{L}^1U^{e_\mathfrak{L} d+}_{\mathfrak{L}}$.

For $\mathfrak{L}=\mathfrak{M}_{i_\circ,x}^w$, where $w=y$ or $z$, we construct similarly a character $\tilde{{\theta}}_{\mathfrak{M}_{i_\circ,x}^w}=\tilde{{\theta}}_{\mathfrak{M}_{i_\circ,x}^w,\tilde{\xi}}$ of $\tilde{H}^1_{\mathfrak{M}_{i_\circ,x}^w}$ such that $\tilde{{\theta}}_{\mathfrak{M}_{i_\circ,x}^w}|_{\tilde{T}^1_{\mathfrak{M}_{i_\circ,x}^w}}
=\tilde{\xi}_{\mathfrak{M}_{i_\circ,x}^w}|_{\tilde{T}^1_{\mathfrak{M}_{i_\circ,x}^w}}$. The process is similar to the one above, but we need to first extend the character $\tilde{\xi}_{\mathfrak{M}_{i_\circ,x}^w}|_{\tilde{T}^1_{\mathfrak{M}_{i_\circ,x}^w}}$ to
$$\tilde{\xi}':=\left(\tilde{\xi}_{i_\circ}|_{U^{1}_{E_{i_\circ}}}\circ\mathrm{det}_{i_\circ}\right)\boxtimes_{i\in I-\{i_\circ\}}\tilde{\xi}_i|_{U^{1}_{E_i}}$$
on
$$U^{1}_{{\mathfrak{M}^w_{i_\circ,x}},E}=U^{1}_{{\mathfrak{M}^w_{\pm i_\circ}},E_{i_\circ}}\times \prod_{i\in I-\{i_\circ\}}U^{1}_{E_i},$$
where $\mathrm{det}_{i_\circ}$ is the determinant map on $\mathrm{GL}_{3}(E_{i_\circ})$. We then extend $\tilde{\xi}'|_{U_{\mathfrak{M}_{\pm i_\circ}^w}^{ d+}\tilde{T}_{\mathfrak{M}_{i_\circ,x}^w}^{ d+}}$ to a character $\tilde{\theta}'$ on to $U_{\mathfrak{M}_{i_\circ,x}^w}^{ d+}$ using the root-space decomposition
$$\mathfrak{P}_{\mathfrak{M}_{i_\circ,x}^w}^{ d+}/\mathfrak{P}_{\mathfrak{M}_{i_\circ,x}^w}^{ 2d+}\cong  \left(\mathfrak{P}_{\mathfrak{M}_{\pm i_\circ}^w,E_{i_\circ}}^{ d+}/\mathfrak{P}_{\mathfrak{M}_{\pm i_\circ}^w,E_{i_\circ}}^{ 2d+}\bigoplus_{i\neq i_\circ}\mathfrak{p}_{E_i}^{d+}/\mathfrak{p}_{E_i}^{2d+}\right)\oplus \mathfrak{n}_{\mathfrak{M}_{i_\circ,x}^w}^{ d+}/ \mathfrak{n}_{\mathfrak{M}_{i_\circ,x}^w}^{ 2d+}$$
where $\mathfrak{n}_{\mathfrak{M}_{i_\circ,x}^w}$ is the nilpotent radical of the parabolic sub-algebra in $\tilde{A}_W$ whose Levi-component is $\tilde{A}_{V_{\pm i_\circ},E}\oplus_{i\neq i_\circ}E_i$, and $\mathfrak{n}_{\mathfrak{M}_{i_\circ,x}^w}^r=\mathfrak{n}_{\mathfrak{M}_{i_\circ,x}^w}\cap \mathfrak{P}_{\mathfrak{M}_{i_\circ,x}^w}^{ r}$. Finally, we multiply $\tilde{\xi}'$ and $\tilde{\theta}'$ to obtain our desired character $\tilde{\theta}$.

Finally, let $\mathfrak{L}$ be one of the lattice sequences in (\ref{one-of-the-lattices}). Since  $(\tilde{H}^1_\mathfrak{L})^\sigma=H^1_\mathfrak{L}$ , we define
  $$\theta_{\mathfrak{L}}=(\tilde{\theta}_{\mathfrak{L}}|_{H_{\mathfrak{L}}^1})^{1/2}.$$
   Note that the $1/2$-power on characters is well-defined. It is because $H^1_\mathfrak{L}$ is a $p$-group and $p$ is odd, so that $x\mapsto x^2$ is a bijection on $H^1_\mathfrak{L}$, and $\theta\mapsto \theta^2$ is a bijection on the set of characters (also semi-simple characters and skew semi-simple characters) of ${H}^1_\mathfrak{L}$.

\begin{prop}
$\theta_{\mathfrak{L}}$ is a skew semi-simple character of $H^1_\mathfrak{L}$ in the sense of \cite[Definition 3.13]{stevens-ss-char}.
\end{prop}
\proof

For each component $\tilde{\xi}_i$, let $b_i\in  \mathfrak{p}_{E_i}^{-2d}\mod \mathfrak{p}_{E_i}^{-d}$ be the element defined as in (\ref{element-b}), and $\beta_i$ be an element in $\mathfrak{p}_{E_i}^{-2d}$ approximating $b_i$ as in (\ref{b-approx}). We will show that the character $\tilde{\theta}_\mathfrak{L}$ is based on a skew semi-simple stratum of the form
$${\mathfrak{s}}=[\mathfrak{L},e_\mathfrak{L} d,0,\beta].$$
Here $\beta=\sum_{i\in I}\beta_i$ belongs to $\mathfrak{P}_{\mathfrak{L},F/\Fo}^{-2d}$ such that $\mathfrak{L}$ is $\mathcal{I}_{i_\circ,x}(E^\times)$-invariant, where $\mathcal{I}_{i_\circ,x}(E^\times)$ is defined in (\ref{image-of-E}). (As a remark, the $d$-regularity implies that each $[\mathfrak{L}_i,d,0,\beta_i]$ is a minimal stratum over $F$.)

We now check the two conditions in \cite[Definition 3.13]{stevens-ss-char}.
\begin{enumerate}[(i)]
  \item
  When  $i\neq i_\circ$, or when $i= i_\circ$ and $\mathfrak{L}=\Lambda_x$, the character $\tilde{\theta}_\mathfrak{L}|_{\tilde{H}^1_\mathfrak{L}\cap \tilde{G}_{V_i}}$ is equal to $\tilde{\theta}_{\Lambda_i}$ by construction. When $i= i_\circ$ and $\mathfrak{L}=\Lambda_{i_\circ,x}$ or $\mathfrak{M}_{i_\circ,x}^w$, then $\tilde{G}_{V_{\pm i_\circ}}=\mathrm{GL}_{3n_\circ}(F)$ and $$\tilde{H}^1_{\Lambda_{i_\circ,x}}\cap \tilde{G}_{V_{\pm i_\circ}}=\tilde{H}^1_{\Lambda_{\pm i_\circ}}=(\tilde{T}^1_{i_\circ})^3 U^{3d+}_{\Lambda_{\pm i_\circ}},$$
   and
   $$\tilde{H}^1_{\mathfrak{M}_{i_\circ,x}^w}\cap \tilde{G}_{V_{\pm i_\circ}}=\tilde{H}^1_{\mathfrak{M}_{\pm i_\circ}^w}=U^{1}_{\mathfrak{M}_{\pm i_\circ}^w,E_{i_\circ}}U^{3d+}_{\mathfrak{M}_{\pm i_\circ}^w}.$$
   Denote $\mathfrak{L}_\pm=\Lambda_{\pm i_\circ}$ or $\mathfrak{M}_{\pm i_\circ}^w$, then in both cases we have
      $$\tilde{\theta}_{\mathfrak{M}_{i_\circ,x}^w}|_{\tilde{H}^1_{\mathfrak{L}}\cap \tilde{G}_{V_{\pm i_\circ}}}
      =\tilde{\xi}_{i_\circ}|_{U^1_{E_{i_\circ}}}\circ\mathrm{det}_{i_\circ}|_{\tilde{H}^1_{\mathfrak{L}_{\pm}}},$$
      which is a semi-simple character based on $[\mathfrak{L}_\pm,e_\mathfrak{L}d,0,\beta_{i_\circ}]$.

 \item  We also check that, when $\mathfrak{L}=\Lambda_x$, then
      \begin{equation*}
        \begin{split}
        &\tilde{\theta}_{\Lambda_x}|_{\tilde{H}^{d+}_{\Lambda_x}}=\psi_F(tr_{\tilde{A}_V/F}(\beta X))=\psi_F(tr_{E/F}(\beta s^{\tilde{A}_V}_{E}X))
        \\
        =&\psi_F\left(\sum_{i\in I}tr_{\tilde{E_i}/F}(\beta_i s^{\tilde{A}_V}_{E_i}(X))\right)
      \end{split}
      \end{equation*}
      where $s^{\tilde{A}_V}_{E}$ and $s^{\tilde{A}_V}_{E_i}$ are suitable tame co-restrictions as in (\ref{tame-cores}) (see also \cite[(1.3)]{BK}), and when $\mathfrak{L}=\Lambda_{i_\circ,x}$ or $\mathfrak{M}_{i_\circ,x}^w$, then
      \begin{equation*}
        \begin{split}
        &\tilde{\theta}_{\mathfrak{L}}|_{\tilde{H}^{d+}_{\mathfrak{L}}}=\psi_F(tr_{\tilde{A}_W/F}(\beta X))=\psi_F(tr_{\tilde{B}/F}(\beta s^{\tilde{A}_W}_{\tilde{B}}X))
        \\
        =&\psi_F\left(tr_{\tilde{A}_{V_{\pm i_\circ}}/F}(\beta_{i_\circ} s^{\tilde{A}_W}_{\tilde{A}_{V_{\pm i_\circ}}}(X))+\sum_{i\neq i_\circ}tr_{\tilde{E_i}/F}(\beta_i s^{\tilde{A}_W}_{E_i}(X))\right)
      \end{split}
      \end{equation*}
            where $s^{\tilde{A}_W}_{\tilde{A}_{V_{\pm i_\circ}}}$ is a tame co-restriction and $\tilde{B}=\tilde{A}_{V_{\pm i_\circ}}\oplus_{i\neq i_\circ}E_i$.

\end{enumerate}

\qed


\begin{rmk}
Indeed, we can check that $\tilde{\theta}_\mathfrak{L}={\theta}_\mathfrak{L}\circ N_{\tilde{G}/G}$, where $ N_{\tilde{G}/G}$ is the norm map, a bijection from the set of stable $\sigma$-conjugacy classes of $\tilde{G}$ to the set of stable conjugacy classes of $G$ \cite[Propposition 3.11.1(c)]{Row-u3}, which restricts to a bijection from the set of $\sigma$-conjugacy classes of $\tilde{H}^1_\mathfrak{L}$ to the set of conjugacy classes of ${H^1_\mathfrak{L}}$. The proof is similar to \cite[Lemma 3.1(ii)]{Blasco-u3}. Although the paper concerns only $\mathrm{U}_n$ for $n=3$, the proof of this Lemma holds for arbitrary $n$ without significant modification.
\qed\end{rmk}


We provide some relations between semi-simple characters on compact subgroups associated to various lattice sequences. First note that $\tilde{H}^1_{\Lambda_{i_\circ}} \times {H}^1_{\Lambda_{x}}\hookrightarrow {H}^1_{\Lambda_{i_\circ,x}}$ and
$$\theta_{\Lambda_{i_\circ,x}}|_{\tilde{H}^1_{\Lambda_{i_\circ}} \times {H}^1_{\Lambda_{x}}}= \tilde{\theta}_{\Lambda_{i_\circ}}\boxtimes\theta_{\Lambda_x},$$
where $\tilde{\theta}_{\Lambda_{i_\circ}}$ is a simple character constructed in Section \ref{section Maximal types for GL} and $\theta_{\Lambda_x}$ is constructed just above.

Moreover, we recall the transfer property for semi-simple characters from \cite[Section 3.5]{stevens-ss-char}. Suppose that $\mathfrak{L}$ and $\mathfrak{L}'$ are two lattice sequences in $W$, and that  $\theta_{\mathfrak{L}}=\tilde{\theta}_{\mathfrak{L}}|_{H_\mathfrak{L}^1}$ and $\theta_{\mathfrak{L}'}=\tilde{\theta}_{\mathfrak{L}'}|_{H_{\mathfrak{L}'}^1}$ are two semi-simple characters associated respectively to the strata of the forms $[\mathfrak{L},n,0,\beta]$ and $[\mathfrak{L}',n',0,\beta']$ such that $\beta'=\beta$, then we say that $\theta_{\mathfrak{L}}$ is the transfer of $\theta_{\mathfrak{L}'}$ if
$$Z_{\tilde{G}_W}(\mathcal{I}_{i_\circ,x}(E^\times))\cap I_{\tilde{G}_{W}}(\tilde{\theta}_{\mathfrak{L}},\tilde{\theta}_{\mathfrak{L}'})\neq\emptyset,$$
where $\mathcal{I}_{i_\circ,x}(E^\times)$ is the image of $E^\times$ embedded in $\tilde{G}_W$, as defined in Section \ref{section orders}.

We now set $\mathfrak{L}={\Lambda_{i_\circ,x}}$ and $\mathfrak{L}'={\mathfrak{M}}={{\mathfrak{M}^w_{i_\circ,x}}}$, for $w=y$ or $z$, and consider the semi-simple characters $\theta_{\Lambda}$ and $\theta_{{\mathfrak{M}}}$, both constructed from $\tilde{\xi}|_{U_E^1}$ and associated to the strata $[{\Lambda},n_\Lambda,0,\beta]$ and $[{{\mathfrak{M}}},n_\mathfrak{M},0,\beta]$ respectively.
\begin{prop}
  The simple characters $\theta_{\Lambda}$ and $\theta_{{\mathfrak{M}}}$ are transfer of each other.
\end{prop}
\proof
It suffices to show that $1\in I_{\tilde{G}_{W}}(\tilde{\theta}_{\Lambda},\tilde{\theta}_{{\mathfrak{M}}})$, or equivalently
$$\tilde{\theta}_{\Lambda}|_{\tilde{H}^1_{\Lambda}\cap \tilde{H}^1_{{\mathfrak{M}}}}=\tilde{\theta}_{{\mathfrak{M}}}|_{\tilde{H}^1_{\Lambda}\cap \tilde{H}^1_{{\mathfrak{M}}}},$$
which is clear from the construction of both characters.
\qed

\subsubsection{Second step: Heisenberg representations}

Again let $\mathfrak{L}$ be one of the lattice sequences in (\ref{one-of-the-lattices}). The quotient $\mathfrak{V}_\mathfrak{L}={J}_\mathfrak{L}^1/{H}_\mathfrak{L}^1$ is naturally a $\mathbf{k}_{\Fo}$-vector space. Since ${\theta}_\mathfrak{L}$ is a semi-simple character, the space $\mathfrak{V}_\mathfrak{L}$ is equipped with a non-degenerate alternating form \cite[Proposition 3.5]{stevens-supercusp}
$$h_{{\theta}_\mathfrak{L}}:(1+x,1+y)\mapsto {\theta_\mathfrak{L}}([x,y]).$$
Using the theory of Heisenberg representation, there exists an irreducible representation ${\eta}_\mathfrak{L}=\eta_{\mathfrak{L},\tilde{\xi}}$ of ${J}_\mathfrak{L}^1$, unique up to isomorphism, whose restriction to ${H}_\mathfrak{L}^1$ is a multiple of ${\theta}_\mathfrak{L}$.

We again abbreviate $\Lambda=\Lambda_{i_\circ,x}$ and $\mathfrak{M}={\mathfrak{M}^w_{i_\circ,x}}$. Define
\begin{equation*}
{J}^1_{\Lambda,{\mathfrak{M}}}={U}^1_{\Lambda,E/\Eo}{J}^1_{\mathfrak{M}^w_{i_\circ,x}}
\text{ and }
  {J}_{\Lambda,{\mathfrak{M}}}={U}_{\Lambda,E/\Eo}{J}^1_{\mathfrak{M}^w_{i_\circ,x}}.
  \end{equation*}

  \begin{prop}\label{eta-Lambda-M}
      There exists a unique irreducible representation $\eta_{\Lambda,\mathfrak{M}}$ of ${J}^1_{\Lambda,{\mathfrak{M}}}$ such that
          \begin{enumerate}[(i)]
            \item $\eta_{\Lambda,\mathfrak{M}}|_{{J}^1_{{\mathfrak{M}}}}=\eta_{{\mathfrak{M}}}$;
            \item $\eta_{\Lambda,\mathfrak{M}}$ and $\eta_{\Lambda}$ induces equivalent irreducible representations of ${U}^1_{\Lambda,F/\Fo}$.
          \end{enumerate}
  \end{prop}
  \proof
  \cite[Proposition 3.7 and 3.8]{stevens-supercusp}.
  \qed

\subsubsection{Third step: Beta-extensions}

Now we specify $\mathfrak{L}=\Lambda_x$ and construct supercuspidal representations of $G_V$.

\begin{prop}\label{beta-extension-U}
\begin{enumerate}[(i)]
 \item There is a unique irreducible representation ${\kappa}_{x,\tilde{\xi}}$ of ${J}_{\Lambda_x}$ such that
  \begin{enumerate}[(a)]
    \item ${\kappa}_{x,\tilde{\xi}}|_{{J}_{\Lambda_x}^1}\cong {\eta}_{\Lambda_x}$;
    \item ${\kappa}_{x,\tilde{\xi}}$ is intertwined by $T_{x}$;
        \item the character $\det{\kappa}_{x,\tilde{\xi}}$ has a finite order of a $p$-power.
  \end{enumerate}
  \item All extensions satisfying (a) and (b) above are of the form ${\kappa}_{x,\tilde{\xi}}\cdot{\chi}$, where ${\chi}$ is a tamely ramified character of ${J}_{\Lambda_x}$, which means that it is inflated from ${J}_{\Lambda_x}/{J}_{\Lambda_x}^1\cong T_x/T_x^1$.
  \end{enumerate}
 \end{prop}
 \proof
This is from \cite[Theorem 4.1]{stevens-supercusp}.
\qed

Define \begin{equation*}
  \tilde{\xi}_{{\Lambda_x},\mathrm{wd}}=\prod_{i\in I}\tilde{\xi}_{i,\mathrm{wd}} \text{ and }\tilde{\xi}_{{\Lambda_x},\mathrm{tm}}=\tilde{\xi}_{\Lambda_x}\tilde{\xi}_{{\Lambda_x},\mathrm{wd}}^{-1}
\end{equation*}
as a character of $\prod_{i\in I}U^1_{E_i}$, where $\tilde{\xi}_{i,\mathrm{wd}}$ is a character defined using Proposition \ref{wild-part-tilde}. 
Also, define
 \begin{equation}\label{wild-tame-U}
  {\xi}_{{\Lambda_x},\mathrm{tm}}\text{ and }{\xi}_{{\Lambda_x},\mathrm{wd}}
\end{equation}
 such that  $\tilde{\xi}_{{\Lambda_x},\mathrm{tm}}={\xi}_{{\Lambda_x},\mathrm{tm}}\circ(1-{c})$ and  $\tilde{\xi}_{{\Lambda_x},\mathrm{wd}}={\xi}_{{\Lambda_x},\mathrm{wd}}\circ(1-{c})$. We put
$$\lambda_{{x},\tilde{\xi}}={\xi}_{{\Lambda_x},\mathrm{tm}}\cdot {\kappa}_{x,\tilde{\xi}}$$
and
$$\pi_{{x},\tilde{\xi}}=\cInd^{G_V}_{{J}_{\Lambda_x}}\lambda_{{x},\tilde{\xi}}.$$

\begin{prop}
  $\pi_{{x},\tilde{\xi}}$ is an irreducible supercuspidal representation of $G_V$.
\end{prop}
\proof
Note that $({{J}_{\Lambda_x}},\lambda_{{x},\tilde{\xi}})$ is a \emph{cuspidal type} as defined in \cite[Definition 4.3]{Stevens-Miya}, since  $\kappa_{{x},\tilde{\xi}}$ is a beta-extension and ${\xi}_{{\Lambda_x},\mathrm{tm}}$ is the inflation of an irreducible cuspidal representation, ,indeed a character, of $J_{\Lambda_x}/J_{\Lambda_x}^1 \cong U_{\Lambda_x,E/\Eo}/U^1_{\Lambda_x,E/\Eo}=T_x/T_x^1$ such that $U_{\Lambda_x,E/\Eo}=T_x$ is a maximal parahoric subgroup of $Z_{G_V}(T_x)=T_x$ itself. (For unitary groups, the conditions that $U_{\Lambda_x,E/\Eo}=U^\circ_{\Lambda_x,E/\Eo}$ and that $Z_{G_V}(T_x)$ has a compact center are always satisfied.) 
Moreover, the intertwining $I_{G_V}(\lambda_{{x},\tilde{\xi}})$ is ${{J}_{\Lambda_x}}$ by \cite[Proposition 6.18]{stevens-supercusp}, so that $\pi_{{x},\tilde{\xi}}$ is irreducible.
\qed


\begin{rmk}\label{compare-cuspidal-types}
  In the construction of cuspidal types, we used the method in \cite{reeder-pos-depth} for simple characters, and the method in \cite{stevens-supercusp} for beta-extensions. We now verify that the cuspidal type $\lambda_{{x},\tilde{\xi}}$ is equal to $\kappa_{\xi}$ in \cite[Section 3.2]{reeder-pos-depth} (with the notation of character $\chi$ in \emph{loc. cit.} replaced by our character $\xi$). First our $T_x$, ${{H}^1_{\Lambda_x}}$, ${{J}^1_{\Lambda_x}}$, and ${{J}_{\Lambda_x}}$ are $T^F$, $T^F_{0+}J_s$, $T^F_{0+}J_{s+}$, and $K_s=T^F\ltimes J_{s+}$ in \emph{loc. cit.} respectively (with $s=d$), and
  \begin{equation}\label{quotients}
    J_s/J_{s+}\cong {{J}^1_{\Lambda_x}}/{{H}^1_{\Lambda_x}}.
  \end{equation}
   Let $\tilde{\xi}_{{\Lambda_x},\mathrm{wd}}$ be defined as in (\ref{wild-tame-U}), and view it as a character on $T^F\ltimes J_{s}$ by natural extension. Let $\phi_{\xi}$ be the Weil representation of $T^F\ltimes J_{s}$ defined in \emph{loc. cit.}. The tensor product $\tilde{\xi}_{{\Lambda_x},\mathrm{wd}}\otimes \phi_{\xi}$ descends to a representation $\kappa$ on $K_s={{J}_{\Lambda_x}}$. We then check the conditions in Proposition \ref{beta-extension-U}.
  \begin{enumerate}[(i)]
    \item $\kappa|_{{J}^1_{\Lambda_x}}$ is a multiple of $\eta_{\xi}=\hat{\xi}$, the Heisenberg extension defined in \emph{loc. cit.}. (Note that we can view $\hat{\xi}$ as a representation of ${{J}^1_{\Lambda_x}}$ using (\ref{quotients}).)

        \item $\kappa$ is intertwined by $T_x$, which can be proved using \cite{adler-cusp}.

            \item $\det\kappa$ has a $p$-power order. This can be proved by tracing back the constructions as follows. We again use the notation in \cite[Section 3.2]{reeder-pos-depth}. Denote the quotient in (\ref{quotients}) by $\mathfrak{V}$, which is isomorphic to $\bar{\mathfrak{V}}\ltimes \mathbb{F}_p$ for an abelian $p$-group $\bar{\mathfrak{V}}$. Now $\phi_{\xi}$ is defined as a pull-back of the Weil-representation $\omega_\xi$ of $\mathrm{Sp}(\bar{\mathfrak{V}})\ltimes \mathfrak{V}$ using a special isomorphism
                $$W:T^F\ltimes J_{s}\rightarrow \mathrm{Sp}(\bar{\mathfrak{V}})\ltimes \mathfrak{V}$$
                (see \cite[Section 10]{Yu-cusp} for details). Therefore, for every $t\in T^F$ and $j\in J_{s}$,
                \begin{equation*}
                  \begin{split}
                    \det\kappa(tj)&=\det(\tilde{\xi}_{{\Lambda_x},\mathrm{wd}}\otimes \phi_{\xi})(t,j)
                                        \\
                    &=(\det\tilde{\xi}_{{\Lambda_x},\mathrm{wd}}(t,j))^{\dim \phi_\xi} \det\phi_{\xi}(t,j)
                    \\
                    &=(\det\tilde{\xi}_{{\Lambda_x},\mathrm{wd}}(t))^{\dim \phi_\xi} \det\omega_{\xi}(W(t,1)) \det\eta_{\xi}(W(1,j)).
                  \end{split}
                \end{equation*}
                In the product above, the first character is a p-power, the second one is trivial as a character of a symplectic group, and the third character is a p-power as defined on a p-group.
  \end{enumerate}
  Therefore, this $\kappa$ is our beta-extension $\kappa_{{x},\tilde{\xi}}$, and so
  $$ \lambda_{{x},\tilde{\xi}}={\xi}_{{\Lambda_x},\mathrm{tm}}\cdot {\kappa}_{x,\tilde{\xi}}=\tilde{\xi}_{{\Lambda_x}}\otimes \phi_{\xi}$$
  which is $\kappa_{\xi}$ in \emph{loc. cit.}.

\qed\end{rmk}

We now prove a result similar to Proposition \ref{repres-determines-char-GL}, that the isomoprhism class of $\pi_{x,\tilde{\xi}}$ depends only on the equivalence class on the pair $(T_x,\tilde{\xi}_x)$, where the equivalence relation is defined in Proposition \ref{equiv-relation-U} below.

\begin{lem}\label{U-equiv}
The following are equivalent.
\begin{enumerate}[(i)]
  \item $\pi_{x,\tilde{\xi}}\cong \pi_{x',\tilde{\xi}'}$. \label{U-equiv-repres}
  \item There exists $g\in G_V$ such that  \label{U-equiv-types} $\mathrm{Ad}(g)({J}_{\Lambda_x},\lambda_{{x},\tilde{\xi}})=({J}_{\Lambda_{x'}},\lambda_{{x'},\tilde{\xi}'})$.
      \item There exists $g\in G_V$ such that $\mathrm{Ad}(g)({T}_{x},{\xi}_{x})=({T}_{{x'}},{\xi}'_{x'})$. \label{U-equiv-char}

\end{enumerate}
\end{lem}

\proof
After we identified in Remark \ref{compare-cuspidal-types} the constructions of cuspidal types in \cite{stevens-supercusp} and \cite{reeder-pos-depth}, where the latter is based on \cite{adler-cusp}, \cite{Yu-cusp}, we can show that (\ref{U-equiv-repres})$\Leftrightarrow$(\ref{U-equiv-types}) is given by \cite[Theorem 6.7]{HM}. To prove (\ref{U-equiv-types})$\Leftrightarrow$(\ref{U-equiv-char}), we make use of the results in \cite{DR}, and so we switch all notations to those in \emph{loc. cit.} below.
\\\\
To prove (\ref{U-equiv-char})$\Rightarrow$(\ref{U-equiv-types}), given $g(S_\mu,\theta_\mu)=(S_\lambda,\theta_\lambda)$ for some $g\in G^{F_u}$, then we claim that $\mathrm{Ad}(g)\varkappa_\mu{\cong }  \varkappa_\lambda$. From Lemma 6.9.1 of \emph{loc. cit.}, we have $q_\lambda^{-1}gq_\mu\in T$ where $q_\lambda^{-1}gq_\mu=t_\nu\in X=T/{}^0T$. By Lemma 4.5.2(3) of \emph{loc. cit.}, the element
$$g':=p_\lambda t_\nu p_\mu^{-1}=p_\lambda(q_\lambda^{-1}gq_\mu) p_\mu^{-1}=m_\lambda^{-1}gm_\mu$$
in the affine Weyl group $W=N/{}^0T$ transforms $\kappa_\mu$ to $\kappa_\lambda$, and therefore $g$ transforms $\varkappa_\mu$ to $\varkappa_\lambda$.
\\\\
To prove the converse (\ref{U-equiv-types})$\Rightarrow$(\ref{U-equiv-char}), if $\mathrm{Ad}(g)\varkappa_\mu{\cong }  \varkappa_\lambda$ for some $g\in G^{F_u}$, then form $g'=m_\lambda^{-1}gm_\mu$ which transforms $\kappa_\mu$ to $\kappa_\lambda$. We show that
\begin{enumerate}[(i)]
  \item $g'*u_\mu=u_\lambda$ where $m_\mu*u_\mu=m_\lambda*u_\lambda=u$, because
  $$g'*u_\mu=(m_\lambda^{-1}gm_\mu)*u_\mu=(m_\lambda^{-1}g)*u=m_\lambda^{-1}*u$$
  since $g\in G^{F_u}$. The right side is just $u_\lambda$.
  \item $g'J_\mu=J_\lambda$, where $J_\mu$ is the facet in the closure of the alcove $C_\mu$ and similar for $J_\lambda$ and $C_\lambda$. It is because
      $$g'J_\mu=(m_\lambda^{-1}gm_\mu)J_\mu=(m_\lambda^{-1}g)I_\mu=m_\lambda^{-1}I_\lambda=J_\lambda.$$
  \end{enumerate}
Therefore, by Lemmas 4.5.2 and 9.6.1 of \emph{loc. cit.}, we have $g(S_\mu,\theta_\mu)=(S_\lambda,\theta_\lambda)$.
\qed

\begin{prop}\label{equiv-relation-U}
   $\pi_{x,\tilde{\xi}}\cong \pi_{x',\tilde{\xi}'}$ if and only if
   \begin{enumerate}[(i)]
     \item $x= x'\in \mathcal{D}$, i.e., the embeddings $\mathcal{I}_x$ and $\mathcal{I}_{x'}$ are conjugate under $G_V$, and,
     \item if we assume that $T_x=T_{x'}$, there exists $g\in N_{G_V}(T_x)$ such that $\mathrm{Ad}(g)({\xi}_{x})={\xi}'_{x'}$.
   \end{enumerate}
   \end{prop}
\proof
This is just (\ref{U-equiv-repres})$\Leftrightarrow$(\ref{U-equiv-char}) above.
\qed

\subsection{Covering pairs}\label{section cover}

In this section, we fix $i_\circ\in I$, $x\in \mathcal{D}$, $G=G_W$, and a skew character $\tilde{\xi}$ of $\tilde{T}_x$. We write $\Lambda=\Lambda_{i_\circ,x}$  and $\mathfrak{M}=\mathfrak{M}^w_{i_\circ,x}$ for $w=y$ or $z$,
and recall that $\Lambda=\Lambda_{i_\circ,-}\oplus \Lambda_{x}\oplus\Lambda_{i_\circ,+}$.
Hence $\mathfrak{A}_{\Lambda,E}$ is a minimal conjugate-self-dual $\mathfrak{o}_E$-order contained in the maximal conjugate-self-dual $\mathfrak{o}_E$-order $\mathfrak{A}_{\mathfrak{M},E}$ (notation as in Section \ref{section orders}).

\subsubsection{Compatible beta-extensions}

In Proposition \ref{eta-Lambda-M}, we defined a Heisenberg representation $\eta_{\Lambda,\mathfrak{M}}$ of the subgroup ${J}^1_{\Lambda,{\mathfrak{M}}}={U}^1_{\Lambda,E/\Eo}{J}^1_{\mathfrak{M},F/\Fo}$.

\begin{defn}[\text{\cite[Theorem 4.1]{stevens-supercusp}}]\label{beta-extension-maximal-case}
       We call an irreducible representation $\kappa_{\mathfrak{M}}$  of $J_{\mathfrak{M}}$ a beta-extension of $\eta_{\mathfrak{M}}$ of $J^1_{\mathfrak{M}}$ if it is an extension of $\eta_{\Lambda,\mathfrak{M}}$ of ${J}^1_{\Lambda,\mathfrak{M}}$ for arbitrary conjugate-self-dual $\mathfrak{o}_{E}$-lattice sequence $\Lambda$ such that $\mathfrak{A}_{\Lambda,{{E}}}$ is a minimal conjugate-self-dual ${\mathfrak{o}_{E}}$-order contained in  $\mathfrak{A}_{\mathfrak{M},{{E}}}$.
\end{defn}
This definition is independent of the minimal lattice chosen \cite[Remark after Proposition 3.7]{stevens-supercusp}. Also, the minimal orders contained in a fixed maximal order are conjugate by $U_{\mathfrak{M},E}$.

The following Proposition is similar to Proposition \ref{beta-extension-U}.
\begin{prop}\label{beta-extension-max-min}
\begin{enumerate}[(i)]
 \item There is a unique beta-extension ${\kappa}_{{\mathfrak{M}},\tilde{\xi}}$ of ${J}_{\mathfrak{M}}$ such that
  \begin{enumerate}[(a)]
    \item ${\kappa}_{{\mathfrak{M}},\tilde{\xi}}|_{J^1_{\Lambda,\mathfrak{M}}}\cong \eta_{\Lambda,\mathfrak{M}}$;
    \item ${\kappa}_{{\mathfrak{M}},\tilde{\xi}}$ is intertwined by $U_{\mathfrak{M},E/\Eo}$;
        \item the character $\det{\kappa}_{{\mathfrak{M}},\tilde{\xi}}$ has a finite order of a $p$-power.
  \end{enumerate}
  \item All beta-extensions satisfy (a) and (b) above. They are of the form ${\kappa}_{\mathfrak{M}}={\kappa}_{{\mathfrak{M}},\tilde{\xi}}\cdot{\chi}$ for a character ${\chi}$ of ${J}_{\mathfrak{M}}$ inflated from ${J}_{\mathfrak{M}}/{J}_{\mathfrak{M}}^1$. \label{beta-extension-max-min-char}
  \end{enumerate}
 \end{prop}
\proof
This is from \cite[Theorem 4.1]{stevens-supercusp} and the following fact. Let $\mathbf{F}/\mathbf{F}_\bullet$ be any quadratic field extension, then
\begin{equation*}
  \mathrm{SU}_n(\mathbf{F}/\mathbf{F}_\bullet)\text{ is generated by unipotent elements in } \mathrm{U}_n(\mathbf{F}/\mathbf{F}_\bullet).
\end{equation*}
    This appears in \cite[p.49,(4)]{Dieudonne}, which states that (except when $n=3$ and $\mathbf{F}/\mathbf{F}_\bullet=\mathbb{F}_4/\mathbb{F}_2$) $\mathrm{SU}_n(\mathbf{F}/\mathbf{F}_\bullet)$ is equal to the group generated by its unitary transvections, which are unitary unipotent elements in the modern language. \qed

Since ${J}_{\mathfrak{M}}/{J}_{\mathfrak{M}}^1$ is a product finite unitary groups, a character appearing in (\ref{beta-extension-max-min-char}) factors through the determinant map of the corresponding finite groups.

\begin{prop}\label{bijection of extensions}
\begin{enumerate}[(i)]
  \item Suppose that
  \begin{itemize}
    \item $\theta_\Lambda$ and $\theta_\mathfrak{M}$ are simple characters of $H^1_\Lambda$ and $H^1_\mathfrak{M}$, transfer of each other;
        \item $\eta_\Lambda$ and $\eta_\mathfrak{M}$ are the unique irreducible representations containing $\theta_\Lambda$ and $\theta_\mathfrak{M}$ respectively.
  \end{itemize}
  Then there is a canonical bijection
  \begin{equation}\label{bijection of extensions eq}
    \left\{\begin{matrix}
    \text{extensions }(\kappa_\Lambda,J_\Lambda)
    \\
    \text{of }(\eta_\Lambda,J^1_\Lambda)
  \end{matrix}\right\}\xrightarrow{\mathfrak{B}_{\Lambda,{\mathfrak{M}}}}\left\{\begin{matrix}
    \text{extensions }(\kappa_{\Lambda,{\mathfrak{M}}},J_{\Lambda,{\mathfrak{M}}})
    \\
    \text{of }(\eta_{{\mathfrak{M}}},J^1_{{\mathfrak{M}}})
  \end{matrix}\right\}.
  \end{equation} \label{bijection of extensions def}
  \item If $\mathfrak{A}_{\Lambda}\subseteq \mathfrak{A}_{\mathfrak{M}}$, then this bijection can be described as follows: if $\kappa_\Lambda$ is an extension on the left side of (\ref{bijection of extensions eq}) and  $\kappa_{\Lambda,{\mathfrak{M}}}=\mathfrak{B}_{\Lambda,{\mathfrak{M}}}(\kappa_\Lambda)$, then the induced representations
  $$\Ind_{J_\Lambda}^{U_{\Lambda,{E/\Eo}}U^1_{\Lambda}}\kappa_\Lambda\text{ and } \Ind_{J_{\Lambda,{\mathfrak{M}}}}^{U_{\Lambda,{E/\Eo}}U^1_{\Lambda}}\kappa_{\Lambda,{\mathfrak{M}}}$$
  are equivalent and irreducible.
  \item Moreover, this bijection is compatible with the twisting by characters of $J_\Lambda/J^1_\Lambda\cong J_{\Lambda,{\mathfrak{M}}}/J^1_{\Lambda,{\mathfrak{M}}}\cong {T}_{\Lambda}/T^1_\Lambda$. \label{bijection of extensions char}
\end{enumerate}
\end{prop}
\proof
See \cite[Lemma 4.3]{stevens-supercusp} and \cite[Lemma 1.5(ii)]{Blondel-Weil}
\qed

\begin{defn}\label{definition-compatible-beta}
  \begin{enumerate}[(i)]
    \item (\text{\cite[Definition 4.5]{stevens-supercusp}}) Suppose that $\kappa_\mathfrak{M}$ is a beta-extension of $\eta_\mathfrak{M}$ in Definition \ref{beta-extension-maximal-case}. The beta-extension of $\eta_\Lambda$ to $J_\Lambda$, relative to $\mathfrak{M}$ and compatible with $\kappa_\mathfrak{M}$, is the unique representation $\kappa_\Lambda^\mathfrak{M}$ of $J_\Lambda$ such that $\mathfrak{B}_{\Lambda,{\mathfrak{M}}}(\kappa_\Lambda^\mathfrak{M})=\Res_{J_{\Lambda,\mathfrak{M}}}^{J_{\mathfrak{M}}}\kappa_\mathfrak{M}$. Hence if $\mathfrak{A}_\Lambda\subseteq \mathfrak{A}_\mathfrak{M}$, then
        $$\Ind_{J_\Lambda}^{U_{\Lambda,{E/\Eo}}U^1_\Lambda}\kappa_\Lambda^\mathfrak{M}\cong  \Ind_{J_{\Lambda,\mathfrak{M}}}^{U_{\Lambda,{E/\Eo}}U^1_\Lambda}\Res_{J_{\Lambda,\mathfrak{M}}}^{J_{\mathfrak{M}}}\kappa_\mathfrak{M}.$$
        \label{definition-compatible-beta-extension}

        \item If $\kappa_{\mathfrak{M},\tilde{\xi}}$ is the beta-extension fixed in Proposition \ref{beta-extension-max-min}, then we define $\kappa^{\mathfrak{M}}_{\Lambda,\tilde{\xi}}$ as the one compatible with $\kappa_{\mathfrak{M},\tilde{\xi}}$. \label{definition-compatible-beta-extension-nc}

  \end{enumerate}
\end{defn}

If $\kappa^0_{\Lambda,\tilde{\xi}}$ is the extension of $\eta_{\Lambda}$ such that $\det\kappa^0_{\Lambda,\tilde{\xi}}$ has a p-power order, then by Proposition \ref{bijection of extensions}(\ref{bijection of extensions char}) there exists a character $\mu^w_{\tilde{\xi}}$ of $J_\Lambda$, inflated from a character of $J_\Lambda/J^1_\Lambda\cong {T}_{\Lambda}/T^1_\Lambda$, such that
\begin{equation}\label{amending appear}
  \kappa^0_{\Lambda,\tilde{\xi}}\cdot\mu^w_{\tilde{\xi}}= \kappa^{\mathfrak{M}^w}_{\Lambda,\tilde{\xi}}.
\end{equation}
We will extensively study this character in Section \ref{section amending char}.

\subsubsection{Jacquet functor of types}

Let $P$ be the parabolic subgroup of $G_W$ stabilizing the flag
$ V_{i_\circ,-} \subset V_{i_\circ,-}\oplus V\subset W$, such that its Levi subgroup $M$ is isomorphic to $\tilde{G}_{V_{i_\circ}}\times G_V$. Let $U$　be the unipotent radical of $P$.

\begin{rmk}
  Before we proceed, it is necessary to check that the decomposition $W=V_{i_\circ,-}\oplus V\oplus V_{i_\circ,+}$ is exactly subordinate to the semi-simple stratum $[\Lambda,3d,0,\beta]$. In the notation of \cite[Section 5]{stevens-supercusp}, we have \begin{equation*}
    \begin{split}
     & W^{(-1)}=V_{i_\circ,-},\, W^{(0)}=V,\,W^{(1)}=V_{i_\circ,+},
      \\
    &V^i=V_i\text{ if }i\neq i_\circ\text{, and }V^{i_\circ}=  V_{i_\circ,-}\oplus V_{i_\circ}\oplus V_{i_\circ,+}\text{ if }i=i_\circ.
    \end{split}
  \end{equation*}
   \begin{enumerate}[(i)]
    \item It is easy to see that
    $$W^{(-1)}=V_{i_\circ,-}=\oplus_{i\in I}W^{(-1)}\cap V^i,$$
    where $V_{i_\circ,-}\cap V^i=0$ if $i\neq i_\circ$ and $V_{i_\circ,-}\cap V^i=V_{i_\circ,-}$ if $i=i_\circ$. Similar statement holds if we replace $W^{(-1)}$ and $V_{i_\circ,-}$ by $W^{(1)}$ and $V_{i_\circ,+}$ respectively. Moreover
    $$W^{(0)}=V=\oplus_{i\in I}(W^{(0)}\cap V^i)=\oplus_{i\in I} V_i.$$
    We can then check that
    $$V^i=(W^{(-1)}\cap V^i)\oplus (W^{(0)}\cap V^i)\oplus (W^{(1)}\cap V^i)$$
    for all $i\in I$.

    \item We then check that $\Lambda$ is subordinate to the stratum, which means that $$\Lambda(r)=(\Lambda(r)\cap V_{i_\circ,-})\oplus (\Lambda(r)\cap V)\oplus (\Lambda(r)\cap V_{i_\circ,+}),$$
        where $\Lambda(r)\cap V_{i_\circ,-}=\Lambda_{i_\circ}((r-1)/3),\, \Lambda_{i_\circ,x}(r)\cap V=\Lambda_{x}(r/3)$, and $\Lambda(r)\cap V_{i_\circ,+}=\Lambda_{i_\circ}((r+1)/3)$. As a remark, a similar statement holds for $\mathfrak{M}$, because when it is viewed as a function from $\mathbb{Z}$ to the set of lattices, its image is lying in that of $\Lambda$.

  \item We then check that $\Lambda$ is properly subordinate.
        \begin{itemize}
            \item When $i\neq i_\circ$, then
          $$\Lambda(r)\cap W^{(0)}\cap V^i\supsetneq\Lambda(r+1)\cap W^{(0)}\cap V^i$$
          when $r=3k$, and they are equal otherwise;
          $$\Lambda(r)\cap W^{(j)}\cap V^i=\Lambda(r+1)\cap W^{(j)}\cap V^i=0$$
          since $W^{(j)}\cap V^i=0$ for $j\neq 0$.
          \item  When $i= i_\circ$, then
          $$\Lambda(r)\cap W^{(j)}\cap V^{i_\circ}\supsetneq\Lambda(r+1)\cap W^{(j)}\cap V^{i_\circ}$$
          when $(r,j)=(3k,0),(6k+1,-)$ or $(6k-1,+)$, and they are equal otherwise.
        \end{itemize}
        In all cases, we see that for every $r\in \mathbb{Z}$ and $i\in I$, there is at most one $j=-,0,+$ such that
          $$\Lambda(r)\cap W^{(j)}\cap V^i\supsetneq\Lambda(r+1)\cap W^{(j)}\cap V^i.$$
        As a remark, it is clear that $\mathfrak{M}$ is never properly subordinate, because
          $$\mathfrak{M}^y(0)\cap W^{(j)}\cap V^{i_\circ}\supsetneq\mathfrak{M}^y(1)\cap W^{(j)}\cap V^{i_\circ}$$
          for all $j$, and
                    $$\mathfrak{M}^z(0)\cap W^{(j)}\cap V^{i_\circ}\supsetneq\mathfrak{M}^z(1)\cap W^{(j)}\cap V^{i_\circ}$$
                    for $j=-,+$.
 \item Finally, we check that $\Lambda$ is exactly subordinate, using the definition on \cite[p.331]{stevens-supercusp}.
 \begin{enumerate}
   \item Now $\mathfrak{a}_0(\Lambda^{(0)})\cap B^{(0)}$ in \emph{loc. cit.} is
   $$\mathfrak{A}_{\Lambda_x,E}:=\mathfrak{A}_{\Lambda_x}\cap Z_{\tilde{A}_V}((\oplus_{i\in I}E_i))\cong\oplus_{i\in I}\mathfrak{o}_{E_i},$$
   which is clearly a maximal conjugate-self-dual $(\oplus_{i\in I}\mathfrak{o}_{E_i})$-order in $Z_{\tilde{A}_V}(\oplus_{i\in I}E_i)\cong \oplus_{i\in I}E_i$,
       \item If $j\neq 0$, then $W^{(-1)}$ and $W^{(1)}$ are contained in $V^{i_\circ}= V_{i_\circ,-}\oplus V_{i_\circ}\oplus V_{i_\circ,+}$ and $\mathfrak{a}_0(\Lambda^{(\pm)})\cap B^{(\pm)}$ in \emph{loc. cit.} is $$\mathfrak{A}_{\Lambda_{i_\circ},E}=\mathfrak{A}_{\Lambda_{i_\circ}}\cap Z_{\tilde{A}_{V_{i_\circ}}}(E_{i_\circ})=\mathfrak{o}_{E_{i_\circ}},$$
            again a maximal $\mathfrak{o}_{E_{i_\circ}}$-order in $ Z_{\tilde{A}_{V_{i_\circ}}}(E_{i_\circ})=E_{i_\circ}$.
 \end{enumerate}
  \end{enumerate}
\qed\end{rmk}

Given the simple characters $\theta_{\Lambda}$ of $H^1_{\Lambda}$ and $\theta_{\mathfrak{M}}$ of $H^1_{\mathfrak{M}}$ transfer of each other, let $\kappa_{\mathfrak{M}}$ be the beta-extension containing a multiple of $\theta_{\mathfrak{M}}$ and $\kappa^{\mathfrak{M}}_{\Lambda}$ be the one compatible with $\kappa_{\mathfrak{M}}$.

We define
\begin{itemize}
  \item a simple character $\theta_{P}=\theta_{\Lambda,P}$ of $H^1_{P}=H^1_{\Lambda,P}=H^1_{\Lambda}(J^1_{\Lambda}\cap U)$ by extending $\theta_{\Lambda}$ trivially to $J^1_{\Lambda}\cap U$;
      \item an irreducible representation $\eta_{P}=\eta_{\Lambda,P}$ of $J^1_{P}=J^1_{\Lambda,P}=H^1_{\Lambda}(J^1_{\Lambda}\cap P)$ extending $\theta_{P}$, then we can check \cite[Lemma 5.12]{stevens-supercusp} that
          $$\eta_{\Lambda}\cong \Ind_{J^1_{P}}^{J^1_{\Lambda}}\eta_P;$$
        \item an irreducible representation $\kappa^w_{P}=\kappa^{\mathfrak{M}^w}_{\Lambda,P}$ of $J_{P}=J_{\Lambda,P}=H^1_{\Lambda}(J_{\Lambda}\cap P)$ by restricting $\kappa^{\mathfrak{M}^w}_{\Lambda}$ to the space of $(J^1_{\Lambda}\cap U)$-fixed vectors of  $\eta_{\Lambda}$. Then we can check \cite[Proposition 5.13]{stevens-supercusp} that
            $\kappa^w_{P}$ extends $\eta_{P}$ and
            $$\kappa_{\Lambda}^{\mathfrak{M}^w}\cong \Ind_{J_{P}}^{J_{\Lambda}}\kappa_P.$$
\end{itemize}

The relation between $\kappa^y_{P}$ and $\kappa^z_{P}$ is given as follows.
\begin{prop}\label{transferring-character-maximal}
 There exists a tamely ramified character $\chi_{y}^{z,P}$ of $J_\Lambda$ (inflated from a character of $J_\Lambda/J^1_\Lambda\cong {T}_{\Lambda}/T^1_\Lambda$) such that
 $$\kappa_{\Lambda}^{\mathfrak{M}^z}\cong \Ind_{J_{P}}^{J_{\Lambda}}(\kappa^{\mathfrak{M}^y}_{\Lambda,P}\cdot \chi_{y}^{z,P}).$$
\end{prop}
\proof
This is part of \cite[Corollary 6.13]{stevens-supercusp}. The condition that both $ \mathfrak{A}_{\mathfrak{M}^y,E}$ and  $\mathfrak{A}_{\mathfrak{M}^z,E}$ are maximal conjugate-self-dual orders containing $\mathfrak{A}_{\Lambda,E}$ clearly holds in our situation.
\qed

Hence by definition, both $\chi_{y}^{y,P}$ and $\chi_{z}^{z,P}$ are trivial, and $\chi_{z}^{y,P}=(\chi_{y}^{z,P})^{-1}$. Moreover, an argument similar to \cite[p.26]{GKS} shows that $\chi_{y}^{z,P}$ is at most quadratic, and is therefore skew and tamely ramified.

By the Iwahori-decompositions of the groups $H^1_{P}$, $J^1_{P}$, and $J_{P}$, we define
\begin{equation*}
  \theta_{P}|_{H_\Lambda^1\cap M}=\theta_{\Lambda}|_{H_\Lambda^1\cap M}=\tilde{\theta}_{\Lambda_{i_\circ}}\boxtimes\theta_{\Lambda_{x}},
\end{equation*}
  and by the uniqueness of Heisenberg representations,
  \begin{equation*}
    \eta_{P}|_{J_\Lambda^1\cap M}=\tilde{\eta}_{\Lambda_{i_\circ}}\boxtimes\eta_{\Lambda_{x}}.
  \end{equation*}
  We now write
   \begin{equation*}
  \kappa^w_{P}|_{J_\Lambda\cap M}=\tilde{\kappa}^{w,P}_{i_\circ}\boxtimes\kappa^{w,P}_{x}.
  \end{equation*}
where $\kappa^{w,P}_{x}$ and $\tilde{\kappa}^{w,P}_{i_\circ}$ are extensions of $\eta_{\Lambda_{x}}$ and $\tilde{\eta}_{\Lambda_{i_\circ}}$ respectively.

\begin{prop}\label{restrictions-are-beta}
   $\kappa^{w,P}_{x}$ (resp. $\tilde{\kappa}^{w,P}_{i_\circ}$) is a beta-extension of $\eta_{\Lambda_x}$ (resp. $\tilde{\eta}_{\Lambda_{i_\circ}}$).
  \end{prop}
\proof
This is \cite[Proposition 6.3]{stevens-supercusp}. The condition that  $ \mathfrak{A}_{\Lambda_x,E}$ (resp. $ \mathfrak{A}_{\Lambda_{i_\circ},E}$) is a {maximal} conjugate-self-dual order clearly holds in our situation.
\qed

Suppose that  $\rho_{i_\circ,x}=\tilde{\rho}_{i_\circ}\boxtimes \rho_{x}$ is a character of $J_\Lambda$ inflated from a character of the quotient
$$J_{\Lambda}/J^1_{\Lambda}\cong J_{\Lambda,P}/J^1_{\Lambda,P}\cong T_{\Lambda}/T^1_{\Lambda}.$$
We define
$$\lambda^w_P=\kappa^w_P\cdot \rho_{i_\circ,x}$$
such that
$\lambda^w_{P}|_{J\cap M}=\tilde{\lambda}^{w,P}_{i_\circ}\boxtimes\lambda^{w,P}_{x}$, where $\tilde{\lambda}^{w,P}_{i_\circ}=\tilde{\kappa}^{w,P}_{i_\circ}\cdot \tilde{\rho}_{i_\circ}$ and $\lambda^{w,P}_{x}=\kappa^{w,P}_{x}\cdot \rho_{x}$. Moreover, by \cite[Lemma 6.1]{stevens-supercusp},
\begin{itemize}
  \item $\lambda^w_{P}$ is the restriction of $\lambda^{\mathfrak{M}^w}_{\Lambda}$ to the space of $(J^1\cap U)$-fixed vectors, and
  \item     $\lambda^{\mathfrak{M}^w}_{\Lambda}\cong \Ind_{J_{P}}^{J_{\Lambda}}\lambda^w_P.$
\end{itemize}

The representation $(J_P,\lambda^w_P)$ is a cover, a notion defined in \cite[Section 8]{BK-cover}, of
  $$(J_\Lambda\cap M,\lambda^w_P|_{J_\Lambda\cap M})=(\tilde{J}_{\Lambda_{i_\circ}}\times {J}_{\Lambda_{x}},\tilde{{\lambda}}^{w,P}_{i_\circ}\boxtimes {\lambda^{w,P}_{x}}),$$
  in the following sense. Let  $\tilde{\pi}^{w,P}_{i_\circ}\boxtimes \pi^{w,P}_{x}$  be a cuspidal representation of $M$ containing the maximal type $\tilde{{\lambda}}^{w,P}_{i_\circ}\boxtimes {\lambda^{w,P}_{x}}$. From \cite[Theorem 5.3]{Stevens-Miya} or \cite[Section 3.2]{Blondel-Weil}, there is an injective morphism of Hecke algebras
$$t_P: \mathcal{H}(M, \tilde{{\lambda}}^{w,P}_{i_\circ}\boxtimes {\lambda^{w,P}_{x}}) \rightarrow \mathcal{H}(G_W,\lambda^w_P)$$
 which induces the following commutative diagram,
\begin{equation*}
  \xymatrixcolsep{5pc}\xymatrix{
\mathcal{R}^{[\tilde{\pi}^{w,P}_{i_\circ}\boxtimes\pi^{w,P}_{x},M]}(G_W)  \ar[r]^{h_{G_W}}    &\mathrm{Mod-}\mathcal{H}(G_W,\lambda^w_P)
\\
\mathcal{R}^{[\tilde{\pi}^{w,P}_{i_\circ}\boxtimes\pi^{w,P}_{x},M]}(M) \ar[u]_{\Ind_P^{G_W}}\ar[r]^{h_M}  &\mathrm{Mod-}\mathcal{H}(M,\tilde{{\lambda}}^{w,P}_{i_\circ}\boxtimes {\lambda^{w,P}_{x}}),  \ar[u]_{(t_P)_*}
}
\end{equation*}
where
\begin{itemize}
\item $\mathrm{Mod-}\mathcal{H}(G_W,\lambda^w_P)$ is the category of right module over the Hecke algebra $\mathcal{H}(G_W,\lambda^w_P)$ (and similarly for $\mathrm{Mod-}\mathcal{H}(M,\tilde{{\lambda}}^{w,P}_{i_\circ}\boxtimes {\lambda^{w,P}_{x}})$;

    \item $[\tilde{\pi}^{w,P}_{i_\circ}\boxtimes\pi^{w,P}_{x},M]$ is the inertial class of $\tilde{\pi}^{w,P}_{i_\circ}\otimes \pi^{w,P}_{x}$, and $\mathcal{R}^{[\tilde{\pi}^{w,P}_{i_\circ}\boxtimes\pi^{w,P}_{x},M]}(G_W)$  (resp. $\mathcal{R}^{[\tilde{\pi}^{w,P}_{i_\circ}\boxtimes\pi^{w,P}_{x},M]}(M) $) is the Bernstein spectrum of representations of $G_W$ (resp. $M$) whose irreducible subquotients have cuspidal supports lying in $[\tilde{\pi}^{w,P}_{i_\circ}\boxtimes\pi^{w,P}_{x},M]$ ;

    \item $\Ind_P^{G_W}$ is the normalized parabolic induction;

  \item The functor $(t_P)_*$ associates to a $\mathcal{H}(M,\tilde{{\lambda}}^{w,P}_{i_\circ}\boxtimes {\lambda^{w,P}_{x}})$-module $X$ the right $\mathcal{H}(G_W,\lambda^w_P)$-module $\mathrm{Hom}_{\mathcal{H}(M,\tilde{{\lambda}}^{w,P}_{i_\circ}\boxtimes {\lambda^{w,P}_{x}})}( \mathcal{H}(G_W,\lambda^w_P), X)$.

      \item The functor $h_M$ associates to a representation $\tau$ of $M$ the module $\mathrm{Hom}_{\tilde{J}_{\Lambda_{i_\circ}}\times {J}_{\Lambda_{x}}}(\tilde{{\lambda}}^{w,P}_{i_\circ}\boxtimes {\lambda^{w,P}_{x}},\tau)$, and $h_{G_W}$ associates to a representation $\mathcal{T}$ of $G_W$ the module $\mathrm{Hom}_{J_P}({{\lambda}}^w_P,\mathcal{T})$. Both functors are equivalences of categories.
\end{itemize}

Suppose that $\tilde{\pi}_\circ$ lies in the inertial class of $\tilde{\pi}^{w,P}_{i_\circ}$ and $\pi=\pi^{w,P}_{x}$, then the reducibility of $\Ind_P^{G_W}(\tilde{\pi}_\circ\boxtimes \pi)$ can be studied by the $\mathcal{H}(G_W,\lambda^w_P)$-module $(t_P)_*h_M(\tilde{\pi}_\circ\boxtimes \pi)$. In the next section, we will study the structure of the above Hecke algebras.

\newpage\section{Hecke algebras}\label{section Hecke algebra}

\subsection{Structures of Hecke algebras}

This section summarizes the facts from \cite[Section 3.1]{Blondel-Weil}. Recall that $W=V_{i_\circ,-}\oplus V\oplus V_{i_\circ,+}$. We embed $\tilde{G}_{V_{i_\circ}}\cong \mathrm{GL}_F(V_{i_\circ})$ and the unitary group $G_V$ into a larger unitary group $G_W$, by the maps
$$\mathcal{I}^M_+:\tilde{G}_{V_{i_\circ}}\rightarrow G_W,\,g\mapsto ({}^\sigma{g},I_{V},g)\text{ and }\mathcal{I}^M_0:{G}_V\rightarrow G_W,\,h\mapsto (I_{V_{i_\circ}},h,I_{V_{i_\circ}}).$$
Then the Levi subgroup $M$ is equal to $\mathcal{I}^M(\tilde{G}_{V_{i_\circ}}\times {G}_V)=\mathcal{I}^M_+(\tilde{G}_{V_{i_\circ}})\times \mathcal{I}^M_0({G}_V)$. We take $P$ the parabolic subgroup of $G_W$ stabilizing the flag
$ V_{i_\circ,-} \subset V_{i_\circ,-}\oplus V\subset W$.

Choose a basis
$\mathcal{B}_+=\{\mathbf{e}_1,\dots,\mathbf{e}_{n_{i_\circ}}\}$ for $\Lambda_{i_\circ,+}(0)$
and another
$\mathcal{B}_-=\{\mathbf{e}_{-1},\dots,\mathbf{e}_{-n_{i_\circ}}\}$ for $\Lambda_{i_\circ,-}(1)$
such that the Hermitian form $h_W$ on $W$ is defined as in (\ref{even-Hermitian}). 
In terms of the ordered basis $\mathcal{B}_W=\mathcal{B}_-\sqcup\mathcal{B}_V\sqcup\mathcal{B}_+$, we define
$$s_{y}^{P}=\left[\begin{smallmatrix}
  &&I_{V_{i_\circ}}
  \\
  &I_V&
  \\
  I_{V_{i_\circ}}&&
\end{smallmatrix}\right]\text{ and }s_{z}^{P}=\left[\begin{smallmatrix}
  &&\varpi I_{V_{i_\circ}}
  \\
  &I_V&
  \\
  \varpi^{-1} I_{V_{i_\circ}}&&
\end{smallmatrix}\right].$$
 These elements are respectively denoted by $s_1^\varpi$ and $s_1$ in \cite{stevens-supercusp} and \cite{Blondel-Weil}. They satisfy the following properties.
 \begin{prop}\label{s_1-and-s_1-varpi-property}
\begin{enumerate}[(i)]
  \item Both $s_{y}^{P}$ and $s_{z}^{P}$ are not in $U_{\Lambda_{},E/\Eo}$;\label{s_1-and-s_1-varpi-property-notin}
      \item $s_{y}^{P}\in U_{\mathfrak{M}^y,E/\Eo}- U_{\mathfrak{M}^z,E/\Eo}$ and $s_{z}^{P}\in U_{\mathfrak{M}^z,E/\Eo}- U_{\mathfrak{M}^y,E/\Eo}$; \label{s_1-and-s_1-varpi-property-in}
      \item $(s_{y}^{P})^2=(s_{z}^{P})^2=1$ and $s_{z}^{P} s_{y}^{P}=(s_{y}^{P}s_{z}^{P})^{-1}=\mathcal{I}^M_+(\varpi I_{V_{i_\circ}})^{-1}$.
          \item $J_Ps_{z}^{P}J_Ps_{y}^{P} J_P=J_Ps_{z}^{P}s_{y}^{P} J_P=J_P(\mathcal{I}^M_+(\varpi)^{-1})J_P$. \label{s_1-and-s_1-varpi-relation}
\end{enumerate}
 \end{prop}
\proof
These can be found in \cite[(7.2.2) and Lemma 7.11]{stevens-supercusp}.
\qed

\begin{prop}\label{s_1-and-s_1-varpi-kappa}
Let $w$ be $y$ or $z$. The beta-extension $\tilde{\kappa}^{w,P}_{i_\circ}$ is conjugate-self-dual, which means that $\tilde{\kappa}^{w,P}_{i_\circ}$ is equivalent to $\tilde{\kappa}^{w,P}_{i_\circ}\circ\sigma$.


\end{prop}
\proof
 By \cite[Corollary 6.10(ii)]{stevens-supercusp}, the above equivalence is induced by the conjugation of $s_{w}^{P}\in U_{\mathfrak{M}^w,E/\Eo}$ on $\kappa^w_{P}|_{J_\Lambda\cap M}=\tilde{\kappa}^{w,P}_{i_\circ}\boxtimes\kappa^{w,P}_{x}$.
\qed

Following \cite[(4.11)]{stevens-supercusp}, we fix $w=z$ and call $\kappa=\kappa_\Lambda^{\mathfrak{M}^z}$ the \emph{standard beta-extension}. We then write $\kappa_P=\kappa_{\Lambda,P}^{\mathfrak{M}^z}$, $\lambda=\lambda_{\Lambda}^{\mathfrak{M}^z}$, and $\lambda_P=\lambda_{\Lambda,P}^{\mathfrak{M}^z}$. This choice is not very important as we can transit between $\kappa_P^y$ and $\kappa_P^z$ using the relation in Proposition \ref{transferring-character-maximal}. Hence the subsequent results for $\kappa=\kappa_\Lambda^{\mathfrak{M}^z}$ apply to $\kappa_\Lambda^{\mathfrak{M}^y}$ as well, with slight modification.

Recall that the pair $(J_P,\lambda_P)$  is a cover for $(J\cap M,\tilde{\lambda}^{P}_{i_\circ}\boxtimes \lambda^{P}_{x})$, where $\lambda^{P}_{x}=\lambda^{z,P}_{x}$ and $\tilde{\lambda}^{P}_{i_\circ}=\tilde{\lambda}^{z,P}_{i_\circ}$. Note that $\tilde{\lambda}^{P}_{i_\circ}\cong \tilde{\lambda}^{P}_{i_\circ}\circ\sigma$ if and only if $\tilde{\rho}_{i_\circ}\cong \tilde{\rho}_{i_\circ}\circ\sigma$, since  $\tilde{\kappa}^{P}_{i_\circ}\cong \tilde{\kappa}^{P}_{i_\circ}\circ\sigma$ by Proposition \ref{s_1-and-s_1-varpi-kappa}.

The following Proposition provides the structure of the Hecke algebras $\mathcal{H}(M,\tilde{\lambda}^{P}_{i_\circ}\boxtimes \lambda^{P}_{x})$ and $\mathcal{H}(G_W,\lambda_P)$, and describes the injective morphism $t_P$.
\begin{prop}\label{supp-Hecke-alg}
  \begin{enumerate}[(i)]
  \item $\mathcal{H}(M,\tilde{\lambda}^{P}_{i_\circ}\boxtimes \lambda^{P}_{x})$ is isomorphic to $$\mathcal{H}(\tilde{G}_{V_{i_\circ}}, \tilde{\lambda}^{P}_{i_\circ})\cong\mathbb{C}[Z,Z^{-1}],$$
    where $Z$ is an operator valued function supported on $\varpi\tilde{J}_{\Lambda_{i_\circ}}\subset \tilde{G}_{V_{i_\circ}}$.

    \item Suppose that $\tilde{\lambda}^{P}_{i_\circ}\ncong \tilde{\lambda}^{P}_{i_\circ}\circ\sigma$.
    \begin{enumerate}
      \item The support of $\mathcal{H}(G_W,\lambda_P)$ is
    $$J_P(I_{G_V}(\lambda^{P}_{x})\times \mathcal{I}^M_+(E^\times_{i_\circ}))J_P=J_P(J_{\Lambda_x}\times \mathcal{I}^M_+(E^\times_{i_\circ}))J_P.$$

    \item $t_P:\mathcal{H}(M,\tilde{\lambda}^{P}_{i_\circ}\boxtimes \lambda^{P}_{x}) \rightarrow \mathcal{H}(G_W,\lambda_P)$ is an isomorphism.
    \end{enumerate}

      \item
Suppose that $\tilde{\lambda}^{P}_{i_\circ}\cong \tilde{\lambda}^{P}_{i_\circ}\circ\sigma$.
\begin{enumerate}
  \item Define $\mathfrak{W}$ to be the affine Weyl group generated by $s_{y}^{P},\,s_{z}^{P}$ and the compact torus $T_{\Lambda_{i_\circ,x}}^0\cong \mathfrak{o}_{E_{i_\circ}}^\times
      \times T_{\lambda_x}$, then the support of $\mathcal{H}(G_W,\lambda_P)$ is
    $$J_P(I_{G_V}(\lambda^{P}_{x})\rtimes \mathfrak{W})J_P=J_P(J_{\Lambda_x}\rtimes \mathfrak{W})J_P.$$

    \item $\mathcal{H}(G_W,\lambda_P)$ is isomorphic to
     $$ \mathbb{C}[T_{y}^{P},T_{z}^{P}],$$
     where, for $w=y,\,z$, the operator valued function $T_{w}^{P}$  has support $J_Ps_{w}^{P}J_P$ and satisfies a quadratic relation of the form
  $$(T_w^{P}-\omega^{P}_{w,1})(T_w^{P}-\omega^{P}_{w,2})=0,$$
   where $\omega^{P}_{w,1}/\omega^{P}_{w,2}=-q^{r_w^{P}}$ for some $r_w^{P}\geq 0$, and moreover, we can normalize both $T_{y}^{P},\,T_{z}^{P}$ such that
  \begin{equation}\label{TyTz=Z}
  T_{y}^{P}T_{z}^{P}=t_P(Z).
\end{equation}

  \item $t_P$ defines a free $\mathcal{H}(M,\tilde{\lambda}^{P}_{i_\circ}\boxtimes \lambda^{P}_{x}) $-module structure on $\mathcal{H}(G_W,\lambda_P)$ of rank 2.
\end{enumerate}

    \end{enumerate}
\end{prop}

 Note that the parameters $r_w^{P}$ do not depend on the normalizations of $T_w^{P}$.



\proof
Most of the statements can be found in \cite[Proposition 3.3 and COrollary 3.4]{Blondel-Weil} (which summarizes several Propositions in \cite[Section 6]{stevens-supercusp}), and the others are easily consequences. In our situation, the element $\mathbf{p}$ in \emph{loc. cit.} is just $1$, and $s_{z}^{P}$ is always in $Z_{\tilde{G}_W}(\mathcal{I}_{i_\circ,x}(E^\times))$. Hence the conditions (ib) and (ic) in \emph{loc. cit.} are always unsatisfied.
\qed

We can compute explicitly the parameters ${r_w^{P}}$, for $w=y,\,z$, in the above quadratic relations. This will be done at the end of this section. We first reduce the computation of the above Hecke algebra to the Hecke algebras for finite groups.

\begin{prop}\label{hecke-alg-reduce-to-finite}
  There is an injective support-preserving algebra morphism
  $$j_{z}^{P}:\mathcal{H}(U_{\mathfrak{M}^z,E/\Eo},\rho_{i_\circ,x})\hookrightarrow \mathcal{H}(G_W,\lambda^z_P).$$
  The generator in $\mathcal{H}(U_{\mathfrak{M}^z,E/\Eo},\rho_{i_\circ,x})$ with support in $U_{\Lambda,E/\Eo}s_{z}^{P}U_{\Lambda,E/\Eo}$ has image $T_{z}^{P}$.
\end{prop}
\proof
This is \cite[(7.3)]{stevens-supercusp}, from where we have a sequence of morphisms
\begin{equation*}
\begin{split}
  &\mathcal{H}(U_{\mathfrak{M}^z,E/\Eo},\rho_{i_\circ,x})\cong \mathcal{H}(J_{\mathfrak{M}^z},\Res^{J_{\mathfrak{M}^z}}_{J_{\Lambda,\mathfrak{M}^z}}\kappa_{\mathfrak{M}^z}\cdot \rho_{i_\circ,x})
  \\
  \hookrightarrow & \mathcal{H}(G_W,\Res^{J_{\mathfrak{M}^z}}_{J_{\Lambda,\mathfrak{M}^z}}\kappa_{\mathfrak{M}^z}\cdot \rho_{i_\circ,x})\cong\mathcal{H}(G_W,\kappa^{\mathfrak{M}^z}_{\Lambda}\cdot \rho_{i_\circ,x})\cong \mathcal{H}(G_W,\kappa^{z}_{P}\cdot \rho_{i_\circ,x}).
\end{split}
 \end{equation*}
\qed

Note that a similar Proposition holds if we replace $z$ by $y$, since $\kappa^{z}_{P}=\kappa^{y}_{P}\cdot\chi^{z,P}_y$ for a tamely ramified character $\chi^{z,P}_y$. We therefore have the following (see \cite[Corollary 3.7]{Blondel-Weil}).

\begin{cor}\label{Hecke-generators}
  The generators of $\mathcal{H}(G_W,\lambda_P)$ are computed from those of
  $$\mathcal{H}(U_{\mathfrak{M}^y,E/\Eo},\rho_{i_\circ,x}\cdot\chi^{z,P}_y)\text{ and }\mathcal{H}(U_{\mathfrak{M}^z,E/\Eo},\rho_{i_\circ,x}).$$
\end{cor}
\proof
From Proposition \ref{hecke-alg-reduce-to-finite}, we have an injective morphism
$$j_{z}^{P}:\mathcal{H}(U_{\mathfrak{M}^z,E/\Eo},\rho_{i_\circ,x})\hookrightarrow \mathcal{H}(G_W,\lambda^z_P).$$
By Proposition \ref{transferring-character-maximal}, the right hand side above is isomorphic to $\mathcal{H}(G_W,\lambda^y_P\cdot\chi^{z,P}_y)$, which has a subalgebra injected from
$$j_{y}^{P}:\mathcal{H}(U_{\mathfrak{M}^y,E/\Eo},\rho_{i_\circ,x}\cdot\chi^{z,P}_y)\hookrightarrow \mathcal{H}(G_W,\lambda^y_P\cdot\chi^{z,P}_y).$$
\qed


\subsection{Points of reducibility}

Remember that we have set $\tilde{\lambda}^{P}_{i_\circ}=\tilde{\lambda}^{z,P}_{i_\circ}$ and $\lambda^{P}_{x} =\lambda^{z,P}_{x} $. Let $\pi^{P}_{x}:=\cInd_{J_{\Lambda_x}}^{G_V}\lambda^{P}_{x}$ and $\tilde{\pi}^{P}_{i_\circ}:=\cInd_{\tilde{\mathbf{J}}_{\Lambda_x}}^{\tilde{G}_{V_{i_\circ}}}\tilde{\boldsymbol{\lambda}}^{P}_{i_\circ}$,
where  $\tilde{\boldsymbol{\lambda}}^{P}_{i_\circ}$ is an extended maximal type extending $\tilde{\lambda}^{P}_{i_\circ}$.
Since $\tilde{\lambda}^{P}_{i_\circ}$ is conjugate-self-dual, there are two non-equivalent conjugate-self-dual choices of $\tilde{\boldsymbol{\lambda}}^{P}_{i_\circ}$, and hence two corresponding choices of $\tilde{\pi}^{P}_{i_\circ}$ differing from each other by the quadratic unramified character. Hence if the real parts of the points of reducibility of
$$\tilde{\pi}^{P}_{i_\circ}|\det|^s\rtimes \pi^{P}_{x}:=\Ind_P^G(\tilde{\pi}^{P}_{i_\circ}|\det|^s\boxtimes\pi^{P}_{x}),$$
for one of the two $\tilde{\pi}^{P}_{i_\circ}$, is known, then those of the another are the same.

 We have a crucial relation between the parameters of Hecke algebras and the points of reducibility.
\begin{prop}
  For $w=y$ or $z$, if $T_w^{P}$ satisfies the quadratic relation $(T_w^{P}-\omega^{P}_{w,1})(T_w^{P}-\omega^{P}_{w,2})=0$, where $\omega^{P}_{w,1}/\omega^{P}_{w,2}=-q^{r_w^{P}}$ for some $r_w^{P}\geq 0$, then the multi-set of the real parts of the points of reducibility is
  $$\{\pm \frac{r_{y}^{P}\pm r_{z}^{P}}{2n_{i_\circ}}\}$$
  (4 elements as a multi-set).
\end{prop}
\proof
The arguments appear in \cite[Section 3.2]{Blondel-Weil}, from which we summarize below very briefly. Let $\tilde{\pi}_\circ$ be a representation lying in the inertial class of $\tilde{\pi}^{P}_{i_\circ}$, and $\pi=\pi^{P}_{x}$. Let $\Xi$ be the irreducible module $h_M(\tilde{\pi}_\circ\boxtimes \pi)$, which is always a character since $\mathcal{H}(M,\tilde{\lambda}^{P}_{i_\circ}\boxtimes \lambda^{P}_{x})\cong\mathbb{C}[Z,Z^{-1}]$ is commutative. If $\tilde{\pi}_\circ$ is conjugate-self-dual (which happens when $\Ind_P^{G_W}(\tilde{\pi}_\circ|\det|^s\boxtimes \pi)$ is reducible by \cite{silberger-special}), then from Proposition \ref{supp-Hecke-alg} we know that
$$\mathrm{rank}_{\mathcal{H}(M,\tilde{\lambda}^{P}_{i_\circ}\boxtimes \lambda^{P}_{x})}(\mathcal{H}(G_W,\lambda_P))=
\#(N_{G_W}([\tilde{\pi}_\circ\boxtimes\pi]_{G_W}))=2.$$
If furthermore $\Ind_P^{G_W}(\tilde{\pi}_\circ|\det|^s\boxtimes \pi)$ is reducible, then $(t_P)_*(\Xi)$ must decompose into proper components, which must be characters by reason of dimension. On the one hand, from (\ref{TyTz=Z}), the character values $\Xi(t_P(Z_M))$ must be equal to one of the 4 products of eigenvalues (counted with multiplicities) of $T_y^P$ and $T_z^P$ in two different ways (counted with multiplicities) (see \cite[Proposition 1.13]{Bl-Bl-SP4}). On the other hand, there are two conjugate-self-dual representations $\tilde{\pi}_\circ$ in the inertial class of $\tilde{\pi}^{P}_{i_\circ}$ (differing from each other by the quadratic unramified character), each of them gives rise to the reducibility of $\Ind_P^{G_W}(\tilde{\pi}_\circ|\det|^s\boxtimes \pi)$ at two different points (counted with multiplicities). Comparing these four points with the 4 character values of $\Xi(t_P(Z_M))$ implies our desired result.
\qed

\subsection{Reduction to finite reductive quotients}

By Corollary \ref{Hecke-generators}, we reduce the computation of the Hecke algebra $\mathcal{H}(G_W,\lambda^w_P)$ to the algebras $\mathcal{H}(U_{\mathfrak{M}^y,E/\Eo},\rho_{i_\circ,x}\cdot\chi^{z,P}_y)\text{ and }\mathcal{H}(U_{\mathfrak{M}^z,E/\Eo},\rho_{i_\circ,x})$. Since the characters are tamely ramified, the computation is eventually reduced to the Hecke algebras for finite unitary groups.

Let $\Lambda_{\pm i_\circ}$ and $\mathfrak{M}^w_{\pm i_\circ}$ be the lattice sequences defined near the end of Section \ref{section lattice higher ranks} such that $\Lambda_{i_\circ,x}=\Lambda_{\pm i_\circ}\oplus_{i\neq i_\circ}\Lambda_{i}\text{ and }\mathfrak{M}^w_{i_\circ,x}=\mathfrak{M}^w_{\pm i_\circ}\oplus_{i\neq i_\circ}\Lambda_{i}$. The decomposition
$$U_{\mathfrak{M}^w_{i_\circ,x},E/\Eo}\cong U_{\mathfrak{M}^w_{\pm i_\circ},E_{i_\circ}/\Eo_{i_\circ}}\times \prod_{i\neq i_\circ}\mathrm{U}_{1}(E_i/\Eo_i)$$
induces an isomorphism of Hecke algebras
\begin{equation*}
     \mathcal{H}(U_{\mathfrak{M}^w_{i_\circ,x},E/\Eo},\rho_{i_\circ,x})\cong \mathcal{H}(U_{\mathfrak{M}^w_{\pm i_\circ},E_{i_\circ}/\Eo_{i_\circ}},\tilde{\rho}_{i_\circ}\boxtimes {\rho}_{i_\circ}).
\end{equation*}
Note that at this moment there is no relation between $\tilde{\rho}_{i_\circ}$ and ${\rho}_{i_\circ}$.  
Suppose $\tilde{\rho}_{i_\circ}$ is inflated from $\bar{\tilde{\rho}}_{i_\circ}$ a character of $\mathbf{k}_E^\times$, and ${\rho}_{i_\circ}$ is inflated from $\bar{{\rho}}_{i_\circ}$ a character of $\mathbf{k}_{E/\Eo}^\times$. Therefore,
\begin{equation*}
      \mathcal{H}(U_{\mathfrak{M}^w_{\pm i_\circ},E_{i_\circ}/\Eo_{i_\circ}},\tilde{\rho}_{i_\circ}\boxtimes {\rho}_{i_\circ})\cong
\mathcal{H}(\mathsf{G},\bar{\tilde{\rho}}_{i_\circ}\boxtimes \bar{\rho}_{i_\circ}),
\end{equation*}
where $\mathsf{G}$ is the finite reductive quotient $U_{\mathfrak{M}^w_{\pm i_\circ},E_{i_\circ}/\Eo_{i_\circ}}/U^1_{\mathfrak{M}^w_{\pm i_\circ},E_{i_\circ}/\Eo_{i_\circ}}$.

We now apply the statements analogous to \cite[Sec 4.5 and Appendix]{kutzko-Morris}, which are based on \cite{Lusztig-book}, to compute the parameters. Let $\tilde{\tau}$ be a $(q^{n_{i_\circ}}-1)$th root of unity in $\mathbb{F}_{q^{n_{i_\circ}}}^\times$ corresponding to $\bar{\tilde{\rho}}_{i_\circ}$, and $\tau$ be a $(\qo^{n_{i_\circ}}+1)$th root of unity in $\mathbb{F}_{q^{n_{i_\circ}}}^\times$ corresponding to $\bar{{\rho}}_{i_\circ}$ (see \cite[Prop 3.2.3]{Carter-book}).
        \begin{enumerate}[(i)]
          \item We first consider $\mathsf{G}=\mathsf{U}_2(\mathbf{k}_E/\mathbf{k}_\Eo)\times \mathsf{U}_1(\mathbf{k}_E/\mathbf{k}_\Eo)$, which is of course reduced to $\mathsf{G}=\mathsf{U}_2(\mathbf{k}_E/\mathbf{k}_\Eo)$, so that $\mathsf{G}^*=\mathsf{G}=\mathsf{U}_2(\mathbf{k}_E/\mathbf{k}_\Eo)$. We take $\mathsf{L}=\mathsf{T}$ the Siegel Levi, which is isomorphic to $\mathbf{k}_E^\times$, and take a character of $\mathsf{T}$ corresponding to $s=\diag(\tilde{\tau},{}^\sigma\tilde{\tau})$ in $\mathsf{T}^*\cong \mathsf{T}$. If ${}^\sigma\tilde{\tau}\neq\tilde{\tau}$, i.e., ${}^\sigma\bar{\tilde{\rho}}_{i_\circ}\ncong \bar{\tilde{\rho}}_{i_\circ}$, then the parameter is clearly trivial 1. Hence we let ${}^\sigma\tilde{\tau}=\tilde{\tau}$, or let $\bar{\tilde{\rho}}_{i_\circ}$ be a skew character, i.e., ${}^{{c}}\bar{\tilde{\rho}}_{i_\circ}=(\bar{\tilde{\rho}}_{i_\circ})^{\qo^{n_{i_\circ}}}=(\bar{\tilde{\rho}}_{i_\circ})^{-1}$, and so $s$ is central in $\mathsf{G}^*$. Hence
              \begin{equation*}
                                 Z_{\mathsf{G}^*}(s)=\mathsf{G}^*\text{ and }Z_{\mathsf{L}^*}(s)=\mathsf{T}^*,
                              \end{equation*}
                              In other words, the root system of $Z_{\mathsf{G}^*}(s)$ is of type A1 with Weyl group $\mathbb{Z}/2$. The Frobenius acts as the outer automorphism and preserves the simple roots, or in fact it acts trivially on the root system. We have a pair of representations
                              $(\mathsf{W}',\mathsf{W}'')=(\mathbb{Z}_2,1)$, such that $\mathsf{E}'=\mathsf{J}^{\mathsf{W}'}_{\mathsf{W}''}(\mathbf{1}_{\mathsf{W}''})=\mathbf{1}_{\mathsf{W}'}$ and $\mathsf{E}''=\sgn_{\mathsf{W}'}$. The parameter is $\qo^{n_{i_\circ}}=q^{n_{i_\circ}/2}$ by \cite[p.447]{Carter-book}.

                              \item We then consider $\mathsf{G}=\mathsf{U}_3(\mathbf{k}_E/\mathbf{k}_\Eo)$, so that $\mathsf{G}^*=\mathsf{G}$. We again take $\mathsf{L}=\mathsf{T}$ the Siegel Levi, which is isomorphic to $\mathbf{k}_E^\times\times \mathsf{U}_1(\mathbf{k}_E/\mathbf{k}_\Eo)$, and take a character of $\mathsf{T}$ corresponding to $s=\diag(\tilde{\tau},\tau,{}^\sigma\tilde{\tau})\in \mathsf{T}^*\cong \mathsf{T}$. If
                                   ${}^\sigma\tilde{\tau}\neq\tilde{\tau}$, i.e., ${}^\sigma\bar{\tilde{\rho}}_{i_\circ}\ncong \bar{\tilde{\rho}}_{i_\circ}$, or if
                                    $\tilde{\tau}\neq{\tau}$, i.e., $\bar{\tilde{\rho}}_{i_\circ}\ncong \bar{{\rho}}_{i_\circ}$,
                              then it is reduced to the previous case. Hence we let  $\tilde{\tau}={\tau}$, i.e., $\bar{\tilde{\rho}}_{i_\circ}\cong \bar{{\rho}}_{i_\circ}$, and so $s$ is central in $\mathsf{G}^*$. Hence
              \begin{equation*}
                                 Z_{\mathsf{G}^*}(s)=\mathsf{G}^*\text{ and }Z_{\mathsf{L}^*}(s)=\mathsf{T}^*,
                              \end{equation*}
                              The root system of $Z_{\mathsf{G}^*}(s)$ is of type A2, generated by $\mathbf{e}_1-\mathbf{e}_2$ and $\mathbf{e}_2-\mathbf{e}_3$, with Weyl group isomorphic to $S_3$. The Frobenius acts as $$\mathbf{e}_1\mapsto -\mathbf{e}_3,\,\mathbf{e}_2\mapsto -\mathbf{e}_2,$$
                               and so it acts on the simple roots transitively. We now apply \cite[Theorem 8.6.2]{Lusztig-book}. We have
                              $(\mathsf{W}',\mathsf{W}'')=(S_3,1)$. Now $\mathsf{E}'=\mathsf{J}^{\mathsf{W}'}_{\mathsf{W}''}(\mathbf{1}_{\mathsf{W}''})$. This is again $\mathbf{1}_{\mathsf{W}'}$, since the classification of representations of $S_3$ shows that only the trivial representation has generic degree $a=1=a({\mathbf{1}_{\mathsf{W}''}})$ \cite[p.446]{Carter-book}. Hence $\mathsf{E}''=\sgn_{\mathsf{W}'}$ and the parameter is $(\qo^{n_{i_\circ}})^3=q^{3n_{i_\circ}/2}$ by \cite[p.447]{Carter-book}.
        \end{enumerate}

 We compute the real parts of the points of reducibility. If
 $\bar{\tilde{\rho}}_{i_\circ}$ is not the base change of $ \bar{{\rho}}_{i_\circ}$, then $r_0=r_1=n_{i_\circ}/2$, and so $Re(s)=0$ or $\pm1/2$. Otherwise, $r_0=3n_{i_\circ}/2$ and $r_1=n_{i_\circ}/2$, and so $Re(s)=\pm1/2$ or $\pm1$.

\newpage\section{Amending}\label{section amending}

In this section, we study the characters $\mu^w_{\tilde{\xi}}$ and $\chi_{y}^{z,P}$ that appear in (\ref{amending appear}) and Proposition \ref{transferring-character-maximal}. The former one is called an \emph{amending character}, and the latter one is called a \emph{transfer character}.

\subsection{Amending characters}\label{section amending char}

Given a group $G$ and a finite subgroup $H$ if finite index, we denote by $\delta_H^G$ the discriminant character $\det\Ind_H^G(\mathbf{1}_H)$ of $G$, which is equal to the signature character
$\sgn_{G}(G/H)$ for the $G$-action on the coset space $G/H$.

\subsubsection{The first amending character}

We denote by $U$ and $U^-$ the unipotent subgroups of the parabolic subgroup $P$ and its opposite $P^-$ respectively. For $\epsilon=-$ or $+$, denote $V_\epsilon=V_{i_\circ,\epsilon}$, and for $j,k\in \{-,+\}\cup I$, denote the space $U_{(
    j , k
  )}=\mathrm{Hom}_F(V_k,V_j)$. Then $U$ can be written as $$U_{(
    - , +
  )}\oplus{\bigoplus_{i\in I}}(U_{(
    - , i
  )}\oplus U_{(
    i , +
  )}),$$
and similarly $U^-$ can be written as ,
  $$U_{(
      + , -
  )}\oplus{\bigoplus_{i\in I}}(U_{(
    i , -
  )}\oplus U_{(
    + , i
  )}).$$
Write $\Lambda=\Lambda_{i_\circ,x}$ and $J_P=(J_\Lambda\cap P)H^1_\Lambda$, so that
$$J_\Lambda/J_P\cong \mathfrak{V}_{\Lambda,U^-}:= J_\Lambda^1{\cap U^-}/H_\Lambda^1{\cap U^-}.$$
and
$ \delta_{J_P}^{J_\Lambda}|_{J_\Lambda\cap M}=\sgn_{{J_\Lambda\cap M}}(\mathfrak{V}_{\Lambda,U^-}).$
Note that this is a tamely ramified character; more precisely, if we write ${J_\Lambda\cap M}=J_{\Lambda_x}\times\tilde{J}_{\Lambda_{i_\circ}}$, then the character is trivial on the pro-p-subgroup ${J^1_\Lambda\cap M}=J^1_{\Lambda_x}\times\tilde{J}^1_{\Lambda_{i_\circ}}$ since the signature character is quadratic. Hence it is enough to compute the character on
$${J_{\Lambda}}/{J^1_{\Lambda}}={J_\Lambda\cap M}/{J^1_\Lambda\cap M}\cong T_\Lambda/T^1_\Lambda\cong
  {\boldsymbol{\mu}}_\Lambda:={\boldsymbol{\mu}}_{E_{i_\circ}}\times{\boldsymbol{\mu}}_x,$$
  where ${\boldsymbol{\mu}}_x:=\prod_{i\in I}{{\boldsymbol{\mu}}_{E_{i}/\Eo_{i}}}$.

We write $\mathfrak{M}^w=\mathfrak{M}^w_{i_\circ,x}$, where $w=y$ or $z$.
\begin{prop}\label{rectifier 1}
  Let $\kappa^w=\kappa_\Lambda^{\mathfrak{M}^w}$ and $\kappa_P^w=\kappa_{\Lambda,P}^{\mathfrak{M}^w}$, then
  $$\det\kappa^w|_{{J_\Lambda\cap M}}=\det\left(\kappa^w_P|_{{J_\Lambda\cap M}}\cdot\sgn_{{J_\Lambda\cap M}}(\mathfrak{V}_{\Lambda,U^-})\right)^{\#\mathfrak{V}_{\Lambda,U^-}}.$$
\end{prop}
\proof
Recall that $\kappa^w=\Ind_{J_P}^{J_\Lambda}\kappa^w_P$. By taking determinant, we have
$$\det\kappa^w=(\delta_{J_P}^{J_\Lambda})^{\dim\kappa^w_P}(\det\kappa^w_P\circ T_{J_P}^{J_\Lambda}).$$
We restrict to ${J_\Lambda\cap M}$ and observe that
\begin{itemize}
  \item since $\dim \kappa^w_P={\#\mathfrak{V}_{\Lambda,U^-}}$ is a $p$-power, in particular odd, we can write
  \begin{equation*}
     \begin{split}
       (\delta_{J_P}^{J_\Lambda}|_{{J_\Lambda\cap M}})^{\dim\kappa^w_P}
       =&(\sgn_{{J_\Lambda\cap M}}(\mathfrak{V}_{\Lambda,U^-}))^{\#\mathfrak{V}_{\Lambda,U^-}}
       \\
       =&(\sgn_{{J_\Lambda\cap M}}(\mathfrak{V}_{\Lambda,U^-}))^{(\#\mathfrak{V}_{\Lambda,U^-})^2}
       \\
       =&\det(\sgn_{{J_\Lambda\cap M}}(\mathfrak{V}_{\Lambda,U^-})\cdot\mathbf{1}_{\kappa^w_P})^{\#\mathfrak{V}_{\Lambda,U^-}}
     \end{split}
   \end{equation*}
   Here $\mathbf{1}_{\kappa^w_P}$ is the trivial representation on the representation space of $\kappa^w_P$.
  \item Since  ${J_\Lambda\cap M}$ normalizes $J_\Lambda^1{\cap U^-}$ and $H_\Lambda^1{\cap U^-}$, we have $T_{J_P}^{J_\Lambda}|_{{J_\Lambda\cap M}}(m)=m^{\#\mathfrak{V}_{\Lambda,U^-}}$, and so
      $$(\det\kappa^w_P\circ T_{J_P}^{J_\Lambda})|_{{J_\Lambda\cap M}}=
      \left(\det\kappa^w_P|_{{J_\Lambda\cap M}}\right)^{\#\mathfrak{V}_{\Lambda,U^-}}.$$
      \end{itemize}
  \qed

We denote the above amending character $\sgn_{J_\Lambda\cap M}(\mathfrak{V}_{\Lambda,U^-})$ by $\mu^P_{\tilde{\xi}}=\tilde{\nu}^P_{i_\circ,x,\tilde{\xi}}\boxtimes \nu^P_{x,\tilde{\xi}}$, where
$$\tilde{\nu}^P_{i_\circ,x,\tilde{\xi}}=\sgn_{\tilde{J}_{\Lambda_{i_\circ}}}(\mathfrak{V}_{\Lambda,U^-})\text{ and }\nu^P_{x,\tilde{\xi}}=\sgn_{J_{\Lambda_x}}(\mathfrak{V}_{\Lambda,U^-}) .$$
They are inflated from a character of  ${\boldsymbol{\mu}}_{E_{i_\circ}}:=U_{E_{i_\circ}}/U_{E_{i_\circ}}^1$ and ${\boldsymbol{\mu}}_x:=T_x/T_x^1$ respectively.

  \begin{cor}\label{cor-nu-P}
   $\nu^P_{x,\tilde{\xi}}$ is always trivial.
  \end{cor}
 \proof

Each component ${\boldsymbol{\mu}}_{E_i/\Eo_i}$ of ${\boldsymbol{\mu}}_x$ acts on the space $\mathfrak{V}_{\Lambda,U^-}$, which can be decomposed as
$$\mathfrak{V}_{\Lambda,U^-}=\mathfrak{V}_{\Lambda,(
      + , -
  )}\oplus {\bigoplus_{i\in I}}\mathfrak{V}_{\Lambda,(
    i , -
  ),(
    + , i
  )}
  $$
where
$$\mathfrak{V}_{\Lambda,(
      + , -
  )}=\frac{J^1_\Lambda\cap U_{(
      + , -
  )}^\sigma}{H^1_\Lambda\cap U^-}
  \text{ and }\mathfrak{V}_{\Lambda,(
    i , -
  ),(
    + , i
  )}=\frac{J^1_\Lambda\cap (U_{(
    i , -
  )}\oplus U_{(
    + , i
  )})^\sigma}{H^1_\Lambda\cap U^-}
.$$
The action  of ${\boldsymbol{\mu}}_{E_i/\Eo_i}$ on $\mathfrak{V}_{\Lambda,(
      + , -
  )}$ is clearly trivial. For each $i\in I$, the space $\mathfrak{V}_{\Lambda,(
    i , -
  ),(
    + , i
  )}$ is either trivial or isomorphic to $\mathbf{k}_{E_i}\otimes_{\mathbf{k}_F}\mathbf{k}_{E_{i_\circ}}$, where ${\boldsymbol{\mu}}_{E_i/\Eo_i}\cong \mathbf{k}^\times_{E_i/\Eo_i}$ acts on the first factor faithfully by multiplication. The order of ${\boldsymbol{\mu}}_{E_i/\Eo_i}$ is $\qo_i+1$, while the number $\#(\mathbf{k}_{E_i}\otimes_{\mathbf{k}_F}\mathbf{k}_{E_{i_\circ}}-\{0\})$ is a multiple of $(q_i-1)=(\qo_i-1)(\qo_i+1)$. Hence $\#(\mathbf{k}_{E_i}\otimes_{\mathbf{k}_F}\mathbf{k}_{E_{i_\circ}}-\{0\})$ is an even multiple of the order of order of ${\boldsymbol{\mu}}_{E_i/\Eo_i}$, and so the signature of this action is trivial.
\qed

\subsubsection{The second amending character}

  Let $A$ be a group on which an abelian group $T$ acts as automorphisms. Let $B$ and $C$ be two subgroups of $A$ such that $[B:B\cap C]$ and $[C:B\cap C]$ are finite. We define
  $$\mathrm{sgn}_{T}(B:C)=\mathrm{sgn}_{T}(B:B\cap C)\mathrm{sgn}_{T}(C:B\cap C).$$

  \begin{prop}\label{rectifier-2}
Suppose that $\mathfrak{A}_{\mathfrak{M}}\supseteq \mathfrak{A}_{\Lambda}$, and that
$\kappa^{\mathfrak{M}}_{\Lambda}$ and $\det\kappa_{\mathfrak{M}}$ are related as in Definition \ref{definition-compatible-beta}(\ref{definition-compatible-beta-extension}), then
$$\det\kappa^{\mathfrak{M}}_{\Lambda}|_{J_{\Lambda}\cap {M}}\cdot \sgn_{|_{J_{\Lambda}\cap {M}}}(J^1_{\Lambda}:J^1_{\mathfrak{M}})\equiv (\det\kappa_{\mathfrak{M}}|_{J_{\Lambda}\cap {M}})^{p^k}$$
    for some $p$-power $p^k$. The above equivalence is modulo the group of characters of ${J_{\Lambda}\cap {M}}$ whose orders are $p$-powers.
  \end{prop}

  \proof
  Recall from Definition \ref{definition-compatible-beta}(\ref{definition-compatible-beta-extension}) the defining property for $\kappa^\mathfrak{M}_\Lambda$ compatible with $\kappa^\mathfrak{M}$, that
  $$\Ind_{J_\Lambda}^{U_{\Lambda,E/\Eo}U^1_\Lambda}\kappa^\mathfrak{M}_\Lambda\cong \Ind^{U_{\Lambda,E/\Eo}U^1_\Lambda}_{J_{\Lambda,\mathfrak{M}}} \Res^{J_{\mathfrak{M}}}_{J_{\Lambda,\mathfrak{M}}}\kappa_{\mathfrak{M}}.$$
  If we take determinant on both sides, we obtain
  \begin{equation}\label{compare-p-power-char}
  \begin{split}
    &\left(\delta_{J_\Lambda}^{U_{\Lambda,E/\Eo}U^1_\Lambda}\right)^{\dim\kappa^\mathfrak{M}_\Lambda}
    \left(\det\kappa^\mathfrak{M}_\Lambda\circ T_{J_\Lambda}^{U_{\Lambda,E/\Eo}U^1_\Lambda}\right)
    \\
    =&
    \left(\delta^{U_{\Lambda,E/\Eo}U^1_\Lambda}_{J_{\Lambda,\mathfrak{M}}}\right)^{\dim\kappa_{\mathfrak{M}}}
  \left(\det\kappa_{\mathfrak{M}}|_{J_{\Lambda,\mathfrak{M}}}\circ T^{U_{\Lambda,E/\Eo}U^1_\Lambda}_{J_{\Lambda,\mathfrak{M}}}\right).
  \end{split}
  \end{equation}
  We can get rid of the powers of the $\delta$-characters, because these powers are all odd and $\delta$-characters are quadratic. We restrict to ${J_{\Lambda}\cap {M}}$ and observe the following.

  \begin{itemize}

    \item We have
    $$T^{U_{\Lambda,E/\Eo}U^1_\Lambda}_{J_{\Lambda}}|_{{J_{\Lambda}\cap {M}}}(t)\equiv t^{\#({U_{\Lambda,E/\Eo}U^1_\Lambda}/{J_{\Lambda}})}\mod U^1_\Lambda,$$
     since if we choose the transversals of $${U_{\Lambda,E/\Eo}U^1_\Lambda}/{J_{\Lambda}}={{U_{\Lambda,E/\Eo}U^1_\Lambda}}/{U_{\Lambda,E/\Eo}J^1(\Lambda)}=
     {U^1_\Lambda}/{J^1(\Lambda)}$$
     of the form $1+X_i\in U^1_\Lambda$, then for every $t(1+X)\in{U_{\Lambda,E/\Eo}U^1_\Lambda}$, we have
        $t(1+X)(1+X_1)=(1+X_j)t(1+Y_i)$ for some $1+X_j\in  U^1_\Lambda$ and $t(1+Y_i)\in J_\Lambda$.
        \\
        Similar result holds if we replace $J_\Lambda$ by ${J_{\Lambda,\mathfrak{M}}}$.

  \item

   We recall that, if $J\subseteq H\subseteq G$ is a tower of subgroups of finite indexes, then
  $$\delta_J^G=(\delta_H^G)^{\#(H/J)}\cdot(\delta_J^H\circ T_H^G).$$
  If now $H$ and $K$ are two subgroups of $G$ such that the index $[G:H\cap K]$ is odd, then
  \begin{equation*}
  \begin{split}
    &\delta_H^G\delta_K^G= (\delta_H^G)^{\#(H/H\cap K)}(\delta_K^G)^{\#(K/H\cap K)}
    \\
    =&  (\delta_{H\cap K}^H\circ T_H^G)(\delta_{H\cap K}^K\circ T_K^G)
    \\
  =&\sgn_{H}(H/H\cap K)^{\#(G/H)}\sgn_{K}(K/H\cap K)^{\#(G/K)}
  \\
  =&\sgn_{H}(H/H\cap K)\sgn_{K}(K/H\cap K).
  \end{split}
      \end{equation*}

  We set $G={U_{\Lambda,E/\Eo}U^1_\Lambda}$, $H=J_{\Lambda}$ and $K={J_{\Lambda,\mathfrak{M}}}$.
  Since $H/H\cap K=J_{\Lambda}/J_{\Lambda}\cap {J_{\Lambda,\mathfrak{M}}}=J^1_{\Lambda}/J^1_{\Lambda}\cap {J^1_{\Lambda,\mathfrak{M}}}$, and the action of $H$ on $H/H\cap K$ factors through ${\boldsymbol{\mu}}_\Lambda$ (and similar results hold if we replace $H$ by $K$), the above product of characters is equal to
  $$\sgn_{{J_{\Lambda}\cap {M}}}(J^1_{\Lambda}/J^1_{\Lambda}\cap {J^1_{\Lambda,\mathfrak{M}}})\sgn_{{J_{\Lambda}\cap {M}}}(J^1_{\mathfrak{M}}/J^1_{\Lambda}\cap {J^1_{\Lambda,\mathfrak{M}}})=\sgn_{{J_{\Lambda}\cap {M}}}(J^1_{\Lambda}:{J^1_{\mathfrak{M}}}).$$

       \end{itemize}

       From (\ref{compare-p-power-char}) and the above observations, we know that
       $$\left(\det\kappa^\mathfrak{M}_\Lambda |_{{J_{\Lambda}\cap {M}}} \right)^{\#({U_{\Lambda,E/\Eo}U^1_\Lambda}/{J_\Lambda})}
       \sgn_{{J_{\Lambda}\cap {M}}}({J^1_{\Lambda}}:{J^1_{\mathfrak{M}}})
      $$
       and $$
      \left(\det\kappa_{\mathfrak{M}}|_{{J_{\Lambda}\cap {M}}}\right)^{\#({U_{\Lambda,E/\Eo}U^1_\Lambda}/{J_{\Lambda,\mathfrak{M}}})}$$
      differ by a character of a $p$-power order. Suppose we write  $p^a={\#({U_{\Lambda,E/\Eo}U^1_\Lambda}/{J_\Lambda})}$ and $p^b={\#({U_{\Lambda,E/\Eo}U^1_\Lambda}/{J_{\Lambda,\mathfrak{M}}})}$. We now choose a large integer $m$ such that
      $$p^m\equiv 1\mod\text{ the order of }\det\kappa^\mathfrak{M}_\Lambda |_{{{\boldsymbol{\mu}}_{\Lambda}}}\cdot
       \sgn_{{{\boldsymbol{\mu}}_{\Lambda}}}({J^1_{\Lambda}}:{J^1_{\mathfrak{M}}})$$
       and so
       \begin{equation*}
         \begin{split}
          & \det\kappa^\mathfrak{M}_\Lambda |_{{J_{\Lambda}\cap {M}}} \cdot
       \sgn_{{J_{\Lambda}\cap {M}}}({J^1_{\Lambda}}:{J^1_{\mathfrak{M}}})
       \\
        \equiv &
      (\det\kappa^\mathfrak{M}_\Lambda |_{{J_{\Lambda}\cap {M}}} \cdot
       \sgn_{{J_{\Lambda}\cap {M}}}({J^1_{\Lambda}}:{J^1_{\mathfrak{M}}}))^{p^m}
       \\
        \equiv &
      ((\det\kappa^\mathfrak{M}_\Lambda |_{{J_{\Lambda}\cap {M}}} \cdot
       \sgn_{{J_{\Lambda}\cap {M}}}({J^1_{\Lambda}}:{J^1_{\mathfrak{M}}}))^{p^a})^{p^{m-a}}
       \\
        \equiv &
        (\det\kappa_{\mathfrak{M}}|_{{J_{\Lambda}\cap {M}}})^{p^{b+m-a}},
         \end{split}
       \end{equation*}
       where the equivalence is modulo the group of characters of ${J_{\Lambda}\cap {M}}$ whose orders are $p$-powers.
  \qed

Suppose we take $\kappa^{\mathfrak{M}}_{\Lambda}=\kappa^{\mathfrak{M}}_{\Lambda,\tilde{\xi}}$ and $\kappa_{\mathfrak{M}}=\kappa_{\mathfrak{M},\tilde{\xi}}$ as defined in Proposition \ref{beta-extension-max-min} and Definition \ref{definition-compatible-beta}.

\begin{cor}\label{rectifier-2-nc}
  The character $\det\kappa^{\mathfrak{M}}_{\Lambda,\tilde{\xi}}|_{J_{\Lambda}\cap {M}}$ differs from $\sgn_{|_{J_{\Lambda}\cap {M}}}(J^1_{\Lambda}:J^1_{\mathfrak{M}})$ by a character of a $p$-power order.
\end{cor}
\proof
Note that the choice of $\kappa_{\mathfrak{M},\tilde{\xi}}$ from Proposition \ref{beta-extension-max-min} implies that $\det\kappa_{\mathfrak{M},\tilde{\xi}}|_{J_\Lambda\cap M}$ has a p-power order.
\qed

We denote the above amending character $\sgn_{J_{\Lambda}\cap {M}}(J^1_\Lambda:J^1_{\mathfrak{M}^w})$ by $\mu^w_{\tilde{\xi}}=\tilde{\nu}^w_{i_\circ,x,\tilde{\xi}}\boxtimes \nu^w_{x,\tilde{\xi}}$, where
$$\tilde{\nu}^w_{i_\circ,x,\tilde{\xi}}:=\sgn_{\tilde{J}_{\Lambda_{i_\circ}}}(J^1_\Lambda:J^1_{\mathfrak{M}^w})\text{ and }\nu^w_{x,\tilde{\xi}}:=\sgn_{J_{\Lambda_x}}(J^1_\Lambda:J^1_{\mathfrak{M}^w}).$$
Again they are inflated from a character of ${\boldsymbol{\mu}}_{E_{i_\circ}}:=U_{E_{i_\circ}}/U_{E_{i_\circ}}^1$ and ${\boldsymbol{\mu}}_x:=T_x/T_x^1$ respectively.

\begin{cor}
   For $w=y$ or $z$, the character $\nu^w_{x,\tilde{\xi}}$ is always trivial.
  \end{cor}
\proof
 For each component ${\boldsymbol{\mu}}_{E_{i_\circ}/\Eo_{i_\circ}}$ of ${\boldsymbol{\mu}}_x$, the root spaces on which this component acts non-trivially are described as follows. \begin{enumerate}[(i)]
   \item   The $\sigma$-invariant parts of the sums of blocks $(U_{( -,i_\circ  )}\oplus U_{(i_\circ, + )})^\sigma$ in $U$, $(U_{( i_\circ,-  )}\oplus U_{(+,i_\circ)})^\sigma$ in $U^-$, and $(U_{( i,i_\circ  )}\oplus U_{(i_\circ, i )})^\sigma$ in $G_V$, for $i\in I-\{ i_\circ\}$.
       \item The root spaces in $G_{V_{i_\circ}}$ over the torus $T_{i_\circ}\cong \mathrm{U}_1(E_{i_\circ}/\Eo_{i_\circ})$, which are of the form $U_{(\sigma_{i_\circ}^j,\sigma_{i_\circ}^k)}$ for $j,k=1,\dots,n_{i_\circ}$ and $j\neq k$, with the action given as follows: if we write
           $$U_{\sigma_{i_\circ},l}:=\bigoplus_{j-k\equiv l\mod n_{i_\circ}}U_{(\sigma_{i_\circ}^j,\sigma_{i_\circ}^k)}$$
           for $l=1,\dots,n_{i_\circ}-1$ , then $U_{\sigma_{i_\circ},l}\cong E_{i_\circ}$ and $t\in T_{i_\circ}$ acts by multiplication by $({}^{\sigma_{i_\circ}^l}t)(t^{-1})$. The sum $U_{\sigma_{i_\circ},l}\oplus U_{\sigma_{i_\circ},-l}$ is $\sigma$-invariant, for $l=1,\dots,\frac{1}{2}(n_{i_\circ}-1)$ (remember that $n_{i_\circ}$ is odd), and the components we look for are
           $(U_{\sigma_{i_\circ},l}\oplus U_{\sigma_{i_\circ},-l})^\sigma$.
 \end{enumerate}
 Using the arguments similar to Corollary \ref{cor-nu-P}, we can show that the signatures of the actions of ${\boldsymbol{\mu}}_{E_{i_\circ}/\Eo_{i_\circ}}$ on these root spaces are trivial.
\qed

  Therefore, if we denote
  $$\mu^{w,P}_{\tilde{\xi}}=\mu^w_{\tilde{\xi}}\mu^P_{\tilde{\xi}}=\tilde{\nu}^{w,P}_{i_\circ,x,\tilde{\xi}}\boxtimes \nu^{w,P}_{x,\tilde{\xi}},$$
  then the product $\nu^{w,P}_{x,\tilde{\xi}}={\nu}^w_{x,\tilde{\xi}}\nu^P_{x,\tilde{\xi}}$ is always trivial, and $\tilde{\nu}^{w,P}_{i_\circ,x,\tilde{\xi}}=\tilde{\nu}^w_{i_\circ,x,\tilde{\xi}}\tilde{\nu}^P_{i_\circ,x,\tilde{\xi}}$ is equal to the inflation of
  \begin{equation}\label{product-nu-w-nu-P}
    \sgn_{{\boldsymbol{\mu}}_{E_{i_\circ}}}(J^1_{\Lambda}:J^1_{\mathfrak{M}^w})
    \sgn_{{\boldsymbol{\mu}}_{E_{i_\circ}}}(J^1_{\Lambda}\cap U^-/H^1_{\Lambda}\cap U^-).
  \end{equation}

\subsubsection{The transfer character}

Recall the character $\chi_{y}^{z,P}$ of $J_\Lambda/J^1_\Lambda\cong {T}_{\Lambda}/T^1_\Lambda$
defined in (\ref{transferring-character-maximal}). Since ${T}_{\Lambda}\cong \tilde{T}_{i_\circ}\times T_x$, we can write
$$\chi_{y}^{z,P}=\tilde{\chi}_{y,i_\circ}^{z,P}\boxtimes \chi_{y,x}^{z,P}.$$

\begin{prop}\label{transition-with-amending}
The transfer character $\chi_{y}^{z,P}$ is equal to
$$\mu^{y,P}_{\tilde{\xi}}\mu^{z,P}_{\tilde{\xi}}=\mu^{y}_{\tilde{\xi}}\mu^{z}_{\tilde{\xi}}.$$
In particular, it is independent of the parabolic subgroup $P$.
\end{prop}
\proof
Recall from Proposition \ref{transferring-character-maximal} that
\begin{equation*}
      \det\kappa_\Lambda^{\mathfrak{M}^z}|_{J_\Lambda\cap M}=(\det\kappa_{\Lambda}^{\mathfrak{M}^y}|_{J_\Lambda\cap M})\chi_y^{z,P}.
\end{equation*}
Moreover, for $w=y$ or $z$, we have
$$   \det\kappa_\Lambda^{\mathfrak{M}^w}|_{J_\Lambda\cap M}=(\det\kappa_{\mathfrak{M}^w}|_{J_\Lambda\cap M})^{p^{k_w}}\cdot\mu^{w}_{\tilde{\xi}}$$
for some $p$-power ${p^{k_w}}$, by Proposition \ref{rectifier-2}. Hence
\begin{equation*}
  \begin{split}
    \mu^{y}_{\tilde{\xi}}\mu^{z}_{\tilde{\xi}}&= \det\kappa_\Lambda^{\mathfrak{M}^z}|_{J_\Lambda\cap M}( \det\kappa_\Lambda^{\mathfrak{M}^y}|_{J_\Lambda\cap M})^{-1}(\det\kappa_{\mathfrak{M}^z}|_{J_\Lambda\cap M})^{-p^{k_z}}(\det\kappa_{\mathfrak{M}^y}|_{J_\Lambda\cap M})^{p^{k_y}}
    \\
    &=\chi_y^{z,P}(\det\kappa_{\mathfrak{M}^z}|_{J_\Lambda\cap M})^{-p^{k_z}}(\det\kappa_{\mathfrak{M}^y}|_{J_\Lambda\cap M})^{p^{k_y}}.
  \end{split}
\end{equation*}
Since $\mu^{y}_{\tilde{\xi}}\mu^{z}_{\tilde{\xi}}(\chi_y^{z,P})^{-1}$ has an order coprime to $p$, and $(\det\kappa_{\mathfrak{M}^z}|_{J_\Lambda\cap M})^{-p^{k_z}}(\det\kappa_{\mathfrak{M}^y}|_{J_\Lambda\cap M})^{p^{k_y}}$ has a p-power order, these characters are trivial and $\mu^{y}_{\tilde{\xi}}\mu^{z}_{\tilde{\xi}}=\chi_y^{z,P}$. The independence of $P$ is then clear.
\qed

We hence write $ \chi_{y}^{z}$ instead of $ \chi_{y}^{z,P}$, and similarly write $\chi_{y}^{z}=\tilde{\chi}_{y,i_\circ}^{z}\boxtimes \chi_{y,x}^{z}$. In particular, the character $ \chi_{y,x}^{z}$ of ${\boldsymbol{\mu}}_x$ is trivial, and $\tilde{\chi}_{y,i_\circ}^{z}$ of ${\boldsymbol{\mu}}_{E_{i_\circ}}$ is a quadratic (hence skew) character.

\subsection{The process of amending}

We now fix $\kappa_\Lambda^{\mathfrak{M}^z}=\kappa_{\Lambda,\tilde{\xi}}^{\mathfrak{M}^z}$, the beta-extension defined in Definition \ref{definition-compatible-beta}(\ref{definition-compatible-beta-extension-nc}). From Proposition \ref{rectifier 1} and Corollary \ref{rectifier-2-nc}, we have
\begin{equation*}
    \det\kappa_{\Lambda,\tilde{\xi}}^{\mathfrak{M}^z}|_{J_\Lambda\cap M}=(\det\kappa_{\Lambda,\tilde{\xi},P}^{\mathfrak{M}^z}|_{J_\Lambda\cap M})\cdot{{\mu}_{\tilde{\xi}}^P}\equiv \mu_{\tilde{\xi}}^z.
\end{equation*}
Note that $\kappa_{\Lambda,\tilde{\xi},P}^{\mathfrak{M}^z}|_{J_\Lambda\cap M}\cdot\mu_{\tilde{\xi}}^{z,P}=(\kappa_{\Lambda,\tilde{\xi}}^{\mathfrak{M}^z}\cdot\mu_{\tilde{\xi}}^{z,P})_P|_{J_\Lambda\cap M}$ is a beta-extension of $\tilde{\eta}_{i_\circ}\boxtimes\eta_x$ by Proposition \ref{restrictions-are-beta}, and has determinant of a $p$-power order. Therefore,
$$\kappa_{\Lambda,\tilde{\xi},P}^{\mathfrak{M}^z}|_{J_\Lambda\cap M}\cdot\mu_{\tilde{\xi}}^{z,P}\cong \tilde{\kappa}_{\tilde{\xi}_{i_\circ}}\boxtimes\kappa_{x,\tilde{\xi}},$$
where $\tilde{\kappa}_{\tilde{\xi}_{i_\circ}}$ and $\kappa_{x,\tilde{\xi}}$ are the fixed beta-extensions in Propositions \ref{beta-extension-tilde} and \ref{beta-extension-U} respectively. Moreover, $\mathcal{H}(G_W,\kappa_{\Lambda,\tilde{\xi}}^{\mathfrak{M}^z}\cdot \rho_{i_\circ,x})$ is a cover of $\mathcal{H}(M,\kappa_{\Lambda,\tilde{\xi},P}^{\mathfrak{M}^z}|_{J_\Lambda\cap M}\cdot \rho_{i_\circ,x})=\mathcal{H}(M,\tilde{\lambda}^{z,P}_{i_\circ}\boxtimes\lambda^{}_{x})$, where $$\tilde{\lambda}^{z,P}_{i_\circ}\boxtimes\lambda^{}_{x}=(\tilde{\kappa}_{\tilde{\xi}_{i_\circ}}\boxtimes\kappa_{x,\tilde{\xi}})\cdot (\tilde{\nu}_{i_\circ,x,\tilde{\xi}}^{z,P}\tilde{\rho}_{i_\circ}\boxtimes\rho_x).$$
If $\tilde{\pi}^{z,P}_{i_\circ}\boxtimes\pi^{}_{x}$ is a supercuspidal of $M$ with type $\tilde{\lambda}^{z,P}_{i_\circ}\boxtimes\lambda^{}_{x}$, then the real parts of the points $s$ where
$$\tilde{\pi}^{z,P}_{i_\circ}|\det|^s\rtimes \pi^{}_{x}:=\Ind_P^G(\tilde{\pi}^{z,P}_{i_\circ}|\det|^s\boxtimes\pi^{}_{x})$$
  is reducible are determined by $\mathcal{H}(G_W,\kappa_{\Lambda,\tilde{\xi}}^{\mathfrak{M}^z}\cdot \rho_{i_\circ,x})$.

\begin{thm}\label{theorem-inertial-class}
  There exists a conjugate-self-dual representation $\tilde{\pi}_{i_\circ}$ (i.e., $\tilde{\pi}_{i_\circ}\cong \tilde{\pi}_{i_\circ}\circ\sigma$) of $\tilde{G}_{V_i}$ whose inertial class is determined by the maximal simple type
  $$\tilde{\kappa}_{\tilde{\xi}_{i_\circ}}\cdot (\tilde{\xi}_{i_\circ,\tame}\tilde{\nu}^{y,P}_{i_\circ,x,\tilde{\xi}})$$
  such that if $\pi_{x,\tilde{\xi}}$ is a supercuspidal representation of $G_V$ induced from the cuspidal type
  $${\kappa}_{x,\tilde{\xi}}\cdot {\xi}_{x,\tame},$$
  then $\tilde{\pi}_{i_\circ}|\det|^s\rtimes \pi_{x,\tilde{\xi}}$ is reducible at a point $s$ whose real part is 1.
\end{thm}
\proof
Recall that  $\mathcal{H}(G_W,\kappa_{\Lambda,\tilde{\xi}}^{\mathfrak{M}^z}\cdot \rho_{i_\circ,x})$ is covering
$$\mathcal{H}(M,(\tilde{\kappa}_{\tilde{\xi}_{i_\circ}}\boxtimes\kappa_{x,\tilde{\xi}})\cdot (\tilde{\nu}_{i_\circ,x,\tilde{\xi}}^{z,P}\tilde{\rho}_{i_\circ}\boxtimes\rho_x)).$$
From Corollary \ref{Hecke-generators} that its generators are given by those of
$$\mathcal{H}(U_{\mathfrak{M}^y,E/\Eo},\rho_{i_\circ,x}\cdot\chi^{z}_y)\text{ and }\mathcal{H}(U_{\mathfrak{M}^z,E/\Eo},\rho_{i_\circ,x}).$$
We now take $\tilde{\rho}_{i_\circ}=\tilde{\xi}_{i_\circ,\tame}\tilde{\chi}^{z}_{y,i_\circ}$ and ${\rho}_{x}=\xi_{x,\tame}{\chi}^{z}_{y,x}=\xi_{x,\tame}$, then $\mathcal{H}(G_W,\kappa^z\cdot \rho_{i_\circ,x})$  is covering
\begin{equation*}
\begin{split}
  &\mathcal{H}(M,(\tilde{\kappa}_{\tilde{\xi}_{i_\circ}}\boxtimes\kappa_{x,\tilde{\xi}})\cdot (\tilde{\xi}_{i_\circ,\tame}\tilde{\chi}^{z}_{y,i_\circ}\tilde{\nu}_{i_\circ,x,\tilde{\xi}}^{z,P}\boxtimes\xi_{x,\tame}))
  \\
=&\mathcal{H}(M,(\tilde{\kappa}_{\tilde{\xi}_{i_\circ}}\boxtimes\kappa_{x,\tilde{\xi}})\cdot (\tilde{\xi}_{i_\circ,\tame}\tilde{\nu}_{i_\circ,x,\tilde{\xi}}^{y,P}\boxtimes\xi_{x,\tame})),
\end{split}
  \end{equation*}
and is generated by
\begin{equation*}
\mathcal{H}(U_{\mathfrak{M}^y,E/\Eo},\tilde{\rho}_{i_\circ}\cdot\tilde{\chi}^{z}_y\boxtimes\rho_x)
  =\mathcal{H}(U_{\mathfrak{M}^y,E/\Eo},\tilde{\xi}_{i_\circ,\tame}\boxtimes\xi_{x,\tame})
  \end{equation*}
which gives the parameter $3/2$, and by
\begin{equation*}
 \mathcal{H}(U_{\mathfrak{M}^z,E/\Eo},\tilde{\rho}_{i_\circ}\boxtimes\rho_x)
  =\mathcal{H}(U_{\mathfrak{M}^z,E/\Eo},\tilde{\xi}_{i_\circ,\tame}\tilde{\chi}^{z}_{y,i_\circ}\boxtimes\xi_{x,\tame})
\end{equation*}
which gives the parameter $1/2$ since $\tilde{\chi}^{z}_{y,i_\circ}$ is a skew character.
\qed

\subsection{The product of amending characters}

 Using the Iwahori decomposition on the compact groups, the product (\ref{product-nu-w-nu-P}) is expanded into a product
      \begin{equation}\label{4-factors}
      \begin{split}
      &\sgn_{{\boldsymbol{\mu}}_{E_{i_\circ}}}(J^1_{\Lambda}\cap M:J^1_{\mathfrak{M}^w}\cap M)
      \\
      &\sgn_{{\boldsymbol{\mu}}_{E_{i_\circ}}}(J^1_{\Lambda}\cap U:J^1_{\mathfrak{M}^w}\cap U)
      \\
      &     \sgn_{{\boldsymbol{\mu}}_{E_{i_\circ}}}(J^1_{\mathfrak{M}^w}\cap U^-/H^1_{\mathfrak{M}^w}\cap U^-)
      \\
    &\sgn_{{\boldsymbol{\mu}}_{E_{i_\circ}}}(H^1_{\Lambda}\cap U^-:H^1_{\mathfrak{M}^w}\cap U^-)
      \end{split}
    \end{equation}
of four characters.

  \begin{lem}\label{compare-roots on M}
    The character $\sgn_{{\boldsymbol{\mu}}_{E_{i_\circ}}}(J^1_{\Lambda}\cap M:J^1_{\mathfrak{M}^w}\cap M)$ (the first factor in (\ref{4-factors})) is trivial.
  \end{lem}
  \proof
  The difference of $J^1_{\Lambda}\cap M$ and $J^1_{\mathfrak{M}^w}\cap M$  is on the root spaces in the blocks ${U}_{(
    i , i_\circ
  )}$ and ${U}_{(
    i_\circ , i
  )}$ with $i\neq i_\circ$. The torus ${\boldsymbol{\mu}}_{E_{i_\circ}}$ acts trivially on these root spaces because it lies outside the reductive group which contains these root spaces.
  \qed

    \begin{lem}
        The product character
        $$\sgn_{{\boldsymbol{\mu}}_{E_{i_\circ}}}(J^1_{\Lambda}\cap U:J^1_{\mathfrak{M}^w}\cap U)
      \sgn_{{\boldsymbol{\mu}}_{E_{i_\circ}}}(H^1_{\Lambda}\cap U^-:H^1_{\mathfrak{M}^w}\cap U^-)$$
      (the product of the second and the fourth factor in (\ref{4-factors})) is trivial.
      \end{lem}
\proof
This can be check from the explicit descriptions of the compact groups in Sub-section \ref{section appendix levels}. For example, when $d$ is odd, we can check that
$$J^1_{\Lambda}\cap U=J^1_{\mathfrak{M}^y}\cap U\text{ and }H^1_{\Lambda}\cap U^-=H^1_{\mathfrak{M}^y}\cap U^-,$$
and
$$J^1_{\mathfrak{M}^z}\cap U/J^1_{\Lambda}\cap U\cong H^1_{\mathfrak{M}^z}\cap U^-/H^1_{\Lambda}\cap U^-.$$
When $d$ is even, the argument is similar.
\qed

We write
$$\mathfrak{V}_{{\mathfrak{M}^w},U^-}:=J^1_{\mathfrak{M}^w}\cap U^-/H^1_{\mathfrak{M}^w}\cap U^-.$$
      Therefore, the amending character is
            \begin{equation*}
       \tilde{\nu}_{x,\tilde{\xi}}^{w,P}= \boxtimes_{i\in I}\tilde{\nu}^{w,P}_{i,x,\tilde{\xi}}\text{, where } \tilde{\nu}^{w,P}_{i,x,\tilde{\xi}}=\sgn_{{\boldsymbol{\mu}}_{E_i}}(\mathfrak{V}_{{\mathfrak{M}^w},U^-}).
      \end{equation*}

\subsection{Final computation}

We now know that $\tilde{\nu}^{w,P}_{i_\circ,x,\tilde{\xi}}$, as a tamely ramified character of $\mathfrak{o}_{E_{i_\circ}}^\times$, is at most quadratic. Since $\varpi\in F^\times\subseteq {E_{i_\circ}}^\times$ clearly acts on $\mathfrak{V}_{{\mathfrak{M}^w},U^-}$ trivially, we can extend our character to the whole ${E_{i_\circ}}^\times$ by requiring
$$\tilde{\nu}^{w,P}_{i_\circ,x,\tilde{\xi}}(\varpi)=1,$$
 so that it is a $+$-skew character.

 In this subsection, we determine whether the restriction $\tilde{\nu}^{w,P}_{i_\circ,x,\tilde{\xi}}|_{{\boldsymbol{\mu}}_{E_{i_\circ}}}$ is trivial or quadratic.

\begin{prop}
\begin{enumerate}[(i)]
  \item If $i_\circ\in I_\mathrm{o}$, then
    \begin{equation*}
     \tilde{\nu}^{y,P}_{i_\circ,x,\tilde{\xi}}|_{{\boldsymbol{\mu}}_{E_{i_\circ}}}=
\sgn_{{\boldsymbol{\mu}}_{E_{i_\circ}}}\left((U_{(
    I_\mathrm{e} , -
  )}\oplus U_{(
    + , I_\mathrm{e}
  )})^\sigma\right)^d
  \sgn_{{\boldsymbol{\mu}}_{E_{i_\circ}}}\left((U_{(
    I_\mathrm{o} , -
  )}\oplus U_{(
    + , I_\mathrm{o}
  )})^\sigma\oplus U_{(
    + , -
  )}^\sigma\right)^{d+1}
  \end{equation*}
  and
  \begin{equation*}
     \tilde{\nu}^{z,P}_{i_\circ,x,\tilde{\xi}}|_{{\boldsymbol{\mu}}_{E_{i_\circ}}}=
  \sgn_{{\boldsymbol{\mu}}_{E_{i_\circ}}}\left((U_{(
    I_\mathrm{o} , -
  )}\oplus U_{(
    + , I_\mathrm{o}
  )})^\sigma\right)^d
  \sgn_{{\boldsymbol{\mu}}_{E_{i_\circ}}}\left((U_{(
    I_\mathrm{e} , -
  )}\oplus U_{(
    + , I_\mathrm{e}
  )})^\sigma\oplus U_{(
    + , -
  )}^\sigma\right)^{d+1}
         \end{equation*}
    If $i_\circ\in I_\mathrm{e}$, then switch $\tilde{\nu}^{w}_{i_\circ,x,\tilde{\xi}}$ and $\tilde{\nu}^{z}_{i_\circ,x,\tilde{\xi}}$ (or equivalently, switch $I_\mathrm{o}$ and $I_\mathrm{e}$) in the two formulae above.
    \item $\tilde{\nu}^{y,P}_{i_\circ,x,\tilde{\xi}}$ and $\tilde{\nu}^{z,P}_{i_\circ,x,\tilde{\xi}}$ are independent of the two parabolic subgroups $P$ and $P^-$ that contain the Levi subgroup $M\cong \tilde{G}_{V_i}\times G_V$.
\end{enumerate}
\end{prop}
\proof
These are observed from the descriptions of the compact subgroups, which are given in the next section.
\qed

We drop the notation $P$ from $\tilde{\nu}^{w,P}_{i_\circ,x,\tilde{\xi}}$ and just write $\tilde{\nu}^{w}_{i_\circ,x,\tilde{\xi}}$, for $w=y$ or $z$.

\begin{cor}\label{character-quad-or-not}
Suppose $i_\circ\in I_\mathrm{o}$.
\begin{enumerate}[(i)]
  \item $\tilde{\nu}^{y}_{i_\circ,x,\tilde{\xi}}$ is quadratic if and only if $d$ is even and $\#I$ (or equivalently $\#I_\mathrm{o}$) is even.
  \item $\tilde{\nu}^{z}_{i_\circ,x,\tilde{\xi}}$ is quadratic if and only if $d$ is even, or $d$ is odd and $\#I$ is odd.
\end{enumerate}
If $i_\circ\in I_\mathrm{e}$, then switch $\tilde{\nu}^{w}_{i_\circ,x,\tilde{\xi}}$ and $\tilde{\nu}^{z}_{i_\circ,x,\tilde{\xi}}$ in the two statements above.
\end{cor}


\proof
Note that ${\boldsymbol{\mu}}_{E_{i_\circ}}\cong \mathbf{k}_{E_{i_\circ}}^\times$ acts on
$$\left({U}_{(
    +
    ,
    i
  )}\oplus{U}_{(
    i
    ,
    -
  )}\right)^\sigma\cong \left({U}_{(
    -
    ,
    i
  )}\oplus{U}_{(
    i
    ,
    +
  )}\right)^\sigma
  \cong \mathbf{k}_{E_{i_\circ}}\otimes_{\mathbf{k}_{F}}\mathbf{k}_{E_i}
  $$
by multiplication on the $ \mathbf{k}_{E_{i_\circ}}$-factor. It is then enough to show that the signature of this action is quadratic. Again the arguments are similar to the proof of Corollary \ref{cor-nu-P}.
\qed





We provide a few properties of amending characters under the equivalence relation by the Weyl group action (see Proposition \ref{equiv-relation-U}).

\begin{cor} For all $\gamma\in N_{G_V}(T_x)$, we have
$\tilde{\nu}^y_{x,\tilde{\xi}}=\tilde{\nu}^y_{x,{}^\gamma\tilde{\xi}}={}^{\gamma}\tilde{\nu}^y_{x,\tilde{\xi}}$.
\end{cor}
\proof
Indeed, $N_{G_V}(T_x)/T_x$ is isomorphic to $\left(\prod_{i\in I}\Gamma_{E_i/F}\right)\rtimes \mathfrak{S}_x$, where $\mathfrak{S}_x$ is a subgroup of the permutation group $\mathfrak{S}_I$ of the index set $I$ such that $\sigma\in \mathfrak{S}_x$ if and only if $n_{\sigma(i)}=n_i$ and $\sigma I_\mathrm{o}=I_\mathrm{o}$ (and hence also $\sigma I_\mathrm{e}=I_\mathrm{e}$). In other words, the Weyl group action does not change the partition $I=I_\mathrm{o}\sqcup I_\mathrm{e}$ corresponding to $x$, and the corollary follows.
\qed

Finally, we provide the values of the character $\tilde{\chi}_{y,i_\circ}^{z}$ in Proposition \ref{transferring-character-maximal}.

\begin{cor}
 For all $i_\circ\in I$, the character $\tilde{\chi}_{y,i_\circ}^{z}$ is quadratic if and only if $\#I$ is odd. In particular, either all $\tilde{\chi}_{y,i_\circ}^{z}$, $i_\circ\in I$, are trivial, or all are quadratic.
\end{cor}
\proof
$\tilde{\chi}_{y,i_\circ}^{z}$ can be directly calculated using Corollary \ref{character-quad-or-not} and Proposition \ref{transition-with-amending}. The second statement is clear.
\qed

Note that $\tilde{\chi}_{y,i_\circ}^{z}$ is independent of the level of $\tilde{\xi}$.

\begin{prop}\label{translate-bijection}
  For $T$ ranges over all unramified elliptic maximal tori if $G_V$, $x$ ranges over $\mathcal{D}$, and $\tilde{\xi}$ ranges over all regular +-skew characters, the map
  $${\pi}_{x,\tilde{\xi}}\mapsto {\pi}_{x,\tilde{\xi}\tilde{\nu}^y_{x,\tilde{\xi}}}$$
  is a bijection on the set of isomorphism classes of very cuspidal representations.
\end{prop}
\proof
For fixed $T$, $x$, and a restriction $\tilde{\theta}=\tilde{\xi}_x|_{\tilde{T}_x^1}$, then $\tilde{\nu}^y_{x,\tilde{\xi}}$ is a fixed tamely ramified character $\tilde{\nu}=\tilde{\nu}(\tilde{\theta})$ for all regular +-skew character $\tilde{\xi}$ whose restriction to ${\tilde{T}_x^1}$ is  $\tilde{\theta}$. The map between finite sets
$$\{{\pi}_{x,\tilde{\xi}}\}_{\text{fixed }T,\,x,\,\tilde{\theta}}\mapsto \{{\pi}_{x,\tilde{\xi}\tilde{\nu}^y_{x,\tilde{\xi}}}\}_{\text{fixed }T,\,x,\,\tilde{\theta}}$$
is clearly a bijection, and so is the map in the proposition.
\qed

\subsection{Appendix: computing the levels}\label{section appendix levels}

We provide the matrix descriptions of the groups $H_L^1$ and $J_L^1$, for $L$ being the lattice sequence $\Lambda=\Lambda_{i_\circ,x}$, $\mathfrak{M}^y=\mathfrak{M}^y_{i_\circ,x}$, or $\mathfrak{M}^z=\mathfrak{M}^z_{i_\circ,x}$.

We only show the $U$- and $U^-$-parts of the matrices. Recall that $x\in \mathcal{D}$ corresponds to a partition $I=I_\mathrm{o}\sqcup I_\mathrm{e}$. Suppose that we order the basis as $\mathcal{B}_{-}\sqcup\mathcal{B}_{I_\mathrm{o}}\sqcup\mathcal{B}_{I_\mathrm{e}}\sqcup\mathcal{B}_{+}$, such that $L$ has a decomposition $L_{-}\oplus{}L_{I_\mathrm{o}}\oplus{}L_{I_\mathrm{e}}\oplus{}L_{+}$. For $J,K$ being ${-},\,{I_\mathrm{o}},\,{I_\mathrm{e}}$ or ${+}$, the entries $\mathfrak{P}^{k}$ in the $(J,K)$-entry represents
$$\bigcap_{m\in \mathbb{Z}}Hom_{\mathfrak{o}_F}(L_K(m),L_J(m+k)).$$
The bullet symbols $\bullet$ are the entries in $M$, which are not important to the calculation.

For $i_\circ\in I_\mathrm{o}$, we separate into two cases.
\begin{itemize}
  \item When $d=2k+1$,
\begin{equation*}
\begin{split}
  &H^1_\Lambda=T_\Lambda^1U_\Lambda^{6k+4}
=\left[\begin{smallmatrix}
    \bullet&\mathfrak{P}^{k+1}&\mathfrak{P}^{k+1}&\mathfrak{P}^{k+1}
    \\
    \mathfrak{P}^{k+1}&\bullet&\bullet&\mathfrak{P}^{k+1}
    \\
    \mathfrak{P}^{k+1}&\bullet&\bullet&\mathfrak{P}^{k}
    \\
    \mathfrak{P}^{k+1}&\mathfrak{P}^{k+1}&\mathfrak{P}^{k+2}&\bullet
  \end{smallmatrix}\right]
  ,\,
  J^1_\Lambda=T_\Lambda^1U_\Lambda^{6k+3}
=\left[\begin{smallmatrix}
    \bullet&\mathfrak{P}^{k+1}&\mathfrak{P}^{k+1}&\mathfrak{P}^{k+1}
    \\
    \mathfrak{P}^{k+1}&\bullet&\bullet&\mathfrak{P}^{k+1}
    \\
    \mathfrak{P}^{k+1}&\bullet&\bullet&\mathfrak{P}^{k}
    \\
    \mathfrak{P}^{k+1}&\mathfrak{P}^{k+1}&\mathfrak{P}^{k+2}&\bullet
  \end{smallmatrix}\right],\,
\\
  &H^1_{\mathfrak{M}^y}=T_{\mathfrak{M}^y}^1U_{\mathfrak{M}^y}^{2k+2}=
\left[\begin{smallmatrix}
    \bullet&\mathfrak{P}^{k+1}&\mathfrak{P}^{k+2}&\mathfrak{P}^{k+1}
    \\
    \mathfrak{P}^{k+1}&\bullet&\bullet&\mathfrak{P}^{k+1}
    \\
    \mathfrak{P}^{k+1}&\bullet&\bullet&\mathfrak{P}^{k+1}
    \\
    \mathfrak{P}^{k+1}&\mathfrak{P}^{k+1}&\mathfrak{P}^{k+2}&\bullet
  \end{smallmatrix}\right]
  ,\,
  J^1_{\mathfrak{M}^y}=T_{\mathfrak{M}^y}^1U_{\mathfrak{M}^y}^{2k+1}=
\left[\begin{smallmatrix}
    \bullet&\mathfrak{P}^{k+1}&\mathfrak{P}^{k+1}&\mathfrak{P}^{k+1}
    \\
    \mathfrak{P}^{k+1}&\bullet&\bullet&\mathfrak{P}^{k+1}
    \\
    \mathfrak{P}^{k}&\bullet&\bullet&\mathfrak{P}^{k}
    \\
    \mathfrak{P}^{k+1}&\mathfrak{P}^{k+1}&\mathfrak{P}^{k+1}&\bullet
  \end{smallmatrix}\right]
\\
  &H^1_{\mathfrak{M}^z}=T_{\mathfrak{M}^z}^1U_{\mathfrak{M}^z}^{2k+2}=
\left[\begin{smallmatrix}
    \bullet&\mathfrak{P}^{k+1}&\mathfrak{P}^{k+1}&\mathfrak{P}^{k}
    \\
    \mathfrak{P}^{k+2}&\bullet&\bullet&\mathfrak{P}^{k+1}
    \\
    \mathfrak{P}^{k+1}&\bullet&\bullet&\mathfrak{P}^{k}
    \\
    \mathfrak{P}^{k+2}&\mathfrak{P}^{k+2}&\mathfrak{P}^{k+2}&\bullet
  \end{smallmatrix}\right]
  ,\,
  J^1_{\mathfrak{M}^z}=T_{\mathfrak{M}^z}^1U_{\mathfrak{M}^z}^{2k+1}=
\left[\begin{smallmatrix}
    \bullet&\mathfrak{P}^{k}&\mathfrak{P}^{k+1}&\mathfrak{P}^{k}
    \\
    \mathfrak{P}^{k+1}&\bullet&\bullet&\mathfrak{P}^{k}
    \\
    \mathfrak{P}^{k+1}&\bullet&\bullet&\mathfrak{P}^{k}
    \\
    \mathfrak{P}^{k+2}&\mathfrak{P}^{k+1}&\mathfrak{P}^{k+2}&\bullet
  \end{smallmatrix}\right].
\end{split}
 \end{equation*}

\item When $d=2k$,
\begin{equation*}
\begin{split}
  &H^1_\Lambda=T_\Lambda^1U_\Lambda^{6k+1}
=\left[\begin{smallmatrix}
    \bullet&\mathfrak{P}^{k}&\mathfrak{P}^{k+1}&\mathfrak{P}^{k}
    \\
    \mathfrak{P}^{k+1}&\bullet&\bullet&\mathfrak{P}^{k}
    \\
    \mathfrak{P}^{k}&\bullet&\bullet&\mathfrak{P}^{k}
    \\
    \mathfrak{P}^{k+1}&\mathfrak{P}^{k+1}&\mathfrak{P}^{k+1}&\bullet
  \end{smallmatrix}\right]
  ,\,
  J^1_\Lambda=T_\Lambda^1U_\Lambda^{6k}
=\left[\begin{smallmatrix}
    \bullet&\mathfrak{P}^{k}&\mathfrak{P}^{k+1}&\mathfrak{P}^{k}
    \\
    \mathfrak{P}^{k+1}&\bullet&\bullet&\mathfrak{P}^{k}
    \\
    \mathfrak{P}^{k}&\bullet&\bullet&\mathfrak{P}^{k}
    \\
    \mathfrak{P}^{k+1}&\mathfrak{P}^{k+1}&\mathfrak{P}^{k+1}&\bullet
  \end{smallmatrix}\right],\,
\\
  &H^1_{\mathfrak{M}^y}=T_{\mathfrak{M}^y}^1U_{\mathfrak{M}^y}^{2k+1}
=\left[\begin{smallmatrix}
    \bullet&\mathfrak{P}^{k+1}&\mathfrak{P}^{k+1}&\mathfrak{P}^{k+1}
    \\
    \mathfrak{P}^{k+1}&\bullet&\bullet&\mathfrak{P}^{k+1}
    \\
    \mathfrak{P}^{k}&\bullet&\bullet&\mathfrak{P}^{k}
    \\
    \mathfrak{P}^{k+1}&\mathfrak{P}^{k+1}&\mathfrak{P}^{k+1}&\bullet
  \end{smallmatrix}\right]
  ,\,
  J^1_{\mathfrak{M}^y}=T_{\mathfrak{M}^y}^1U_{\mathfrak{M}^y}^{2k}
=\left[\begin{smallmatrix}
    \bullet&\mathfrak{P}^{k}&\mathfrak{P}^{k+1}&\mathfrak{P}^{k}
    \\
    \mathfrak{P}^{k}&\bullet&\bullet&\mathfrak{P}^{k}
    \\
    \mathfrak{P}^{k}&\bullet&\bullet&\mathfrak{P}^{k}
    \\
    \mathfrak{P}^{k}&\mathfrak{P}^{k}&\mathfrak{P}^{k+1}&\bullet
  \end{smallmatrix}\right]
\\
  &H^1_{\mathfrak{M}^z}=T_{\mathfrak{M}^z}^1U_{\mathfrak{M}^z}^{2k+1}
=\left[\begin{smallmatrix}
    \bullet&\mathfrak{P}^{k}&\mathfrak{P}^{k+1}&\mathfrak{P}^{k}
    \\
    \mathfrak{P}^{k+1}&\bullet&\bullet&\mathfrak{P}^{k}
    \\
    \mathfrak{P}^{k+1}&\bullet&\bullet&\mathfrak{P}^{k}
    \\
    \mathfrak{P}^{k+2}&\mathfrak{P}^{k+1}&\mathfrak{P}^{k+2}&\bullet
  \end{smallmatrix}\right]
  ,\,
  J^1_{\mathfrak{M}^z}=T_{\mathfrak{M}^z}^1U_{\mathfrak{M}^z}^{2k}
=\left[\begin{smallmatrix}
    \bullet&\mathfrak{P}^{k}&\mathfrak{P}^{k}&\mathfrak{P}^{k-1}
    \\
    \mathfrak{P}^{k+1}&\bullet&\bullet&\mathfrak{P}^{k}
    \\
    \mathfrak{P}^{k}&\bullet&\bullet&\mathfrak{P}^{k-1}
    \\
    \mathfrak{P}^{k+1}&\mathfrak{P}^{k+1}&\mathfrak{P}^{k+1}&\bullet
  \end{smallmatrix}\right].
\end{split}
  \end{equation*}

\end{itemize}
If $i_\circ\in I_\mathrm{e}$, then switch $I_\mathrm{o}$ and $I_\mathrm{e}$ in the ordered basis $\mathcal{B}_{-}\sqcup\mathcal{B}_{I_\mathrm{o}}\sqcup\mathcal{B}_{I_\mathrm{e}}\sqcup\mathcal{B}_{+}$.

\newpage\section{Endoscopic Classification}\label{section endos classification}


 Given a cuspidal representation $\pi_{x,\tilde{\xi}}$ of $G_V$ constructed by a regular +-skew character $\tilde{\xi}$, by Theorem \ref{theorem-inertial-class} there is a conjugate-self-dual supercuspidal representation in the inertial class of $\tilde{\pi}_{\tilde{\xi}_i\tilde{\nu}_{i,x,\tilde{\xi}}}$ and in the extended cuspidal support of $\pi_{x,\tilde{\xi}}$. In this inertial class, there are two conjugate-self-dual representations, different from each other by the quadratic unramified character $\tilde{\chi}_-\circ\mathrm{det}_{\tilde{G}_{V_{i}}}$. In this section, we use the endoscopic classification to determine the one that belongs to the extended cuspidal support of $\pi_{x,\tilde{\xi}}$.

\subsection{L-groups and parameters}

We follow \cite[Section 2.1]{Mok-unitary} and describe the L-groups of $\mathbf{G}_V=\mathrm{U}_{n,F/\Fo}$ and $\tilde{\mathbf{G}}_V=\Res_{F/\Fo}\mathrm{GL}_{n,F}$. Let $\mathcal{W}_\Fo$ be the Weil group of $\Fo$. We set
\begin{equation*}
\begin{split}
&{}^L\mathbf{G}_V=\mathrm{GL}_{n}(\mathbb{C})\rtimes \mathcal{W}_{\Fo},
    \\
    &{}^L\tilde{\mathbf{G}}_V=
    (\mathrm{GL}_{n}(\mathbb{C})\times \mathrm{GL}_{n}(\mathbb{C}))\rtimes \mathcal{W}_{\Fo},
  \end{split}
\end{equation*}
where for the L-group ${}^L\mathbf{G}_V$, the action of $\mathcal{W}_{\Fo} $ on $\hat{\mathbf{G}}_V=\mathrm{GL}_{n}(\mathbb{C})$ factors through $\Gamma_{F/\Fo}$ and each $w\in \mathcal{W}_{\Fo} -\mathcal{W}_{F} $ acts as
\begin{equation*}
  g\mapsto J{}(^tg^{-1})J^{-1},\,\qquad \text{for }g\in \mathrm{GL}_{n}(\mathbb{C}),
\end{equation*}
where $J$ as in (\ref{J-matrix}), and for the L-group ${}^L\tilde{\mathbf{G}}_V$, the action of $\mathcal{W}_{\Fo}$ of $\hat{\tilde{\mathbf{G}}}_V=\mathrm{GL}_{n}(\mathbb{C})\times \mathrm{GL}_{n}(\mathbb{C})$ also factors through $\Gamma_{F/\Fo}$ and each $w\in \mathcal{W}_{\Fo} -\mathcal{W}_{F} $ acts as interchanging the two factors as
\begin{equation*}
  (g,h)\mapsto (h,g),\,\qquad\text{for }(g,h)\in \mathrm{GL}_{n}(\mathbb{C})\times \mathrm{GL}_{n}(\mathbb{C}).
\end{equation*}
We fix a choice $c\in \mathcal{W}_{\Fo}-\mathcal{W}_F$ for convenience, but usually the subsequent statements do not depend on this choice. For $\epsilon=+$ or $-$, we fix two skew characters $\tilde{\chi}_\epsilon$ as in Section \ref{section characters}, so that $\tilde{\chi}_+$ is trivial and $\tilde{\chi}_-$ is unramified quadratic. By local class field theory, we identify $\tilde{\chi}_\epsilon$ with a character of $\mathcal{W}_{F}$, also denoted by $\tilde{\chi}_\epsilon$, which satisfies
\begin{equation*}
  \tilde{\chi}_\epsilon(cwc^{-1})=\tilde{\chi}_\epsilon(w)^{-1}\text{ for all }w\in \mathcal{W}_{F},\text{ and }\tilde{\chi}_\epsilon(c^2)=\epsilon1.
  \end{equation*}
We then define two embeddings
$$\iota_\epsilon^{\mathbf{G}_V}:{}^L \mathbf{G}_V\rightarrow {}^L\tilde{\mathbf{G}}_V$$
as follows:
\begin{equation}\label{stable-embedding-rules}
  \begin{split}
    g\rtimes 1 &\mapsto (g,{}^tg^{-1})\rtimes 1,\,\text{ for all }g\in \mathrm{GL}_n(\mathbb{C}),
    \\
    I\rtimes w &\mapsto (\tilde{\chi}_\epsilon(w)I,\tilde{\chi}_\epsilon^{-1}(w)I)\rtimes w,\,\text{ for all }w\in \mathcal{W}_F\text{, and }
    \\
    I\rtimes c &\mapsto (\epsilon J,J^{-1})\rtimes c.
  \end{split}
\end{equation}
One can show that the $\hat{\tilde{\mathbf{G}}}_V$-conjugacy class of $\iota_\epsilon^{\mathbf{G}_V}$ is independent of the choice of $c\in \mathcal{W}_{\Fo}-\mathcal{W}_F$.

Let $\tilde{\phi}:\mathcal{W}_F\rightarrow \mathrm{GL}_n(\mathbb{C})$ be a parameter for $\tilde{G}_V$. It is sometimes convenient to view $\tilde{\phi}$ as a semi-simple representation of $\mathcal{W}_F$ on a $\mathbb{C}$-vector space $\mathbf{V}$ of dimension $n$. We call $\tilde{\phi}$ conjugate-self-dual if $\tilde{\phi}\cong {}^\sigma\tilde{\phi}:= {}^c\tilde{\phi}^\vee$ where, for all $w\in \mathcal{W}_F$,
${}^c\tilde{\phi}(w)=\tilde{\phi}(c^{-1}wc)$ and $\tilde{\phi}^\vee(w)={}^t\tilde{\phi}(w)^{-1}$, i.e., $\tilde{\phi}^\vee$ is the contragredient of $\tilde{\phi}$. Hence $\tilde{\phi}$ is conjugate-self-dual if there exists a non-degenerate bilinear form $B$ on $\mathbf{V}$ such that,
\begin{equation} \label{B-is-G-invar}
  B(\tilde{\phi}^c(w)u,\tilde{\phi}(w)v)=B(u,v)\text{ for all }u,v\in \mathbf{V}\text{ and }w\in \mathcal{W}_F.
\end{equation}
We call $\tilde{\phi}$ conjugate-orthogonal (resp. conjugate-symplectic) if, furthermore,
\begin{equation}\label{B-is-Herm}
  B(u,v)= B(v,\tilde{\phi}(c^2)u)\qquad\text{(resp. }-B(v,\tilde{\phi}(c^2)u)\text{)}
\end{equation}
for all $u,v\in \mathbf{V}$. Again this condition is independent of the choice of $c\in \mathcal{W}_{\Fo}-\mathcal{W}_F$.

A conjugate-self-dual $\tilde{\phi}$ is called skew if it is irreducible as a representation of $\mathcal{W}_F$, and is called $+$-skew (resp. $-$-skew) if it is furthermore conjugate-orthogonal (resp. conjugate-symplectic).

We can extend a parameter $\tilde{\phi}:\mathcal{W}_F\rightarrow \mathrm{GL}(\mathbf{V})$ to a morphism
\begin{equation}\label{extend-parameter}
  \begin{split}
    &\tilde{\phi}\times{}^\sigma\tilde{\phi}:\mathcal{W}_{\Fo}\rightarrow {}^L\tilde{\mathbf{G}}_V,
    \\
    \text{by }&w\mapsto (\tilde{\phi}(w),{}^\sigma\tilde{\phi}(w))\rtimes w\text{ and }c\mapsto (I,I)\rtimes c.
  \end{split}
\end{equation}
By \cite[Lemma 2.2.1]{Mok-unitary}, if $\tilde{\phi}$ is a skew parameter, then it is $\epsilon$-skew if and only if the image of $\tilde{\phi}\times{}^\sigma\tilde{\phi}$ lies in $\iota_\epsilon^{\mathbf{G}_V}({}^L \mathbf{G}_V)$ (up to conjugacy).

Suppose that $\Eo/\Fo$ is an unramified extension of odd degree $n$ and $E/\Eo$ is quadratic unramified, then we can assume that $c\in \mathcal{W}_{\Eo}$ and $c^2\in \mathcal{W}_E$. Given a regular character $\tilde{\xi}$ of $E^\times$ over $F$, we regard $\tilde{\xi}$ as a character of $\mathcal{W}_E$ by class field theory, then we know that $\tilde{\phi}=\Ind_{E/F}\tilde{\xi}$ is irreducible as a representation of $\mathcal{W}_F$.
\begin{prop}
  If $\tilde{\xi}$ is a $\epsilon$-skew character, then $\tilde{\phi}=\Ind_{E/F}\tilde{\xi}$ is a $\epsilon$-skew parameter.
\end{prop}
\proof
Note that $\tilde{\phi}$ is skew since
$${}^c(\Ind_{E/F}\tilde{\xi})\cong \Ind_{E/F}({}^c\tilde{\xi})\cong \Ind_{E/F}(\tilde{\xi}^{-1})\cong (\Ind_{E/F}\tilde{\xi})^\vee$$
We can choose a non-degenerate bilinear form $B$ such that conditions (\ref{B-is-G-invar}) and (\ref{B-is-Herm}) are satisfied, for a fixed sign $\epsilon$. The restriction $\Res_{E/F}\tilde{\phi}$ contains $\tilde{\xi}$ (with multiplicity 1).
Hence there is an eigenvector $v$ in $\mathbf{V}\cong \mathbb{C}^n$ such that
 $$\tilde{\phi}(c^2)v=\tilde{\xi}(c^2)v=\epsilon v.$$
Then
$$B(v,v)=B(v,\epsilon\tilde{\phi}(c^2)v)=\epsilon B(v,\tilde{\phi}(c^2)v),$$
which implies that $\tilde{\phi}$ is $\epsilon$-skew.
\qed

\subsection{Asai L-functions}

We can interpret the skewness in terms of Asai L-functions on the Galois side and Asai-Shahidi L-functions the representation side.

On the Galois side, we define, for $\epsilon=+$ or $-$,
$$\Asai_\epsilon:{}^L\tilde{\mathbf{G}}_V=(\mathrm{GL}(\mathbf{V})\times \mathrm{GL}(\mathbf{V}))\rtimes \mathcal{W}_{\Fo}\rightarrow \mathrm{GL}(\mathbf{V}\otimes \mathbf{V}),$$
where $\mathrm{GL}(\mathbf{V})\times \mathrm{GL}(\mathbf{V})$ acts on $\mathbf{V}\otimes \mathbf{V}$ by the standard representation on tensor product, $ \mathcal{W}_{F}$ acts trivially on $\mathbf{V}\otimes \mathbf{V}$, and
$$\Asai_\epsilon(c)(v\otimes w)=\epsilon (w\otimes v).$$
If $n\in \mathbb{Z}$, then we write $\Asai_n$ for $\Asai_+$ if $n$ is even, and for $\Asai_-$ if $n$ is odd.

Let $|\cdot|:\mathcal{W}_{\Fo}\rightarrow \mathbb{Z}$ be the valuation of $\mathcal{W}_{\Fo}$. We define the Asai L-functions by
$$L(\tilde{\phi},\Asai_\epsilon,s)=\det\left(I_{\mathbf{V}\otimes \mathbf{V}}-\Asai_\epsilon\circ(\tilde{\phi}\times{}^\sigma\tilde{\phi})|\cdot|^s\right)^{-1}$$
for $\epsilon=+$ or $-$.

\begin{prop}
  $L(\tilde{\phi}\otimes {}^\sigma\tilde{\phi},s)=L(\tilde{\phi},\Asai_+,s)L(\tilde{\phi},\Asai_-,s)$.
\end{prop}
\proof
Indeed, there is an isomorphism of representations
$$\Ind_{F/\Fo}(\tilde{\phi}\otimes {}^\sigma\tilde{\phi})\cong \Asai_+\circ(\tilde{\phi}\times{}^\sigma\tilde{\phi})\oplus \Asai_-\circ(\tilde{\phi}\times{}^\sigma\tilde{\phi}).$$
We then obtain the result by taking the L-functions on both sides.
\qed

\begin{prop}\label{L-function-pole}
  A parameter $\tilde{\phi}$ is conjugate-orthogonal (resp. conjugate-symplectic) if and only if $L(\tilde{\phi},\Asai_+,s)$ (resp. $L(\tilde{\phi},\Asai_-,s)$ ) has a pole at $s=0$.
\end{prop}
\proof
Note that $\tilde{\phi}$ is conjugate-orthogonal exactly when $\Asai_+\circ(\tilde{\phi}\times{}^\sigma\tilde{\phi})$ has an invariant vector, which is equivalent to saying that $\det\left(I_{\mathbf{V}\otimes \mathbf{V}}-\Asai_+\circ(\tilde{\phi}\times{}^\sigma\tilde{\phi})\right)=0$. Similarly for $\tilde{\phi}$ being conjugate-symplectic and taking $\Asai_-$.
\qed

It is clear that
\begin{equation}\label{Asai-twist}
  \begin{split}
    &L(\tilde{\phi},\Asai_\epsilon,s)\text{ has a pole at }s=0
    \\
    \text{if and only if }&L(\tilde{\phi}\tilde{\chi}_-,\Asai_{-\epsilon},s)\text{ has a pole at }s=0.
  \end{split}
\end{equation}
If $\tilde{\phi}$ is irreducible and skew, then $L(\tilde{\phi}\otimes {}^\sigma\tilde{\phi},s)$ has a simple pole at $s=0$, and so either $\Asai_+\circ(\tilde{\phi}\times{}^\sigma\tilde{\phi})$ or $\Asai_-\circ(\tilde{\phi}\times{}^\sigma\tilde{\phi})$ has an invariant vector but not both. Combining with (\ref{Asai-twist}), this provides a dichotomy on the singularity of the L-functions:
\begin{prop}\label{dicho-on-sing}
  If $\tilde{\phi}$ is irreducible and skew and $\epsilon$ is fixed, then either $L(\tilde{\phi},\Asai_\epsilon,s)$ or $L(\tilde{\phi}\tilde{\chi}_-,\Asai_\epsilon,s)$ has a (simple) pole at $s=0$, but not both.
\end{prop}
\proof
If $\tilde{\phi}$ is skew, then at least one of $L(\tilde{\phi},\Asai_\epsilon,s)$ and $L(\tilde{\phi},\Asai_{-\epsilon},s)$ has a pole at $s=0$. The latter is equivalent to that $L(\tilde{\phi}\tilde{\chi}_-,\Asai_{\epsilon},s)$ has a pole at $s=0$. If both of them have such a pole, then this contradicts the dichotomy.
\qed

On the representation side, we have the Asai-Shahidi L-function $L(\tilde{\pi},\mathrm{AS}_\epsilon,s)$ introduced from the Langlands-Shahidi method \cite{shahidi-complementary} (see also \cite{Shahidi-Cogdell-2014}). The relation between Asai-Shahidi L-functions and Asai L-functions is provided by the following Proposition.

\begin{prop}\label{one-of-the-two}
  If $\tilde{\pi}_{\tilde{\phi}}$ is an irreducible representation of a general linear group with Langlands parameter ${\tilde{\phi}}$, then
$$L(\tilde{\pi}_{\tilde{\phi}},\mathrm{AS}_\epsilon,s)=L({\tilde{\phi}},\Asai_\epsilon,s).$$
\end{prop}
\proof
 This is proved in \cite{Hen-ext-sym}.
\qed

We can now answer the question posed at the beginning of this section. Let $\pi_{x,\tilde{\xi}}$ be as in the beginning of this section, and $\tilde{\pi}_{i_\circ}$ is one of the two conjugate-self-dual supercuspidal representations lying in the inertial class of $\tilde{\pi}_{\tilde{\xi}_{i_\circ}}$.
\begin{prop}\label{Asai-implies-reducible}
  $\tilde{\pi}_{i_\circ}|\det|^{s}\rtimes \pi_{x,\tilde{\xi}}$ is reducible at $s=1$ if and only if
$L(\tilde{\pi}_{i_\circ},\mathrm{AS}_{n-1},s)$ has a pole at $s=0$.
\end{prop}
\proof
By \cite[5.6 Proposition]{Moeg-base-change} or \cite[A.2.1]{Moeg-exhaustion}.
\qed


Let $\tilde{\xi}=\boxtimes_{i\in I}\tilde{\xi}_i$ be a regular +-skew character. If $\tilde{\nu}^{y}_{i,x,\tilde{\xi}}$ is the tamely ramified character defined in Theorem \ref{theorem-inertial-class}, then Proposition \ref{Asai-implies-reducible} implies directly the following corollary.

\begin{cor}\label{type-ECS}
   For each $i\in I$, $\tilde{\pi}_{\tilde{\xi}_i\tilde{\chi}^{E_i}_{(-1)^{n-1}}\tilde{\nu}^y_{i,x,\tilde{\xi}}}$ lies in the extended cuspidal support of ${\pi}_{x,\tilde{\xi}}$.
\end{cor}


\subsection{Stable Packets}

We apply the parametrization of stable packets in \cite[Section 5.7]{Moeg-base-change}. Suppose that $\tilde{\xi}=\boxtimes_{i\in I}\tilde{\xi}_i$ is a $d$-regular character as in Section \ref{section regular char}, except that we now require that all $\tilde{\xi}_i$ are $(-1)^{n+1}$-skew. We define a representation of $\mathcal{W}_F$ by,
\begin{equation}\label{parameter-induced}
  \tilde{\phi}=\bigoplus_{i\in I}\tilde{\phi}_i\text{, where }\tilde{\phi}_i\cong \Ind_{E_i/F}\tilde{\xi}_i.
\end{equation}
This is a parameter considered in \cite{reeder-pos-depth}, and is called a very cuspidal parameter in this paper. The following properties hold.
\begin{itemize}
  \item The $d$-regularity of each $\tilde{\xi}_i$ implies that each $\tilde{\phi}_i$ is irreducible.
  \item Since $\tilde{\phi}_i$ is $(-1)^{n+1}$-skew, $L(\tilde{\phi}_i,\mathrm{Asai}_{n+1},s)$ has a pole at $s=0$, by Proposition \ref{L-function-pole}.
      \item The regularity of the product $\tilde{\xi}$ implies that $\tilde{\phi}$ is a $\sigma$-discrete parameter, where by definition $\tilde{\phi}$ has no multiplicity and all components are $\sigma$-invariant (c.f. \cite[Section 5.7]{Moeg-base-change}).
\end{itemize}

The parameter $\tilde{\phi}$ is \emph{stable}, which means that if we extend  $\tilde{\phi}$ to a morphism $\tilde{\phi}\times{}^\sigma\tilde{\phi}:\mathcal{W}_{\Fo}\rightarrow {}^L\tilde{\mathbf{G}}_V$ as in (\ref{extend-parameter}), then up to $\hat{\tilde{{\mathbf{G}}}}_V$-conjugacy it factors through the embedding $\iota_{(-1)^{n-1}}^{\mathbf{G}_V}:{}^L\mathbf{G}_V\rightarrow {}^L\tilde{\mathbf{G}}_V$, i.e.,
$$\iota_{(-1)^{n-1}}^{\mathbf{G}_V}\circ \phi=\tilde{\phi}\times{}^\sigma\tilde{\phi}.$$
for some parameter $\phi:\mathcal{W}_{\Fo}\rightarrow {}^L\mathbf{G}_V:=\hat{\mathbf{G}}_V\rtimes \mathcal{W}_{\Fo}$ (see \cite[Lemma 2.2.1]{Mok-unitary} and \cite[Proposition 5.2.3]{Moeg-endosc-L-param}).

Let $\tilde{\pi}_{\tilde{\phi}_i}$ be the supercuspidal representation of $\tilde{\mathbf{G}}_{V_i}(F)\cong \mathrm{GL}_{n_i}(F)$ corresponding to $\tilde{\phi}_i$ via the local Langlands correspondence for $\mathrm{GL}_{n_i} $ \cite{HT}, \cite{Hen-simple}, \cite{Scholze-LLC}; and denote by
$$\tilde{\pi}_{\tilde{\phi}}=\prod_{i\in I}\tilde{\pi}_{\tilde{\phi}_i}$$
the parabolic induction from the representation $\boxtimes_{i\in I}\tilde{\pi}_{\tilde{\phi}_i}$ of the Levi subgroup $\prod_{i\in I}\tilde{\mathbf{G}}_{V_i}(F)\cong \prod_{i\in I}\mathrm{GL}_{n_i}(F)$ of $\tilde{\mathbf{G}}_{V}(F)\cong \mathrm{GL}_{n}(F)$. This is the representation parametrized by ${\tilde{\phi}}$ under the local Langlands correspondence, and is tempered and $\sigma$-discrete, i.e., it is $\sigma$-invariant and parabolically induced from mutually inequivalent discrete series (see \cite[Section 2.2]{Moeg-endosc-L-param}). By \cite{Hen-unram},
$$\tilde{\pi}_{\tilde{\phi}_i}=\tilde{\pi}_{\tilde{\xi}_i}$$
 since each $E_i/F$ is unramified of odd degree.

 Let $\Pi_{\tilde{\phi}}$ be the packet of representations of $G_V$, each of whose extended cuspidal support coincide with the cuspidal support of $\tilde{\pi}_{\tilde{\phi}}$.

We temporarily write $\tilde{\xi}_+=\boxtimes_{i\in I}\tilde{\xi}_{i,+}$, where each component $\tilde{\xi}_{i,+}=\tilde{\xi}_i\tilde{\chi}^{E_i}_{(-1)^{n-1}}$ is $+$-skew. If $\tilde{\nu}^{y}_{i,x,\tilde{\xi}_+}$ is the tamely ramified character defined in Theorem \ref{theorem-inertial-class}, then Corollary \ref{type-ECS} implies that $\tilde{\pi}_{\tilde{\xi}_{i,+}\tilde{\nu}^{y}_{i,x,\tilde{\xi}_+}}$ lies in the extended cuspidal support of ${\pi}_{x,\tilde{\xi}}$. Write $\tilde{\nu}_{x,\tilde{\xi}_+}^{y}=\boxtimes_{i\in I}\tilde{\nu}^{y}_{i,x,\tilde{\xi}_+}$.

\begin{lem}\label{translation property}
  If $\tilde{\chi}=\boxtimes_{i\in I}\tilde{\chi}_i$ is a +-skew tamely ramified character of $\tilde{T}$, then
 $\tilde{\pi}_{\tilde{\xi}_{i,+}\tilde{\chi}_i\tilde{\nu}^{y}_{i,x,\tilde{\xi}_+}}$ lies in the extended cuspidal support of ${\pi}_{x,\tilde{\xi}\tilde{\chi}}$.
\end{lem}
\proof
It is because $\tilde{\nu}^{y}_{i,x,\tilde{\xi}_+\tilde{\chi}}$ are all equal as $\tilde{\chi}$ ranges over +-skew tamely ramified characters. The result is immediate from Corollary \ref{type-ECS}.
\qed

\begin{prop}\label{tame-twist}
  For each $i\in I$, $\tilde{\pi}_{\tilde{\xi}_i}$ lies in the extended cuspidal support of ${\pi}_{x,\tilde{\xi}_+\tilde{\nu}^y_{x,\tilde{\xi}_+}}$.
\end{prop}
\proof
Just take $\tilde{\chi}=\tilde{\nu}^y_{x,\tilde{\xi}_+}$ in Lemma \ref{translation property}.
\qed

\begin{thm}\label{theorem-stable-packet}
  Let $\tilde{\phi}$ be a very cuspidal parameter for $\mathbf{G}_V=\mathrm{U}_{n,F/\Fo}$ such that each $\tilde{\phi}_i$ is $(-1)^{n+1}$-skew, and let $\tilde{\xi}$ be the associated character. The set
  $$\{\pi_{x,\tilde{\xi}_+\tilde{\nu}_{x,\tilde{\xi}_+}^{y}}\}_{x\in\mathcal{D}}$$
  is a stable packet with Langlands parameter $\phi$. In other words, its base-change (relative to $\iota_{(-1)^{n-1}}^{\mathbf{G}_V}$) is the representation $\tilde{\pi}_{\tilde{\phi}}$.
\end{thm}
\proof

By Propositions \ref{Asai-implies-reducible} and \ref{tame-twist}, the set $\{\tilde{\pi}_{\tilde{\xi}_i}\}_{i\in I}$ is a subset of the extended cuspidal support of $\pi_{x,\tilde{\xi}_+\tilde{\nu}_{x,\tilde{\xi}_+}^{y}}$. The two sets are indeed equal by the estimate \cite[4. Proposition]{Moeg-base-change}, which says that if $\tilde{\pi}_{\tilde{\xi}_i}$ is a representation of GL$_{n_i}$, then $$\sum n_i\leq n,$$ where the sum ranges over the representations in the extended cuspidal support.
\\\\
The size of the set $\{\pi_{x,\tilde{\xi}_+\tilde{\nu}_{x,\tilde{\xi}_+}^{y}}\}_{x\in\mathcal{D}}$ is $\mathcal{D}$ by Propositions \ref{equiv-relation-U} and \ref{translate-bijection}, and the representations inside have the same extended cuspidal support. By the discussion in \cite[Section 5.7]{Moeg-base-change}, they belong to the stable packet $\Pi_{\tilde{\phi}}$. Another estimate \cite[7.1 Th{\'e}or{\`e}me]{Moeg-base-change} implies that
$$\#\Pi_{\tilde{\phi}}=2^{\#I-1}=\#\mathcal{D}.$$
Therefore, the set exhausts the whole packet $\Pi_{\tilde{\phi}}$.

\qed

Finally, we prove the uniqueness of our amending characters.

\begin{prop}
  For each $x\in\mathcal{D}$ and regular +-skew character $\tilde{\xi}$, there is a unique tamely ramified character $\tilde{\nu}_{x,\tilde{\xi}}$ such that the base change of $\pi_{x,\tilde{\xi}\tilde{\nu}_{x,\tilde{\xi}}}$ is $\tilde{\pi}_{\tilde{\phi}}$.
\end{prop}
\proof
If the level $d$ of $\tilde{\xi}$ is 0, then $\tilde{\nu}_{x,\tilde{\xi}}$ is the trivial character. The proposition is an easy consequence of Lemma \ref{translation property} and Proposition \ref{equiv-relation-U}. We now assume that $d>0$. Suppose that the base changes of $\pi_{x,\tilde{\xi}\tilde{\nu}_{x,\tilde{\xi}}}$ and $\pi_{x,\tilde{\xi}\tilde{\nu}'_{x,\tilde{\xi}}}$ are both $\tilde{\pi}_{\tilde{\xi}}$ for two tamely ramified characters $\tilde{\nu}_{x,\tilde{\xi}}=\tilde{\nu}'_{x,\tilde{\xi}}$. By Proposition \ref{tame-twist}, the base change of $\pi_{x,\tilde{\xi}}$ is $$\prod_{i\in I}\tilde{\pi}_{\tilde{\xi}_i\tilde{\nu}_{i,x,\tilde{\xi}}^{-1}}=\prod_{i\in I}\tilde{\pi}_{\tilde{\xi}_i(\tilde{\nu}_{i,x,\tilde{\xi}}')^{-1}}.$$
By Zelevinzki's classification of representations of GL${}_n$ \cite{Zelevinsky-GLn-II}, there is a permutation $\sigma$ of the set $I$ such that for all $i\in I$,
\begin{equation}\label{sigma(xi-nu)}
  \tilde{\xi}_{\sigma(i)}\tilde{\nu}_{\sigma(i),x,\tilde{\xi}}^{-1}={}^{\gamma_i}({\tilde{\xi}_i(\tilde{\nu}_{i,x,\tilde{\xi}}')^{-1}})={}^{\gamma_i}{\tilde{\xi}_i(\tilde{\nu}_{i,x,\tilde{\xi}}')^{-1}}
\end{equation}
 for some $\gamma_i\in \Gamma_{E_i/F}$ . Now the restriction
$$\tilde{\xi}_{\sigma(i)}({}^{\gamma_i}{\tilde{\xi}_i}^{-1})|_{U^1_{E_i}}=\tilde{\nu}_{\sigma(i),x,\tilde{\xi}}(\tilde{\nu}_{i,x,\tilde{\xi}}')^{-1}|_{U^1_{E_i}}\equiv 1.$$
The right side is trivial since both ${\nu}_{\sigma(i),x,\tilde{\xi}}$ and $\tilde{\nu}_{i,x,\tilde{\xi}}'$, but the left side is non-trivial by the regularity of $\tilde{\xi}$ unless $\sigma=1$ and all $\gamma_i=1$. Therefore, (\ref{sigma(xi-nu)}) implies that $\tilde{\nu}_{i,x,\tilde{\xi}}=\tilde{\nu}_{i,x,\tilde{\xi}}'$.
\qed

\subsection{Base change in general}\label{section base change H}

We follow \cite[Section 4.7]{Row-u3} (or \cite[Section 2.4]{Mok-unitary}) to describe the elliptic twisted endoscopic data for $\tilde{\mathbf{G}}_V$. Let $n=n_1+n_2$, where $n_1,n_2\in\mathbb{Z}_{\geq 0}$, be a partition of $n$ and, for $j=1,2$, let $V_j$ be a Hermitian space of $F$-dimension $n_j$ such that $\mathbf{G}_{V_j}$ is a quasi-split unitary group over $\Fo$. The elliptic twisted endoscopic data for $\tilde{\mathbf{G}}_V$ are the pairs
$$(\mathbf{H},\iota^\mathbf{H})=(\mathbf{G}_{V_1}\times \mathbf{G}_{V_2},\iota^\mathbf{H}_{\epsilon_1,\epsilon_2}),$$
where $\epsilon_1, \epsilon_2\in\{\pm1\}$ such that
\begin{equation*}
  (\epsilon_1, \epsilon_2)=\begin{cases}
  (+,-)\text{ or }(-,+)&\text{ if }n_1\equiv n_2\mod 2,
  \\
  (+,+)\text{ or }(-,-)&\text{ if }n_1\not\equiv n_2\mod 2,
\end{cases}
\end{equation*}
 and $\iota^\mathbf{H}_{\epsilon_1,\epsilon_2}$ is the composition of the embeddings
$${}^L\mathbf{H}\xrightarrow{\iota^{\mathbf{G}_{V_1}}_{\epsilon_1}\times\iota^{\mathbf{G}_{V_2}}_{\epsilon_2}}{}^L(\tilde{\mathbf{G}}_{V_1}\times \tilde{\mathbf{G}}_{V_2})\xrightarrow{\begin{smallmatrix}
  \text{diagonal}
  \\
  \text{embedding}
\end{smallmatrix}} {}^L\tilde{\mathbf{G}}_V,$$
where $\iota^{\mathbf{G}_{V_j}}_{\epsilon_j}$ is defined in (\ref{stable-embedding-rules}).

We refer the notion of the equivalence of endoscopic data to \cite[Section 2.1]{KS}, and only remark that the data above are inequivalent to each other except when $n_1=n_2$, in which case
$$(\mathbf{G}_{V_1}\times \mathbf{G}_{V_2},\iota^\mathbf{H}_{+,-})\text{ and }(\mathbf{G}_{V_1}\times \mathbf{G}_{V_2},\iota^\mathbf{H}_{-,+})$$
are equivalent.

The purpose of defining elliptic twisted endoscopic data is to stabilize the $\sigma$-discrete spectrum, in the following sense.
\begin{prop}[\text{\cite[6.1 Theorem]{Moeg-base-change}, \cite[Theorems 6.4.1, 6.4.2]{Moeg-endosc-L-param}}]
  If $\pi$ is a $\sigma$-discrete representation of $\tilde{G}_V$, then there exists a unique elliptic endoscopic datum $(\mathbf{H},\iota^\mathbf{H}_{\epsilon_1,\epsilon_2})$ of  $\tilde{G}_V$ such that $\pi$ is $(\mathbf{H},\iota^\mathbf{H}_{\epsilon_1,\epsilon_2})$-stable, which means that the $\sigma$-twisted trace of $\pi$ is the endoscopic transfer \cite{LS}, \cite{KS} of a sum of characters in a stable packet of $\mathbf{H}$.
\end{prop}


We now use the dichotomy (\ref{dicho-on-sing}) to describe the parameter $\phi^H$ of the $(\mathbf{H},\iota^\mathbf{H}_{\epsilon_1,\epsilon_2})$-stable packet $\Pi_{\phi^H}$ of representations of $H=\mathbf{H}(\Fo)$. Suppose that  $\tilde{\xi}=\boxtimes_{i\in I}\tilde{\xi}_i$ is a $d$-regular character, where each component is skew. Define
$\tilde{\phi}=\bigoplus_{i\in I}\tilde{\phi}_i$ as in (\ref{parameter-induced}). We now partition $I=I_1\sqcup I_2$ such that, for $j=1,2$, we have $i\in I_j$ if and only if $\tilde{\xi}_i\text{ is }\epsilon_j$-skew, or equivalently, the Asai L-functions $L(\tilde{\phi}_i,\Asai_{\epsilon_j},s)$ has a pole at $s=0$. Let $n_j=\sum_{i\in I_j}n_i$, and define $\mathbf{H}=\mathbf{G}_{V_1}\times \mathbf{G}_{V_2}$, then $\tilde{\phi}$ factors through the endoscopic embedding $\iota_{\epsilon_1,\epsilon_2}^\mathbf{H}$, such that
$$\tilde{\phi}:\mathcal{W}_F\xrightarrow {\phi^H}{}^L\mathbf{H}\xrightarrow {\iota_{\epsilon_1,\epsilon_2}^\mathbf{H}}{}^L{\tilde{\mathbf{G}}}_V$$
The parameter ${\phi^H}$ corresponds to a stable packet $\Pi_{\phi^H}$ of $H=\mathbf{H}(\Fo)$.

We temporarily write $\tilde{\xi}_+=\tilde{\xi}_{1,+}\boxtimes \tilde{\xi}_{2,+}$, where $\tilde{\xi}_{j,+}=\boxtimes_{i\in I_j}\tilde{\xi}_{i,+}$, and each component $\tilde{\xi}_{i,+}=\tilde{\xi}_i\tilde{\chi}^{E_i}_{\epsilon_j}$ is $+$-skew. Write $\tilde{\nu}_{x_j,\tilde{\xi}_{j,+}}^{y}=\boxtimes_{i\in I_j}\tilde{\nu}^{y}_{i,x_j,\tilde{\xi}_{j,+}}$, for $j=1,2$.
\begin{cor}\label{theorem-H-stable-packet}
The set
  $$\{\pi_{{x_1},\tilde{\xi}_{1,+}\tilde{\nu}_{x_1,\tilde{\xi}_{1,+}}^{y}}\boxtimes \pi_{{x_2},\tilde{\xi}_{2,+}\tilde{\nu}_{x_2,\tilde{\xi}_{2,+}}^{y}}\}_{(x_1,x_2)\in\mathcal{D}_1\times \mathcal{D}_2}$$
  is a stable packet of $H={G}_{V_1}\times {G}_{V_2}$, whose base-change (relative to the elliptic endoscopic datum $(\mathbf{H},\iota^\mathbf{H}_{\epsilon_1,\epsilon_2})$) is the representation $$\tilde{\pi}_{\tilde{\phi}_1}\times \tilde{\pi}_{\tilde{\phi}_2},$$
  understood as the parabolic induction from the representation $\tilde{\pi}_{\tilde{\phi}_1}\boxtimes\tilde{\pi}_{\tilde{\phi}_2}$ of the Levi subgroup $\tilde{{G}}_{V_1}\times \tilde{{G}}_{V_2}$ of $\tilde{{G}}_{V}$.
\end{cor}

This is completely analogous to Theorem \ref{theorem-stable-packet}.

\newpage

\bibliographystyle{alpha}
\bibliography{unitary}

\def\cprime{$'$}
\begin{thebibliography}{CKPSS04}

\bibitem[ACS]{Shahidi-Cogdell-2014}
Mahdi Asgari, James~W. Cogdell, and Freydoon Shahidi.
\newblock Local transfer and reducibility of induced representations of
  $p$-adic groups of classical type, arxiv:1412.6191.
\newblock {\em To appear in Contemporary Mathematics (AMS), volume in honor of
  Jim Cogdell's 60th birthday}.

\bibitem[Adl98]{adler-cusp}
Jeffrey~D. Adler.
\newblock Refined anisotropic {$K$}-types and supercuspidal representations.
\newblock {\em Pacific J. Math.}, 185(1):1--32, 1998.

\bibitem[AL05]{Adler-Lansky-unram}
Jeffrey~D. Adler and Joshua~M. Lansky.
\newblock Depth-zero base change for unramified {${\rm U}(2,1)$}.
\newblock {\em J. Number Theory}, 114(2):324--360, 2005.

\bibitem[AL10]{Adler-Lansky-ram}
Jeffrey~D. Adler and Joshua~M. Lansky.
\newblock Depth-zero base change for ramified {${\rm U}(2,1)$}.
\newblock {\em Trans. Amer. Math. Soc.}, 362(10):5569--5599, 2010.

\bibitem[Art13]{Arthur-new-book}
James Arthur.
\newblock {\em The endoscopic classification of representations}, volume~61 of
  {\em American Mathematical Society Colloquium Publications}.
\newblock American Mathematical Society, Providence, RI, 2013.
\newblock Orthogonal and symplectic groups.

\bibitem[BB02]{Bl-Bl-SP4}
Laure Blasco and Corinne Blondel.
\newblock Alg\`ebres de {H}ecke et s\'eries principales g\'en\'eralis\'ees de
  {${\rm Sp}_4(F)$}.
\newblock {\em Proc. London Math. Soc. (3)}, 85(3):659--685, 2002.

\bibitem[BH05]{BH-ET1}
Colin~J. Bushnell and Guy Henniart.
\newblock The essentially tame local {L}anglands correspondence. {I}.
\newblock {\em J. Amer. Math. Soc.}, 18(3):685--710 (electronic), 2005.

\bibitem[BK93]{BK}
C.J. Bushnell and P.C. Kutzko.
\newblock {\em {The admissible dual of GL(N) via compact open subgroups}}.
\newblock Annals of mathematics studies. Princeton University Press, 1993.

\bibitem[BK94]{BK-SLn-2}
Colin~J. Bushnell and Philip~C. Kutzko.
\newblock The admissible dual of {${\rm SL}(N)$}. {II}.
\newblock {\em Proc. London Math. Soc. (3)}, 68(2):317--379, 1994.

\bibitem[BK98]{BK-cover}
Colin~J. Bushnell and Philip~C. Kutzko.
\newblock Smooth representations of reductive {$p$}-adic groups: structure
  theory via types.
\newblock {\em Proc. London Math. Soc. (3)}, 77(3):582--634, 1998.

\bibitem[BK99]{BK-ss-type}
Colin~J. Bushnell and Philip~C. Kutzko.
\newblock Semisimple types in {${\rm GL}_n$}.
\newblock {\em Compositio Math.}, 119(1):53--97, 1999.

\bibitem[Bla08]{Blasco-u3}
Laure Blasco.
\newblock Types, paquets et changement de base: l'exemple de {${\rm
  U}(2,1)(F\sb 0)$}. {I}. {T}ypes simples maximaux et paquets singletons.
\newblock {\em Canad. J. Math.}, 60(4):790--821, 2008.

\bibitem[Bla10]{Blasco-u2}
Laure Blasco.
\newblock Changements de base explicites des repr\'esentations supercuspidales
  de {${\rm U}(1,1)(F\sb 0)$}.
\newblock {\em Ann. Inst. Fourier (Grenoble)}, 60(3):905--938, 2010.

\bibitem[Blo12]{Blondel-Weil}
Corinne Blondel.
\newblock Repr\'esentation de {W}eil et {$\beta$}-extensions.
\newblock {\em Ann. Inst. Fourier (Grenoble)}, 62(4):1319--1366, 2012.

\bibitem[Car84]{Carayol-cusp}
H.~Carayol.
\newblock Repr\'esentations cuspidales du groupe lin\'eaire.
\newblock {\em Ann. Sci. \'Ecole Norm. Sup. (4)}, 17(2):191--225, 1984.

\bibitem[Car85]{Carter-book}
Roger~W. Carter.
\newblock {\em Finite groups of {L}ie type}.
\newblock Pure and Applied Mathematics (New York). John Wiley \& Sons, Inc.,
  New York, 1985.
\newblock Conjugacy classes and complex characters, A Wiley-Interscience
  Publication.

\bibitem[CKPSS04]{Cogdell-Kim-PS-Shahidi}
J.~W. Cogdell, H.~H. Kim, I.~I. Piatetski-Shapiro, and F.~Shahidi.
\newblock Functoriality for the classical groups.
\newblock {\em Publ. Math. Inst. Hautes \'Etudes Sci.}, (99):163--233, 2004.

\bibitem[Die71]{Dieudonne}
Jean~A. Dieudonn{\'e}.
\newblock {\em La g\'eom\'etrie des groupes classiques}.
\newblock Springer-Verlag, Berlin-New York, 1971.
\newblock Troisi{\`e}me {\'e}dition, Ergebnisse der Mathematik und ihrer
  Grenzgebiete, Band 5.

\bibitem[DR09]{DR}
Stephen DeBacker and Mark Reeder.
\newblock Depth-zero supercuspidal {$L$}-packets and their stability.
\newblock {\em Ann. of Math. (2)}, 169(3):795--901, 2009.

\bibitem[G{\'e}r79]{Gerardin-unram}
Paul G{\'e}rardin.
\newblock Cuspidal unramified series for central simple algebras over local
  fields.
\newblock In {\em Automorphic forms, representations and {$L$}-functions
  ({P}roc. {S}ympos. {P}ure {M}ath., {O}regon {S}tate {U}niv., {C}orvallis,
  {O}re., 1977), {P}art 1}, Proc. Sympos. Pure Math., XXXIII, pages 157--169.
  Amer. Math. Soc., Providence, R.I., 1979.

\bibitem[GKS07]{GKS}
David Goldberg, Philip Kutzko, and Shaun Stevens.
\newblock Covers for self-dual supercuspidal representations of the {S}iegel
  {L}evi subgroup of classical {$p$}-adic groups.
\newblock {\em Int. Math. Res. Not. IMRN}, (22):Art. ID rnm085, 31, 2007.

\bibitem[Gol93]{Goldberg-Siegel-case-U22}
David Goldberg.
\newblock Reducibility of generalized principal series representations of
  {${\rm U}(2,2)$} via base change.
\newblock {\em Compositio Math.}, 86(3):245--264, 1993.

\bibitem[Gol94]{Goldberg-Siegel-case-unitary}
David Goldberg.
\newblock Some results on reducibility for unitary groups and local {A}sai
  {$L$}-functions.
\newblock {\em J. Reine Angew. Math.}, 448:65--95, 1994.

\bibitem[GR10]{Gross-Reeder}
Benedict~H. Gross and Mark Reeder.
\newblock Arithmetic invariants of discrete {L}anglands parameters.
\newblock {\em Duke Math. J.}, 154(3):431--508, 2010.

\bibitem[GT11]{Gan-Takeda-GSP4}
Wee~Teck Gan and Shuichiro Takeda.
\newblock The local {L}anglands conjecture for {${\rm GSp}(4)$}.
\newblock {\em Ann. of Math. (2)}, 173(3):1841--1882, 2011.

\bibitem[Hei11]{Heiermann-Hecke-Bernstein}
Volker Heiermann.
\newblock Op\'erateurs d'entrelacement et alg\`ebres de {H}ecke avec
  param\`etres d'un groupe r\'eductif {$p$}-adique: le cas des groupes
  classiques.
\newblock {\em Selecta Math. (N.S.)}, 17(3):713--756, 2011.

\bibitem[Hen92]{Hen-unram}
Guy Henniart.
\newblock Correspondance de {L}anglands-{K}azhdan explicite dans le cas non
  ramifi\'e.
\newblock {\em Math. Nachr.}, 158:7--26, 1992.

\bibitem[Hen00]{Hen-simple}
Guy Henniart.
\newblock Une preuve simple des conjectures de {L}anglands pour {${\rm GL}(n)$}
  sur un corps {$p$}-adique.
\newblock {\em Invent. Math.}, 139(2):439--455, 2000.

\bibitem[Hen10]{Hen-ext-sym}
Guy Henniart.
\newblock Correspondance de {L}anglands et fonctions {$L$} des carr\'es
  ext\'erieur et sym\'etrique.
\newblock {\em Int. Math. Res. Not. IMRN}, (4):633--673, 2010.

\bibitem[HM08]{HM}
Jeffrey Hakim and Fiona Murnaghan.
\newblock Distinguished tame supercuspidal representations.
\newblock {\em Int. Math. Res. Pap. IMRP}, (2):Art. ID rpn005, 166, 2008.

\bibitem[How77]{Howe}
Roger~E. Howe.
\newblock Tamely ramified supercuspidal representations of {${\rm Gl}\sb{n}$}.
\newblock {\em Pacific J. Math.}, 73(2):437--460, 1977.

\bibitem[HT01]{HT}
Michael Harris and Richard Taylor.
\newblock {\em The geometry and cohomology of some simple {S}himura varieties},
  volume 151 of {\em Annals of Mathematics Studies}.
\newblock Princeton University Press, Princeton, NJ, 2001.
\newblock With an appendix by Vladimir G. Berkovich.

\bibitem[JS03]{Jiang-Soudry-SO-odd}
Dihua Jiang and David Soudry.
\newblock The local converse theorem for {${\rm SO}(2n+1)$} and applications.
\newblock {\em Ann. of Math. (2)}, 157(3):743--806, 2003.

\bibitem[Kim07]{Kim-exhaust}
Ju-Lee Kim.
\newblock Supercuspidal representations: an exhaustion theorem.
\newblock {\em J. Amer. Math. Soc.}, 20(2):273--320 (electronic), 2007.

\bibitem[KM06]{kutzko-Morris}
Philip Kutzko and Lawrence Morris.
\newblock Level zero {H}ecke algebras and parabolic induction: the {S}iegel
  case for split classical groups.
\newblock {\em Int. Math. Res. Not.}, pages Art. ID 97957, 40, 2006.

\bibitem[KMSW]{unitary-inner}
Tasho Kaletha, Alberto Minguez, Sug~Woo Shin, and Paul-James White.
\newblock Endoscopic classification of representations: Inner forms of unitary
  groups, arxiv:1409.3731.

\bibitem[KS99]{KS}
Robert~E. Kottwitz and Diana Shelstad.
\newblock Foundations of twisted endoscopy.
\newblock {\em Ast\'erisque}, (255):vi+190, 1999.

\bibitem[LS87]{LS}
R.~P. Langlands and D.~Shelstad.
\newblock On the definition of transfer factors.
\newblock {\em Math. Ann.}, 278(1-4):219--271, 1987.

\bibitem[Lus84]{Lusztig-book}
George Lusztig.
\newblock {\em Characters of reductive groups over a finite field}, volume 107
  of {\em Annals of Mathematics Studies}.
\newblock Princeton University Press, Princeton, NJ, 1984.

\bibitem[Lus13]{Lust-GSP4}
Jaime Lust.
\newblock Depth-zero supercuspidal {$L$}-packets for inner forms of {$GSp_4$}.
\newblock {\em J. Algebra}, 389:23--60, 2013.

\bibitem[M{\oe}g02]{Moeg-exhaustion}
C.~M{\oe}glin.
\newblock Sur la classification des s\'eries discr\`etes des groupes classiques
  {$p$}-adiques: param\`etres de {L}anglands et exhaustivit\'e.
\newblock {\em J. Eur. Math. Soc. (JEMS)}, 4(2):143--200, 2002.

\bibitem[M{\oe}g07]{Moeg-base-change}
Colette M{\oe}glin.
\newblock Classification et changement de base pour les s\'eries discr\`etes
  des groupes unitaires {$p$}-adiques.
\newblock {\em Pacific J. Math.}, 233(1):159--204, 2007.

\bibitem[M{\oe}g14]{Moeg-endosc-L-param}
Colette M{\oe}glin.
\newblock Paquets stables des s\'eries discr\`etes accessibles par endoscopie
  tordue; leur param\`etre de {L}anglands.
\newblock In {\em Automorphic forms and related geometry: assessing the legacy
  of {I}. {I}. {P}iatetski-{S}hapiro}, volume 614 of {\em Contemp. Math.},
  pages 295--336. Amer. Math. Soc., Providence, RI, 2014.

\bibitem[Mok15]{Mok-unitary}
Chung~Pang Mok.
\newblock Endoscopic classification of representations of quasi-split unitary
  groups.
\newblock {\em Mem. Amer. Math. Soc.}, 235(1108):vi+248, 2015.

\bibitem[Mor93]{Morris-tame-ram-intertwining-alg}
Lawrence Morris.
\newblock Tamely ramified intertwining algebras.
\newblock {\em Invent. Math.}, 114(1):1--54, 1993.

\bibitem[MP94]{Moy-Prasad}
Allen Moy and Gopal Prasad.
\newblock Unrefined minimal {$K$}-types for {$p$}-adic groups.
\newblock {\em Invent. Math.}, 116(1-3):393--408, 1994.

\bibitem[MR98]{MR-classical}
F.~Murnaghan and J.~Repka.
\newblock Reducibility of some induced representations of split classical
  {$p$}-adic groups.
\newblock {\em Compositio Math.}, 114(3):263--313, 1998.

\bibitem[MR99]{MR-unitary}
Fiona Murnaghan and Joe Repka.
\newblock Reducibility of some induced representations of {$p$}-adic unitary
  groups.
\newblock {\em Trans. Amer. Math. Soc.}, 351(1):193--210, 1999.

\bibitem[MS14]{Stevens-Miya}
Michitaka Miyauchi and Shaun Stevens.
\newblock Semisimple types for {$p$}-adic classical groups.
\newblock {\em Math. Ann.}, 358(1-2):257--288, 2014.

\bibitem[Ree08]{reeder-pos-depth}
Mark Reeder.
\newblock Supercuspidal {$L$}-packets of positive depth and twisted {C}oxeter
  elements.
\newblock {\em J. Reine Angew. Math.}, 620:1--33, 2008.

\bibitem[Rog90]{Row-u3}
Jonathan~D. Rogawski.
\newblock {\em Automorphic representations of unitary groups in three
  variables}, volume 123 of {\em Annals of Mathematics Studies}.
\newblock Princeton University Press, Princeton, NJ, 1990.

\bibitem[RY14]{Reeder-Yu}
Mark Reeder and Jiu-Kang Yu.
\newblock Epipelagic representations and invariant theory.
\newblock {\em J. Amer. Math. Soc.}, 27(2):437--477, 2014.

\bibitem[Sav08]{Savin-level-0}
Gordan Savin.
\newblock Lifting of generic depth zero representations of classical groups.
\newblock {\em J. Algebra}, 319(8):3244--3258, 2008.

\bibitem[Sch13]{Scholze-LLC}
Peter Scholze.
\newblock The local {L}anglands correspondence for {$\mathrm{GL}_n$} over
  {$p$}-adic fields.
\newblock {\em Invent. Math.}, 192(3):663--715, 2013.

\bibitem[Sha90]{shahidi-complementary}
Freydoon Shahidi.
\newblock A proof of {L}anglands' conjecture on {P}lancherel measures;
  complementary series for {$p$}-adic groups.
\newblock {\em Ann. of Math. (2)}, 132(2):273--330, 1990.

\bibitem[Sha92]{Shahidi-twisted-endos}
Freydoon Shahidi.
\newblock Twisted endoscopy and reducibility of induced representations for
  {$p$}-adic groups.
\newblock {\em Duke Math. J.}, 66(1):1--41, 1992.

\bibitem[Sil80]{silberger-special}
Allan~J. Silberger.
\newblock Special representations of reductive {$p$}-adic groups are not
  integrable.
\newblock {\em Ann. of Math. (2)}, 111(3):571--587, 1980.

\bibitem[Ste05]{stevens-ss-char}
Shaun Stevens.
\newblock Semisimple characters for {$p$}-adic classical groups.
\newblock {\em Duke Math. J.}, 127(1):123--173, 2005.

\bibitem[Ste08]{stevens-supercusp}
Shaun Stevens.
\newblock The supercuspidal representations of {$p$}-adic classical groups.
\newblock {\em Invent. Math.}, 172(2):289--352, 2008.

\bibitem[Wal01]{walds-nil-orbit}
Jean-Loup Waldspurger.
\newblock Int\'egrales orbitales nilpotentes et endoscopie pour les groupes
  classiques non ramifi\'es.
\newblock {\em Ast\'erisque}, (269):vi+449, 2001.

\bibitem[Yu01]{Yu-cusp}
Jiu-Kang Yu.
\newblock Construction of tame supercuspidal representations.
\newblock {\em J. Amer. Math. Soc.}, 14(3):579--622 (electronic), 2001.

\bibitem[Zel80]{Zelevinsky-GLn-II}
A.~V. Zelevinsky.
\newblock Induced representations of reductive {${p}$}-adic groups. {II}. {O}n
  irreducible representations of {${\rm GL}(n)$}.
\newblock {\em Ann. Sci. \'Ecole Norm. Sup. (4)}, 13(2):165--210, 1980.

\end{thebibliography}

\end{document}